\documentclass[a4paper,british,10pt]{scrartcl}

\AddtoDoHook{heading/endgroup/paragraph}{\headingdot}
\newcommand*{\headingdot}[1]{.}

\usepackage[style=alphabetic,maxbibnames=99,maxalphanames=4,natbib=true]{biblatex}

\usepackage{caption}

\usepackage{authblk}

\usepackage[bookmarks=false,colorlinks=true, linkcolor=blue, urlcolor=blue, citecolor=blue, breaklinks=true]{hyperref}

\title{Stochastic Prediction Equilibrium for Dynamic Traffic Assignment}
\author{Lukas Graf}
\author{Tobias Harks}
\author{Michael Markl}
\affil{University of Passau}
\affil{\normalsize\href{mailto:lukas.graf@uni-passau.de;\%20tobias.harks@uni-passau.de;\%20michael.markl@uni-passau.de?subject=Comment\%20on\%20your\%20paper\%20\%22Stochastic\%20Prediction\%20Equilibrium\%20for\%20Dynamic\%20Traffic\%20Assignment\%22}{\{lukas.graf, tobias.harks, michael.markl\}@uni-passau.de}}

\usepackage[left=1.1in, right=1.1in, top=1.2in, bottom=1.5in]{geometry}

\usepackage{amsthm}

\date{}

\newcommand{\optDisplay}[1]{\[#1\]}

	\newenvironment{revisedEnv}{}{}
	\newcommand{\revised}[1]{#1}
	\newcommand{\revMinor}[1]{#1}
	\newcommand{\lukas}[2][]{}
	\newcommand{\lukasI}[2][]{}
	\newcommand{\marginnote}[1]{}
	\NewDocumentCommand{\michael}{O{}+m}{}
	\NewDocumentCommand{\tobias}{O{}+m}{}
	\NewDocumentCommand{\revisionComment}{O{}+m}{}

\providebool{noappendix}

\usepackage{mathtools,mathrsfs}
\usepackage{amssymb}

\usepackage[T1]{fontenc}

\usepackage[british]{babel}

\usepackage{etoolbox}
\usepackage{xstring}
\usepackage{longtable}

\usepackage{svg}

\usepackage{adjustbox}
\usepackage{changepage}   
\usepackage{xcolor}
\definecolor{darkgreen}{rgb}{0.2,0.8,0.55}

\usepackage{bm} 					
\usepackage{braket}					
\usepackage{nicefrac}

\usepackage{enumitem}  

\usepackage[caption=false]{subfig}

\usepackage{tikz}
\usetikzlibrary{calc} \usetikzlibrary{decorations.pathreplacing} 
\tikzstyle{vertex} = [shape=circle,draw=black]
\tikzstyle{edge} = [draw,->,thick]

\usepackage{pgfplots}
\usepgfplotslibrary{groupplots}
\usetikzlibrary{pgfplots.groupplots}
\usetikzlibrary{intersections}
\usepgfplotslibrary{fillbetween}

\usepackage{cleveref}

\AddToHook{cmd/appendix/before}{\crefalias{section}{appendix}}
	\theoremstyle{plain}
	\newtheorem{theorem}{Theorem}[section]
\newcommand{\AddThm}[2]{
	\newtheorem{#1}[theorem]{#2}
	\AddToHook{env/#1/begin}{\crefalias{theorem}{#1}}
}

	\theoremstyle{definition}
	\AddThm{definition}{Definition}

	\theoremstyle{plain}

	\AddThm{proposition}{Proposition}
	\AddThm{lemma}{Lemma}
	\AddThm{corollary}{Corollary}
\AddThm{conjecture}{Conjecture}
\AddThm{question}{Question}
\AddThm{claim}{Claim}

	\theoremstyle{remark}

\AddThm{remark}{Remark}
\AddThm{example}{Example}
\AddThm{observation}{Observation}
\AddThm{notation}{Notation}

\newenvironment{proofClaim}[1][]{\ifthenelse{\equal{#1}{}}{\begin{proof}}{\begin{proof}[#1]}}{\end{proof}}

\newenvironment{structuredproof}{\begin{description}}{\end{description}}
\newcommand{\proofitem}[1]{\item[{\textnormal{\textit{#1:}}}]}

\newcommand{\ie}{i.e.\ }

\newcommand{\eg}{e.g.\ }

\newcommand{\cf}{cf.\ }
\newcommand{\wrt}{w.r.t.\ }

\newcommand{\wlg}{wlog.\ }

\NewDocumentCommand{\stpath}{O{s}O{t}}{\ensuremath{#1,#2\text{-path}}}
\NewDocumentCommand{\stwalk}{O{s}O{t}}{\ensuremath{#1,#2\text{-walk}}}
\NewDocumentCommand{\stnetwork}{O{s}O{t}}{\ensuremath{#1,#2\text{-network}}}

\newcommand{\IR}{\mathbb{R}}
\newcommand{\IRnn}{\IR_{\geq 0}}
\newcommand{\IRp}{\IR_{>0}}

\newcommand{\IN}{\mathbb{N}}

\newcommand{\IRnnInf}{\overline{\IR}_{\geq0}}

\newcommand{\IP}{\mathbb{P}}
\newcommand{\Prob}{\IP}

\newcommand{\eps}{\varepsilon}
\newcommand*{\boldeps}{\boldsymbol{\eps}}

\newcommand{\emptyarg}{\,\boldsymbol{\cdot}\,}

\newcommand{\foraall}{\forall_{\text{a.a.}}}

\newcommand{\deriv}[2][]{\ifthenelse{\equal{#1}{}}{\partial #2}{\partial_{#1} #2}}

\newcommand{\diff}{\,\mathrm{d}}

\newcommand{\abs}[1]{\left|#1\right|}
\newcommand{\smallabs}[1]{|#1|}
\newcommand{\norm}[1]{\left\lVert#1\right\rVert}
\newcommand{\smallnorm}[1]{\lVert#1\rVert}
\newcommand{\normFct}{\smallnorm{\cdot}}

\NewDocumentCommand{\dist}{O{} m m}{\mathrm{d}_{#1}(#2, #3)}

\newcommand{\restr}[2]{\left.#1\right\rvert_{#2}}
\newcommand{\smallrestr}[2]{#1\rvert_{#2}}

\DeclareMathOperator{\argmin}{\mathrm{argmin}}

\NewDocumentCommand{\pLocInt}{O{p}O{\IR}}{L^{#1}_{\mathrm{loc}}(#2)}
\NewDocumentCommand{\pInt}{O{p}O{\IR}}{L^{#1}(#2)}

\newcommand{\leb}{\lambda}

\DeclareMathOperator*{\symDiff}{\Delta}

\usepackage{dsfont} 
\NewDocumentCommand{\CharF}{O{}}{\ifthenelse{\equal{#1}{}}
	{\mathds{1}}
	{\mathds{1}_{#1}}
}

\DeclareMathOperator{\graph}{\mathrm{graph}}

\newcommand{\edgesTo}[2][]{\delta\ifthenelse{\equal{#1}{}}{}{_{#1}}^-(#2)}
\newcommand{\edgesFrom}[2][]{\delta\ifthenelse{\equal{#1}{}}{}{_{#1}}^+(#2)}

\newcommand{\head}[1][]{\mathop{\mathrm{head}}\ifthenelse{\equal{#1}{}}{}{_{#1}}}
\newcommand{\tail}[1][]{\mathop{\mathrm{tail}}\ifthenelse{\equal{#1}{}}{}{_{#1}}}

\newcommand{\PathSet}{\mathcal{P}}

\newcommand{\dest}{t}

\newcommand{\Path}{p}
\newcommand{\PathAlt}{q}

\newcommand*{\horizon}{T}

\newcommand{\ROp}{\mathscr{R}}
\newcommand*{\ROe}{r}

\newcommand{\EdgeLoading}{\Phi}

\newcommand*{\bound}{B}

\newcommand*{\flowSet}{\rateFcts^{E\times I\times\{+,-\}}}

\NewDocumentCommand{\pCost}{}{\hat C}

\NewDocumentCommand{\pEdges}{}{\hat E}

\NewDocumentCommand{\pActive}{}{\hat\Theta} 

\NewDocumentCommand{\nodeInflow}{m m m}{#1^+_{#2,#3}}

\NewDocumentCommand{\rateFcts}{}{\mathcal{\revMinor{F}}}

\NewDocumentCommand{\splitFcts}{}{\mathcal{S}}

\newcommand*{\coherent}{coherent}

\newcommand*{\uniqueExtProp}{unique extension property}

\newcommand*{\extProp}{extension existence property}

\newcommand*{\Contr}{\Psi}

\newcommand*{\flowVol}{X}

\newcommand*{\qLen}{z}
\newcommand*{\capa}{\nu}
\newcommand*{\ffttime}{c^0}
\newcommand*{\ttime}{c}
\newcommand*{\exitTime}{\tau}

\newcommand*{\DPE}{dynamic prediction equilibrium}
\newcommand*{\DPEs}{dynamic prediction equilibria}

\newcommand*{\SPE}{stochastic prediction equilibrium}
\newcommand*{\SPEs}{stochastic prediction equilibria}

\newcommand*{\IDEs}{instantaneous dynamic equilibria}

\definecolor{flowColLabel}{HTML}{0C5388} \colorlet{flowCol}{flowColLabel!60!white}

\newcommand{\toverset}[2]{\overset{\text{#1}}{#2}}

\newcommand{\refsym}[1]{\ensuremath{(\ifthenelse{\equal{#1}{1}}{\ast}{\ifthenelse{\equal{#1}{2}}{\#}{\ifthenelse{\equal{#1}{3}}{\triangle}{\ifthenelse{\equal{#1}{4}}{\bigcirc}{\ifthenelse{\equal{#1}{5}}{\diamond}
						{\color{red}NaN}}}}})}}
\newcommand{\symoverset}[2]{\toverset{\refsym{#1}}{#2}}

\newcommand{\abstr}{
Stochastic effects significantly influence the dynamics of traffic flows.
Many dynamic traffic assignment (DTA) models attempt to capture these effects by prescribing a specific ratio that determines how flow splits across different routes based on the routes' costs.\par{}
\revised{Other models take a game-theoretic perspective and describe the equilibria resulting from the individual traffic participants' decisions instead of prescribing flow splits, however they usually neglect stochastic effects.}
In this paper,  we propose a new \revised{unifying} framework for DTA that incorporates the interplay between the routing decisions of each single traffic participant, the \revised{potentially} stochastic nature of predicting the future state of the network, and the physical flow dynamics.
Our framework consists of an \emph{edge loading operator} 
modeling the physical flow propagation
and a \emph{routing operator} modeling the routing behavior
of traffic participants.
The routing operator is assumed to be set-valued and, thus, capable 
to model complex (deterministic) equilibrium conditions as well as stochastic equilibrium conditions assuming that measurements for predicting traffic are noisy.
As our main results, we derive several quite general equilibrium existence and uniqueness results
which not only subsume known results from the literature but also lead to new results.
Specifically, for the new stochastic prediction equilibrium, we show existence and uniqueness under natural assumptions on the probability distribution over the predictions.
}

\begin{document}

	\maketitle
	\begin{abstract}\abstr\end{abstract}

	\clearpage
	\tableofcontents
	\clearpage

\section{Introduction}

Understanding traffic flows is an important task that possibly impacts billions of commuters, with key challenges including managing congestion, energy consumption and carbon emissions.
Congestion phenomena and resulting increased emissions are heavily impacted by the routing decisions of individual drivers,
which are influenced by predictions for the delays of road segments (see, e.g., \textcite{GNNSurvey} for an overview of convolutional and graph neural network based approaches).
A key aspect in modelling traffic flows is the underlying complex and self-referential system:
the routing decisions depend on and, at the same time, influence the forecasting models, because they directly change the underlying signature of traffic flows.

In this paper, we address this interplay focusing on the popular dynamic traffic assignment (DTA) framework, on which there has been substantial work over the past decades (see the book of \textcite{Ford62}, or the more recent surveys by \textcite{Friesz19}, \textcite{Skutella08} and \textcite{SchmandMathTrafficModels}).

Such DTA models typically consist of two main parts:
A \emph{physical model} describing how the traffic flow propagates through the network and, in particular, creates and is affected by congestion, and a \emph{behavioural model} capturing how individual drivers choose their routes through the network.
The latter can be either \emph{prescriptive} or \emph{descriptive}:
A prescriptive behavioural model consists of rules declaring explicitly how, under any given state of the network, the flow splits over the available routes.
These rules then have to be chosen such that they replicate the actual observed traffic flow.
As an example, \textcite{Bayen2019} considered adaptive node routing protocols based on real-time traffic predictions.
A descriptive behavioural model, on the other hand, describes the possible equilibrium states resulting from the assumed behaviour -- typically as some variant of Wardrop's first principle \cite[p.~345]{Wardrop}, which states that in an equilibrium all used routes have the same and minimal perceived cost.
Here, ``perceived costs'' can be the current travel times, the actual (future) travel times or some other predictions about those future travel times.

These equilibrium models typically assume that all predictions are deterministic.
However, especially for models with limited information, this assumption seems to be unrealistic as, in reality, both measurements of the current state of the network and predictions about the future evolution will come with some random noise.
Because of this noise, the perceived cost will vary among agents  so that a shortest route may be different from one agent to another.

\subsection{Our Contribution}

We present a new DTA framework that integrates both prescriptive and descriptive routing models. 
The main ingredients of our model are two operators:
an \emph{edge loading operator} and a \emph{routing operator}.
The former determines the resulting edge outflows under given edge inflows and serves as a proxy for the underlying physical model\revMinor{.
It }subsumes several well-known physical models.
The latter operator determines the precise outflow splits for a node given the node inflows.
The routing operator is assumed to be set-valued and, thus, can model prescriptive routing behaviour (via single-valued routing operators), complex equilibrium conditions (via possibly multi-valued operators) and also stochastic models where measurements for predicting traffic are noisy.
Within this framework, we postulate the notion of \emph{\coherent{} flows}, which are dynamic flows that comply with both the given edge loading and routing operator. As our main results we show the following:
 
\begin{enumerate}
	\item We prove existence of coherent flows under natural continuity conditions (see \Cref{thm:existence-finite-time-horizon,thm:ExtensionIfContinuous}).
	This generalizes various existence results from the literature, e.g., 
	the results of \textcite{Bayen2019} for prescriptive routing
	operators, the result of \textcite{PredictionEquilibria} for prediction based equilibria and the general existence result for full information equilibria by \textcite{CCLDynEquil}.
	Our results strictly generalize these existing results since we allow a larger class of physical flow models.
	\item For prescriptive operators, we show that coherent flows are unique, if the operators fulfil a Lipschitz condition on arbitrarily small extension intervals (see \Cref{thm:unique-existence-for-lipschitz-operators}).
	For this, we build on an idea originally proposed by \textcite{Bayen2019} and, as a by-product, correct a technical lemma used in their proof.
	We generalize their result to a larger class of physical models and routing operators that are not necessarily volume-based.
	\item Finally, we introduce a new equilibrium concept called \emph{stochastic prediction equilibrium} that generalizes the dynamic prediction equilibrium to allow for random measurement errors in the predictions (see \Cref{def:SPE}).
	We show that the corresponding routing operator is prescriptive and that a unique \revMinor{stochastic prediction equilibrium} exists, if the random distribution is well-behaved (see \Cref{cor:stochastic-prediction-equilibrium}).
\end{enumerate}

	The rest of this paper is organized as follows: In \Cref{sec:Model} we formally introduce our abstract framework and define our central object of study: \emph{Coherent flows}. Next, in \Cref{sec:ExistenceUniqueness}, we provide several results on the existence and uniqueness of coherent flows.
	In \Cref{sec:Applications} we demonstrate how these results apply to well-known physical and behavioral models.
	Finally, in \Cref{sec:SPE}, we introduce the new concept of stochastic prediction equilibria and analyze their existence and uniqueness by leveraging our framework.

\subsection{Related Work}

Substantial effort has been made to use AI-based methods for traffic predictions.
Due to the high number of relevant papers, we can only highlight a few and refer the reader to \cite{SurveyTrafficPredictionsWithAI} and \cite{GNNSurvey} for an overview on AI-based methods and their challenges in traffic prediction, and on the use of graph neural networks, respectively.
\Textcite{YYZ18} model the temporal dependency as so-called spatio-temporal graph convolutional networks.
Also, \emph{graph attention networks}, introduced by \textcite{VCCRLB18}, have been used for traffic predictions by~\textcite{ZFW020}.
\Textcite{PredictionEquilibria} discuss the integration of such traffic predictions into dynamic traffic assignment models.

Similarly, there is a vast amount of literature on integrating stochastic effects of predicting travel times into traffic assignment models.

\paragraph*{Static Models}
\textcite{Daganzo1977} introduced the concept of a (static) stochastic user equilibrium (SUE).
A SUE is defined as a static path-based flow wherein each (infinitesimal) traveller chooses a path minimizing their perceived travel time.
While the actual travel time cannot be observed by the individual travellers, the \emph{perceived} travel time is a random variable that depends on the actual travel time and that may have different outcomes for every traveller.

In \cite{Sheffi1982}, \citeauthor{Sheffi1982} formulated an equivalent minimization problem which allowed them to prove that a unique SUE exists under natural assumptions on the probability distributions and the travel time-flow relationship, and to propose an algorithm computing SUE.

\textcite{Baillon2006} formulated a concept called Markovian traffic equilibrium (MTE) that generalizes SUE and the deterministic user equilibrium concept in a common framework.
They proposed an edge-based formulation, where travellers update their route choice at every intermediate node of their journey after a new outcome of the believed random travel time is revealed.

\Textcite{CW2016} presented an asymptotic analysis of a day-to-day learning dynamic for static stochastic and deterministic user equilibria.

\paragraph*{Dynamic Models}
\revMinor{As opposed to static models, our model considers within-day dynamics over continuous time.}
Dynamic equilibrium flows have been studied extensively in the transportation science community,
see~\cite{HFYGeneralizedVickreyModel} and~\cite{Friesz19} for a survey.
Starting with the work of~\textcite{KochSkutellaDEandPoA}, the mathematics and computer science community focused (mostly) on \revMinor{Vickrey's queuing} model and investigated several research questions for different behavioral models covering equilibrium existence  (cf.~\citealt{CCLDynEquil,DynamicFlowswithAdaptiveRouteChoice,GH24}), 
the computational complexity of equilibrium computation (cf.~\citealt{KaiserThinFlowComputation,FiniteTimeCombinatorialAlgorithmForIDE}) and the price of anarchy (cf.~\citealt{CorreaCO22,PriceOfAnarchyForIDE}).

\Textcite{Han2003} introduced a variant of the SUE that incorporates time-dependent flows and travel times.
They use as underlying physical model \revMinor{Vickrey's queuing} model together with a logit-based route assignment model, and present a solution method to compute these equilibria.
This model relies on path-based inflows implying that agents cannot adapt their route mid-journey.
\Textcite{PazGuala2024} formulated an \emph{arc-based} dynamic stochastic user equilibrium inspired by \cite{Baillon2006} with a logit-based adaptive route-choice model, and compared this approach with the one in \cite{Han2003}.

Logit-based models, however, exhibit unnatural behaviour as they assume independence of travel times of overlapping paths which was first pointed out by \textcite{Schneider1973} and also discussed in \textcite{Baillon2006}.
In order to mitigate this effect, \textcite{Szeto2011} considered a C-logit model instead, which uses a commonality-factor for overlapping paths.
They proposed an algorithm to compute such equilibria in a cell transmission model. 
The motivation for these models is usually to describe the behaviour of agents having stochastic errors in their predictions of travel times.
The models above are all prescriptive in nature as they explicitly prescribe how the flow has to split over the available routes.
\begin{revisedEnv}
We, instead, propose to start with a descriptive model describing the possible choices of individual agents based on their perception of the network in the same way as it is usually done for deterministic models (\cf \eg \citealt{Friesz19,KochSkutellaDEandPoA,DynamicFlowswithAdaptiveRouteChoice}).
This not only solves the problem of the interdependence of overlapping paths but can also be used to justify under what assumption which prescriptive model accurately describes the traffic flow.
\end{revisedEnv}

\section{Model}\label{sec:Model}

\revised{
	The model we propose is a macroscopic model in which individual agents have no influence on the  resulting traffic assignment; as such the demand is modelled as a continuum of flow particles that enter the network at a given rate.
	Each infinitesimal flow particle can be thought of as a traffic participant or agent.
	We distinguish between a \emph{physical} and a \emph{behavioral} model.
	The physical model describes the physical propagation of flow within edges.
	It is assumed that once flow enters an edge, it has to traverse the edge according to the ``laws'' of the physical model.
	The behavioral model then describes the routing decisions of the agents, that happen at every intermediate node on their journey toward the destination (implying that routing decision may adapt to the evolution of the flow over time).
	Finally, our main objects of study in this work are \emph{coherent} flows, \ie flows that are consistent with both a given physical and a given behavioral  model and that fulfill flow conservation (\cf \Cref{def:coherent}).

	Before we describe this model formally, we introduce some notation:
}

\begin{notation}[Function spaces, norms and special functions]\label{not:function-spaces}
	\revised{Throughout this paper, we fix some $p\in (1, \infty)$ and let $q$ denote its conjugate defined by $\nicefrac{1}{p} + \nicefrac{1}{q} = 1$.}
	\revMinor{We denote by $\rateFcts\subset \pLocInt$ the set of locally $p$-integrable functions that vanish on $(-\infty, 0)$;
	functions from this set $\rateFcts$ are used to model rates of flow over time.}
	\revised{
	We denote by $\splitFcts$ the set of measurable functions from $\IR$ to $[0, 1]$;
	these functions model flow split ratios at nodes as functions over time.
	For any time horizon $\horizon\in\IRnnInf\coloneqq \IRnn\cup\{\infty\}$, we denote by $\rateFcts_{\horizon}$, and $\splitFcts_{\horizon}$ the subsets of these functions that vanish outside $[0, \horizon)$.}

	\revMinor{For a normed space $(X, \normFct_X)$ and some $d\in\IN$, we use the norm $\smallnorm{f}_X \coloneqq \max_{i\in [d]} \smallnorm{f_i}_{X}$ on the product space $X^d$.
	The norm of $\pInt[\tilde{p}][J]$, for $\tilde{p}\in[1,\infty)$ and measurable $J\subseteq \IR$ is denoted by $\norm{f}_{\tilde{p}}\coloneqq (\int_J \abs{f}^{\tilde{p}} \diff\leb)^{1/\tilde{p}}$, where $\leb$ denotes the Lebesgue measure.
	\revised{For $\tilde{p}=\infty$, we denote the norm by $\norm{f}_\infty \coloneqq \mathrm{ess\ sup} \abs{f} = \inf\Set{ M \geq 0 | \foraall x: \abs{f(x)} \leq M }$.
	Unless stated otherwise, the set $\rateFcts_T^d$, for $d\in\IN$ and $\horizon\in\IRnn$, is equipped with the topology induced by the $p$-norm.
	Furthermore, on the set $\splitFcts_T^d$ with $\horizon\in\IRnn$, the topology induced by the $\tilde{p}$-norm coincides with the one induced by $\normFct_1$ for all $\tilde{p}\in(1,\infty)$ due to the boundedness of the codomain (\cf \Cref{prop:p-integrable-bounded-codomain}).}}

	We say that two vectors $f$, $g$ of measurable functions on $\IR$ \emph{coincide until time $\horizon\in\IRnnInf$} if $\CharF_{[0,\horizon)} \cdot f = \CharF_{[0,\horizon)} \cdot g$ is fulfilled almost everywhere.
	Here, $\CharF_J$ is the characteristic function of~$J$, \ie $\CharF_J(\theta)\coloneqq 1$, if $\theta\in J$, and $\CharF_J(\theta)\coloneqq 0$, otherwise. Similarly, for any statement $A$, we define $\CharF_A\coloneqq 1$, if $A$ is fulfilled, and $\CharF_A\coloneqq 0$, otherwise.

	\revised{A list of symbols used in this paper can be found in \Cref{sec:list-of-symbols}.}
\end{notation}

Let $G=(V,E)$ be a directed graph and $I$ a (finite) set of commodities.
Each commodity $i\in I$ comes with a \revised{destination} node $\dest_i \in V$ and a network inflow rate $u_{v,i} \in \rateFcts$ for every node $v \in V$.
\revised{For simplicity, we assume that $\dest_i$ is reachable from every node $v\in V$\footnote{\revised{
	For all our results, this assumption can be relaxed to the condition that $u_{v,i}$ must be zero  if $\dest_i$ is not reachable from $v$, with the additional requirements that (a) routing operators (\cf \Cref{def:routing-operator}) do not send flow into deadends while dropping the requirement that the flow split has to sum up to $1$ for nodes that cannot reach $\dest_i$, and that (b) edge loading operators (\cf \Cref{def:edge-loading-operator}) do not produce positive outflow of a commodity on an edge, if no flow of that commodity has entered the edge.
}}.}

\revMinor{Further, in our model, we assume that a dynamic flow in such a network can be (fully) described in terms of inflow and outflow rates over time at each edge:}
\begin{definition}[Flow]
	A \emph{(dynamic) flow} is a vector $f\in\rateFcts^{E\times I\times\{+,-\}}$ where $f^+_{e,i}$ denotes the inflow rate (over time) and $f^-_{e,i}$ the outflow rate of commodity $i$ at edge $e$.
\end{definition}

\revised{
\begin{remark}
	Note that flows are thus assumed to be locally $p$-integrable with $1 < p < \infty$.
	Integrability plays a central role since in many physical models, the travel time depends on quantities such as flow volume on edges (the difference of inflow rate and outflow rate integrals), or other aggregate measures.
	We disallow $p\in \{1, \infty\}$ as the proof of our main existence result (\Cref{thm:existence-finite-time-horizon}) requires $\rateFcts_T$ to be a reflexive Banach space for $\horizon\in\IRnn$.
\end{remark}
}

\begin{notation}[Graphs]
	For a node $v\in V$, we denote by $\edgesFrom{v}$ and $\edgesTo{v}$ the sets of outgoing and incoming edges of~$v$, respectively.
	\revised{
	If an edge $e$ leads from node $v$ to node $w$, that is, $e\in\edgesFrom{v}\cap\edgesTo{w}$, we may write\footnote{\revised{Despite this notation, our results hold also if parallel edges exist.}} $e=vw$.
	}
\end{notation}

\subsection{Physical Flow Model}

\revised{
The physical (flow) model describes the flow propagation within edges.
In particular, it determines the outflow rate $f_{e,i}^-(\theta)$ at each edge $e$ for all times $\theta$ given the inflow rate functions $f_{e,i}^+$ over time (both at past and future times) and for all edges in the network.
We abstract away the physical details of the propagation within the edge to this mapping, which we term edge loading operator.
}

\begin{definition}\label{def:edge-loading-operator}
	\revMinor{An \emph{edge loading operator} is a function}
	\[\EdgeLoading: \rateFcts^{E \times I} \to \rateFcts^{E\times I}\]
	mapping vectors $f^+ = (f^+_{e,i})_{e,i}$ of edge inflow rates to vectors $f^- = (f^-_{e,i})_{e,i}$ of edge outflow rates.
\end{definition}

\begin{definition}\label{def:ConsistentWithPhysicalModel}
	A flow $f$ is \emph{consistent with the \revised{edge loading operator}~$\EdgeLoading$ until time $\horizon\in\IRnnInf$} if it satisfies
	$f^-_{e,i}(\theta) = \EdgeLoading(f^+)_{e,i}(\theta)$
	for all $e\in E, i\in I$ and almost all $\theta<\horizon$.
\end{definition}

The above definition determines how particles flow through the edges of the network.
Flow-conservation at nodes prescribes that particles entering a node must leave the node again without waiting.

\begin{definition}\label{def:FlowConservation}
	A flow $f$ \emph{fulfils flow conservation until time $\horizon\in\IRnnInf$} if 
	\[
		\sum_{e\in\edgesFrom{v}} f^+_{e,i}(\theta) = \CharF_{v\neq \dest_i} \cdot \nodeInflow{f}{v}{i}(\theta)
\]
	holds for all $i \in I$, $v \in V$ and almost all $\theta<\horizon$ where
	\optDisplay{
		\nodeInflow{f}{v}{i}(\theta) \coloneqq u_{v,i}(\theta)+\sum_{e \in \edgesTo{v}}f^-_{e,i}(\theta)
	}
	denotes the inflow rate of commodity $i$ into node $v$.
\end{definition}

Important properties of a physical flow model include local boundedness \revMinor{and} causality.
\revised{Local boundedness means that on finite intervals, the range of possible outflow rates is norm-bounded:}

\begin{definition}\label{def:locally-bounded}
	We call an edge loading operator $\EdgeLoading$ \emph{locally bounded} if for any finite $\horizon\in\IRnn$ there exists some $\bound'_\horizon\in\IRnn$ such that $\smallnorm{\CharF_{[0, \horizon]}\cdot \EdgeLoading(f^+)}_{\revMinor{p}} \leq \bound'_{\horizon}$ holds for all flows $f$.
\end{definition}

\revised{Causality describes the property that an edge loading operator cannot look into the future; it may only depend on the inflow rates at past times possibly up to the present:}

\begin{definition}\label{def:causal-edge-loading}
	An edge loading operator $\EdgeLoading$ is called \emph{causal} if whenever two edge inflow vectors $f^+$ and $g^+$ coincide until some time $\horizon$, then there exists some $\alpha\in\IRnn$ such that $\EdgeLoading(f^+)$ and  $\EdgeLoading(g^+)$ coincide until $\horizon+\alpha$.
	If $\alpha$ \revMinor{can} always be chosen strictly positively, we call $\EdgeLoading$ \emph{strictly causal}.
	If $\alpha$ \revMinor{can} be chosen strictly positively and independently of $g^+$ (but, potentially, still dependent on~$\horizon$ \revMinor{and $f^+$}), we call $\EdgeLoading$ \emph{uniformly strictly causal}.
\end{definition}

Many well-known physical models \revMinor{admit an edge loading operator in our sense}, in particular, the ones resulting from the following two widely used edge-dynamics:
Vickrey's queuing model \cite{VickreyQueuingModel, CCLDynEquil} with non-negative free-flow travel times and positive edge capacities, and the affine-linear volume-delay dynamics (see \eg \citealt{ZhuM00,Bayen2019}) with positive free-flow travel times and positive volume-coefficients.
Both models are causal and locally bounded, and, if all free-flow travel times are positive, even uniformly strictly causal.
We refer to \ifnoappendix\Cref{EXT:sec:physical-model-examples} in the supplementary material\else\Cref{sec:physical-model-examples}\fi{} for a formal introduction and discussion of these models.

Since we often work with flows restricted to a finite time horizon, we introduce a specific notation for the corresponding edge loading operator.

\begin{notation}\label{not:restrictedEdgeLoading}
	For a finite $\horizon\in\IRnn$ and an edge loading operator $\EdgeLoading$, let $\EdgeLoading_\horizon$ denote the mapping \[
	\EdgeLoading_\horizon: \rateFcts_T^{E\times I} \to \rateFcts_\horizon^{E\times I},
	\quad
	f^+ \mapsto \CharF[[0,\horizon]]\cdot\EdgeLoading(f^+),
	\]
	where $\rateFcts_\horizon$ is the set of functions in $\rateFcts$ vanishing on $\IR\setminus [0, \horizon]$.
	Note that $\rateFcts_\horizon\subseteq \pInt$. 
\end{notation}

\subsection{Behavioural model}

Based on these physical constraints, the \emph{behavioural model} describes how \revMinor{flow is} routed \revMinor{through the network}.
Whenever \revMinor{an infinitesimal agent} arrive\revMinor{s} at a node $v\neq \dest_i$, they have to decide which edge to enter next based on the past, current, and potentially future state of the network.
\revised{
The collective decisions of all these agents of a commodity $i$ arriving in $v$ at time $\theta$ result in a \emph{flow split} $(r_{e,i}(\theta))_{e\in\edgesFrom{v}}$ that describes for each outgoing edge of~$v$ the proportion of flow entering this edge.
As travel times influence the routing decisions, the actualized flow splits depend on the flow $f$ itself.
Further, since for a given flow $f$, the agents may have multiple options of equal quality, there might be a multitude of different decision outcomes.
Thus,} we introduce the concept of a set-valued routing operator:

\revised{
Recall, that $\splitFcts$ denotes the set of measurable functions from $\IR$ to $[0,1]$.
}

\begin{definition}\label{def:routing-operator}
    Let $\ROp$ be a set-valued mapping of the form
	\[
		\ROp: \rateFcts^{E \times I \times \set{+,-}} \rightrightarrows \revised{\splitFcts^{E\times I}}
	\]
    that maps flows $f$ to sets~$\ROp(f)$ of allowed flow splits at nodes.
	We say $\ROp$ is a \emph{routing operator} if all such flow splits~$r \in \ROp(f)$ satisfy $\sum_{e \in \edgesFrom{v}}r_{e,i}= \CharF_{v\neq \dest_i}$ for all nodes $v \in V$ and commodities $i \in I$.
\end{definition}

In other words, given a flow $f$, an element $r$ of $\ROp(f)$ describes for every time $\theta$ and node $v\neq \dest_i$ how the incoming flow rate into node~$v$ splits over the outgoing edges{} of $v$.
This flow-split originates from the particles' decisions that may depend on the flow $f$ and its induced (past and/or future) travel times, or might be enforced by some infrastructure operator.

Naturally, we want to study flows that are consistent with a given routing operator:

\begin{definition}\label{def:ConsistentWithBehaviouralModel}
	A flow $f$ is \emph{consistent with a routing operator~$\ROp$ until time $\horizon \in \IRnnInf$} if there exists some $r\in\ROp(f)$ such that 
	\[
		f^+_{e,i}(\theta) = r_{e,i}(\theta)\cdot \sum_{e'\in\edgesFrom{v}} f^+_{e',i}(\theta) 
	\]
	holds for all $i \in I$, $e=vw \in E$ and almost all $\theta<\horizon$.\end{definition}

As for the physical model, we first define \revMinor{two} key characteristics of \revMinor{routing operators}.
\revised{First, we introduce causality of routing operators, which (as for the causality of edge loading operators in \Cref{def:causal-edge-loading}) means that the operator may only depend on the past evolution of the flow.}

\begin{notation}\label{not:restrictedRoutingOp}
	For a routing operator $\ROp$, we define \optDisplay{
		\CharF_{[0,\horizon]}\cdot \ROp(f) \coloneqq \set{ \CharF_{[0,\horizon]}\cdot r | r\in\ROp(f) }
	}
	for all $f\in\flowSet$ and $\horizon\in\IRnnInf$,
	and we \revMinor{use} $\ROp_\horizon$ \revMinor{to denote} the mapping \[
		\ROp_\horizon: \rateFcts_\horizon^{E\times I\times \{ +, - \}} \rightrightarrows \revised{\splitFcts_\horizon}^{E\times I},
		\quad
		f \mapsto \CharF_{[0,\horizon]}\cdot \ROp(f).
	\]
\end{notation}

\begin{definition}\label{def:causal-routing-operator}
	A routing operator $\ROp$ is called \emph{causal} if whenever two flows $f$ and $g$ coincide until some time $\horizon$, then there exists some $\alpha\in\IRnn$ with $\CharF_{[0,\horizon + \alpha]} \cdot \ROp(f) =  \CharF_{[0,\horizon + \alpha]} \cdot \ROp(g)$.
	If $\alpha$ \revMinor{can} always be chosen strictly positively, we call the operator \emph{strictly causal}.
	If $\alpha$ \revMinor{can} be chosen strictly positively and independently of $g$, we call it \emph{uniformly strictly causal}.
\end{definition}

\citet*{Keimer2020} study predefined routing policies defining the ``laws of routing''.
Adopting this perspective means that particles arriving at intersections are routed according to prescriptive rules.
These rules are modelled as the class of \emph{prescriptive} routing operators\footnote{Note that \citet{Bayen2019} originally introduced the term ``routing operator'' only for what we now call a prescriptive routing operator.}:

\begin{definition}\label{def:prescriptive}
	We call a routing operator $\ROp$ \emph{prescriptive} if for every flow $f$ the set $\ROp(f)$ consists of exactly one element.
	In this case, we denote this unique element by $\ROe(f)$.
\end{definition}

We are interested in flows that are consistent with both a given \revMinor{edge loading operator  $\EdgeLoading$ and a given routing operator $\ROp$}.

\begin{definition}\label{def:coherent}
	Let $\EdgeLoading$ be an edge loading operator and $\ROp$ a routing operator.
	A flow $f$ is called \emph{$\EdgeLoading$-$\ROp$-\coherent} until time $\horizon\in\IRnnInf$ if $f$ is consistent with $\EdgeLoading$ and $\ROp$ until time $\horizon$ and fulfils flow conservation until time $\horizon$.
	If $\EdgeLoading$ and $\ROp$ are clear from the context, we say $f$ is \emph{\coherent} until $\horizon$.{}
\end{definition}

\section{Existence and Uniqueness}\label{sec:ExistenceUniqueness}

In this section we provide several sufficient properties of edge loading and routing operators to guarantee existence and -- in the second part -- uniqueness of coherent flows.

\subsection{Existence}

We start this subsection by showing that for finite time horizons, coherent flows exist given that $\EdgeLoading$ and $\ROp$ fulfil a continuity condition.
We then extend this result to the infinite time horizon if either the models ensure that all relevant flows terminate in \revMinor{finite} time or if both operators are causal.
For the latter case, we show that the time horizon of a coherent flow may always be strictly extended if the operators are continuous on a (small enough) extension interval, and we conclude existence until $\infty$ using Zorn's Lemma.

\begin{revisedEnv}
Existence of coherent flows is by no means trivial due to the self-referential nature of the edge loading and routing operators, coupled by the flow conservation constraint:
Given edge inflow rates, the edge loading operator determines edge outflow rates and thus, the set of flow splits allowed by the routing operator.
The edge outflow rates in turn determine the effective node inflow rate, and, together with the allowed flow splits, yield a new set of possible inflow rates.
A further complication arises if the operators are non-causal, that is, they use future information about the flow (e.g., when agents have perfect foresight):
solving for a coherent flow in a forward-in-time fashion is then typically impossible.

Hence, we apply a fixed point argument for our general existence theorem.
As a premise, some regularity conditions on both the edge loading and the routing operator are necessary.
For the latter, we require that the mapping of tuples of flows and effective node inflow rates to products of corresponding allowed routings and the node inflows is a closed operator:
\end{revisedEnv}

The \emph{graph} of a set-valued function $\Gamma: M \rightrightarrows N$ is defined as
\[
	\graph(\Gamma)\coloneqq \Set{ (x,y) \in M\times N | y\in\Gamma(x) }.
\]

\newcommand{\usg}[1][]{\ifthenelse{\equal{#1}{}}{\tilde g}{\tilde g^{(#1)}}}
\begin{revisedEnv}
	\begin{definition}\label{def:regular-routing-operator}
		For a routing operator $\ROp$ and $\horizon\in\IRnn$, we say that $\ROp_T$ is \emph{regular} if the graph of
		\begin{equation}\label{eq:regularity}
			\begin{aligned}
			&\tilde{\Gamma} : \rateFcts_T^{(E\times I)^2 + V \times I} \rightrightarrows \rateFcts_T^{E\times I}, \\
			&(f, \usg) \mapsto \Set{ g^+ \in \rateFcts_T^{E\times I} | \exists r\in \ROp_T(f):~\forall e=vw\in E: g^+_{e,i} = r_{e,i}\cdot \usg_{v,i} }
			\end{aligned}
		\end{equation}
		is sequentially weakly\footnote{
			\revMinor{For all $p\in [1,\infty)$ with conjugate $q\in (1,\infty]$ satisfying  $\nicefrac{1}{p} + \nicefrac{1}{q} = 1$,} a sequence $(f^{(\revMinor{n})})_\revMinor{n}$ \emph{converges weakly} to some $f$ in $\pInt[p]$ iff \revised{$\int_{\IR} f^{(n)} \cdot g \diff\leb$ converges to $\int_{\IR} f \cdot g \diff\leb$} for all $g\in \pInt[q]$.
		} closed \wrt $\normFct_p$.
	\end{definition}

	Many natural routing operators (especially those representing the result of individual behaviour) are decoupled between nodes, times, and commodities,
	in the sense that given a flow, the choice of routing at individual nodes and times and of individual commodities is independent from each other.
	Note that, the sets of allowed routings
	(at each node and of each commodity) can nevertheless depend on the flow in the whole network!
	For such separable routing operators an easier to verify sufficient condition for being regular is that for any fixed flow the sets of allowed routings at any node and for any commodity is convex and this set depends (weak-strong) continuously on the flow.
	This is formalized in the following definition:

	\newcommand{\family}{\mathcal{K}}\begin{definition}\label{def:decomposable}
		Let $\ROp$ be a routing operator and $\horizon\in\IRnn$.
		We call $\ROp_\horizon$ \emph{(regularly) decomposable},
		if for every node $v\in V$ and commodity $i\in I$ there exists an (arbitrary) index set $\family_{v,i}$
		and a family $(H_k, b_k)_{k\in \family_{v,i}}$ of functions
		$H_k: \flowSet_\horizon \to L^q([0,\horizon])^{\edgesFrom{v}}$, $b_k: \flowSet_\horizon \to L^q([0,\horizon])$ with $\nicefrac{1}{p} + \nicefrac{1}{q} = 1$ such that
		\begin{equation}\label{eq:decomposability}
			\ROp_T(f) = \Set{
				r \in \splitFcts_T^{E\times I}
				|
				\forall v\in V, i\in I, k \in\family_{v,i}: \sum_{e\in\edgesFrom{v}} H_k(f)_e \cdot r_{e,i} \leq b_k(f)
			}
		\end{equation}
		and the functions $H_k$ are sequentially weak-strong continuous\footnote{\revised{For normed spaces $X$ and $Y$,} a function \revised{$f: M \to Y$, $M\subseteq X$}, is \emph{sequentially weak-weak} or \emph{sequentially weak-strong continuous \revised{\wrt $\normFct_X$ and $\normFct_Y$}} iff \revMinor{$f$} maps weakly convergent sequences \revised{in $(M,\normFct_X)$} to weakly or strongly convergent sequences \revised{in $(Y, \normFct_Y)$}, respectively.
			\revised{If $Y=\IR$, we always use the term \emph{sequentially weak-strong continuous}, since the weak and strong topologies of $\IR$ coincide.}
		} \wrt $\normFct_p$ and $\normFct_q$, and each function $b_k$ fulfils at least one of the following two conditions:
		\begin{enumerate}[label=(\roman*)]
			\item $b_k$ is sequentially weak-strong continuous \wrt $\normFct_p$ and $\normFct_q$, or
			\item\label{decomposable:usc} $b_k$ is pointwise sequentially weak-strong upper-semicontinuous, \ie for all sequences $(f^{(n)})_n$ weakly converging to $f$ in $\flowSet$, we have 
			\[
				\foraall \theta: \limsup_{n\to\infty} b_k(f^{(n)})(\theta) \leq b_k(f)(\theta),
			\]
			and $b_k$ is sequentially dominated, \ie for all sequences $(f^{(n)})_n$ converging weakly to $f$, there exists some $B\in\pInt[q][[0, \horizon]]$ such that $\abs{b_k(f^{(n)})(\theta)} \leq B(\theta)$ holds for all $n\in\IN$ and almost all $\theta\in[0,\horizon]$.
		\end{enumerate}
	\end{definition}

	Note how in this definition, each constraint $(H_k,b_k)$ in \eqref{eq:decomposability} defines a flow-dependent half-space on the set of flow splits of the outgoing edges of a node--commodity pair for every point in time.
	Thus, for example, using finitely many such constraints yields a flow- and time-dependent polyhedron (in fact, a polytope considering the additional implicit constraint $r_{e,i}\in\splitFcts_T$) in standard form lying in $\IR^{E\times I}$ where the constraints are decoupled both between nodes and between commodities.

	\begin{lemma}\label{lem:decomposable-implies-regular}
		If $\ROp_\horizon$ is decomposable and has non-empty values\footnote{\revised{We say a set-valued function $\Gamma: M \rightrightarrows N$ \emph{has non-empty} or \emph{convex values} if the set $\Gamma(x)$ is non-empty or convex, respectively, for all $x\in M$.
		}}, then $\ROp_\horizon$ is regular and has convex values.
	\end{lemma}
	\begin{proof}
		We first observe that convexity of $\ROp_\horizon(f)$ for all $f\in\flowSet_\horizon$ is clear, since the constraints in~\eqref{eq:decomposability} are pointwise linear in~$r$ and, thus, convex combinations fulfil the constraints, too.
		It remains to show that $\ROp_T$ is regular.

		Let $(f^{(n)}, \usg^{(n)}, g^{(n),+})$ be a sequence in the graph of the set-valued function $\tilde{\Gamma}$ in \eqref{eq:regularity},
		and assume this sequence converges weakly to $(f, \usg, g^+)$ in $\rateFcts_T^{(E\times I \times \{ +,- \}) + (V\times I) + (E \times I)}$.
		Then, for every $n$, there exists $r^{(n)}\in \ROp_\horizon(f^{(n)})$ with $g_{e,i}^{(n),+} = r^{(n)}_{e,i}\cdot \usg^{(n)}_{v,i}$ for all $e=vw\in E$ and $i \in I$.
		We now need to show that there exists some $r\in\ROp_\horizon(f)$ such that $g_{e,i}^+ = r_{e,i} \cdot \usg_{v,i}$ holds for all $e=vw\in E$ and $i \in I$.

		\begin{claim}\label{claim:DecomposabilityBoundary}
			For every $v\in V, i\in I$, and constraint in \eqref{eq:decomposability} indexed by $k\in \family_{v,i}$, we have
			\[
				\sum_{e\in\edgesFrom{v}} H_{k}(f)_e \cdot g^+_{e,i} \leq b_k(f)\cdot \usg_{v,i}.
			\]
		\end{claim}
		\begin{proofClaim}
			Note first that, for each $n\in\IN$, multiplying $\usg[n]_{v,i}$ to the constraint in \eqref{eq:decomposability} yields almost everywhere on $[0,\horizon]$ that
			\begin{align}\label{eq:DecomposabilityLemmaProof}
				\sum_{e\in\edgesFrom{v}} H_k(f^{(n)})_e \cdot g_{e,i}^{(n),+} \leq b_k(f^{(n)})\cdot \usg[n]_{v,i}.
			\end{align}
			\begin{description}
				\item[Case (i):] $b_k$ is sequentially weak-strong continuous.
				Since all $H_k(\cdot)_{e}$ and $b_k$ are weak-strong continuous to $L^q([0,\horizon])$, 
				the products $H_k(f^{(n)})_e\cdot g_{e,i}^{(n),+}$ and $b_k(f^{(n)}) \cdot \usg[n]_{e,i}$ converge weakly in $L^1([0, \horizon])$ to $H_k(f)_e\cdot g^+_{e,i}$ and $b_k(f)\cdot \usg_{e,i}$, respectively (see \Cref{prop:product-of-weak-and-strong}).
				Hence, the inequality above holds also for the limits almost everywhere on $[0, \horizon]$ (\cf \Cref{prop:inequalities-weakly-continuous}).

				\item[Case (ii):] $b_k$ is pointwise sequentially weak-strong upper-semicontinuous and sequentially dominated.
				Let $B$ be the dominating element in $\pInt[q][[0, \horizon]]$ for the sequence $(b_k(f^{(n)}))_n$.
				We use the shorthand $b^*\coloneqq \limsup_{m\to\infty} b_k(f^{(m)})$.
				Then, for all $n\in\IN$ we have
					\[\smallabs{\sup_{m\geq n} b_k(f^{(m)})(\theta) - b^*(\theta)}^q \leq (2\cdot B(\theta))^q \quad\text{for almost all } \theta\in[0,\horizon].\]
				By Lebesgue's dominated convergence theorem, the limit of $\int_{[0, \horizon]} \smallabs{\sup_{m\geq n} b_k(f^{(m)}) - b^*}$ is~$0$.
				In particular, $\sup_{m\geq n} b_k(f^{(m)})$ is in $\pInt[q][[0, \horizon]]$ and converges strongly to $b^*$.

				This strong convergence implies that $(\sup_{m\geq n} b_k(f^{(m)}))\cdot \usg[n]_{v,i}$ converges weakly in $L^1([0,\horizon])$ to $b^* \cdot \usg_{v,i}$ (see \Cref{prop:product-of-weak-and-strong}).
				Since $\usg[n]_{v,i}$ is non-negative, we get from~\eqref{eq:DecomposabilityLemmaProof} that the inequalities 
				\[
					\sum_{e\in\edgesFrom{v}} H_k(f^{(n)})_e \cdot g^{(n),+}_{e,i} \leq (\sup_{m\geq n} b_k(f^{(m)}))\cdot \usg[n]_{v,i}
				\]
				hold almost everywhere.
				Since both sides converge weakly in $L^1([0, \horizon])$, this inequality also applies to the limits of both sides almost everywhere (\cf \Cref{prop:inequalities-weakly-continuous}).
				The inequality stays valid after replacing $b^*$ on the right-hand side with $b_k(f)$ due to the pointwise upper semi-continuity of $b_k$.
			\end{description}
		\end{proofClaim}
		
		Since we assumed that $\ROp_T$ has non-empty values, let $\underline{r}\in\ROp_T(f)$ be an arbitrary ``fallback'' routing.
		Let us now define $r$ as the function
		\[
			r_{e,i}(\theta) \coloneqq \begin{cases}
				\frac{g^+_{e,i}(\theta)}{\usg_{v,i}(\theta)}, &\text{if $\usg_{v,i}(\theta) > 0$,}\\
				{\underline{r}}_{e,i}(\theta), &\text{otherwise.}
			\end{cases}
		\]
		This function fulfills all constraints for the set $\ROp_T(f)$ in~\eqref{eq:decomposability}:
		It is measurable, non-negative and has its support in $[0,\horizon]$ since it is build from functions which have the same properties.
		Moreover, it is bounded by~$1$ since $g^{(n),+}_{e,i} = r_{e,i}^{(n)} \cdot \usg[n]_{v,i} \leq \usg[n]_{v,i}$ holds for all $n$, and thus also for the weak limits on both sides (\cf \Cref{prop:inequalities-weakly-continuous}).
		Hence, we have $r \in \splitFcts_T^{E\times I}$.
		Finally, $r$ also satisfies the defining inequalities given in~\eqref{eq:decomposability} and, therefore, $r \in \ROp_T(f)$:
		For times $\theta$ with $\usg_{v,i}(\theta) > 0$, this follows from \Cref{claim:DecomposabilityBoundary}, and otherwise it follows from $\underline{r}\in\ROp_T(f)$.
	\end{proof}

	We are now ready to formulate our central existence result:

\end{revisedEnv}

\begin{theorem}\label{thm:existence-finite-time-horizon}
	Let $\horizon\in\IRnn$ and let $\EdgeLoading$ be a locally bounded edge-loading operator such that $\EdgeLoading_\horizon$ is sequentially weak-weak continuous \revised{\wrt $\normFct_p$} and let $\ROp$ be a routing operator such that $\ROp_\horizon$ \revised{is regular\footnote{
		In the first version of this paper, we claimed a stronger version of this theorem by requiring that the routing operator $\ROp_\horizon$ has a sequentially weakly closed graph w.r.t. $\normFct_p$ (instead of requiring regularity as in \Cref{def:regular-routing-operator}).
		However, as pointed out by an anonymous reviewer, the theorem's proof rested 
		on the (incorrect) claim that the product of two weakly convergent 
		sequences is again weakly convergent. Consequently, in this revision, we replaced the requirement accordingly and applied analogous adjustments to \Cref{thm:ExtensionIfContinuous} and \Cref{cor:prescriptive-operators}.
	}{} and} has non-empty and convex values.
	
	{} Then, there exists a flow $f\in\flowSet_{\horizon}$ which is \coherent{} until time $\horizon$.
\end{theorem}

\newcommand{\kakutaniFanGlicksbergThm}{
	\begin{theorem}[{Kakutani-Fan-Glicksberg Fixed-Point Theorem}]\label{thm:KFGFixedPointThm}
		Let $X$ be a locally convex Hausdorff space, $K \subseteq X$ non-empty, convex and compact.
		Moreover, let $\Gamma: K \rightrightarrows K$ be a set-valued function with closed graph and non-empty, convex \revMinor{values}.
		
		Then the set of fixed points of $\Gamma$ is non-empty and compact, where $x\in K$ is called a \emph{fixed point of~$\Gamma$} if $x\in\Gamma(x)$.
	\end{theorem}
}

	\revMinor{For the proof,} we rely on the Kakutani-Fan-Glicksberg Fixed-Point Theorem (\cf \citealt[Corollary 17.55]{InfiniteDimensionalAnalysis}).

	\kakutaniFanGlicksbergThm
\newcommand{\proofThmExistenceFiniteTimeHorizon}[1][Proof]{
\begin{proof}[#1]
	We define the constant $B_\horizon\coloneqq \max_v\abs{\edgesTo{v}} \cdot B_\horizon' + \norm{\CharF_{[0,\horizon]} \cdot u}_{\revMinor{p}}$ where $B_\horizon'$ is given by the local boundedness of $\Phi$.
	We then define the following set $K$ of candidates for consistent flows
	\revised{
	\begin{align*}
		K \coloneqq \Set{(f^+, f^-) \in \flowSet_\horizon  | \smallnorm{(f^+, f^-)}_{\revMinor{p}} \leq \bound_\horizon }
	\end{align*}
	}
	and the correspondence
	\revised{
	\[
		\Gamma: K \rightrightarrows K, f \mapsto \Set{
			g \in K
			|
			\begin{array}{c}
				g^- = \EdgeLoading_\horizon(f^+) \quad\text{ and }\quad \exists r\in\ROp_\horizon(f): \\
				\forall i \in I, e=vw \in E: g^+_{e,i} = r_{e,i}\cdot \Bigl(u_{v,i}+\sum_{e\revMinor{'} \in \edgesTo{v}}g^-_{e\revMinor{'},i}\Bigr)
			\end{array}
		}.
	\]
	Note that this mapping is well-defined by our choice of $B_T$.}
	\revised{T}he fixed points of $\Gamma$ are \revised{then} exactly the coherent flows until time~$\horizon$.
	\revised{This is, because for any such fixed point $f \in \Gamma(f)$ the first equation in the definition of the set~$\Gamma(f)$ is exactly the definition of $f$ being consistent with~$\EdgeLoading$ until~$\horizon$ (i.e., \Cref{def:ConsistentWithPhysicalModel}) while the second equation is flow conservation at nodes together with consistency with the given routing operator until~$\horizon$ (i.e., \Cref{def:FlowConservation,def:ConsistentWithBehaviouralModel}).}
	
	We now want to show that $\Gamma$ has such a fixed point by applying \Cref{thm:KFGFixedPointThm} to $K\subseteq L^p(\IR)^{\revised{E\times I \times \{+,-\}}}$ equipped with the weak topology.
	Hence, we have to show that the requirements of this \namecref{thm:KFGFixedPointThm} are satisfied:

	\begin{claim}\label{claim:PropertiesOfK}
		$K$ is non-empty, convex, weakly closed and weakly compact. 
	\end{claim}

	\begin{proofClaim}
		\begin{structuredproof}
			\proofitem{$K$ non-empty} $K$ contains at least the zero vector.
			\proofitem{$K$ convex} Clear.
			\proofitem{$K$ weakly closed} $K$ is strongly closed and convex and, thus, weakly closed by \cite[Ch.~V, 1.5.~Corollary]{ConwayACourseInFunctionalAnalysis}.
			\proofitem{$K$ weakly compact} $K$ is a subset of the (norm-)closed ball with radius $B_T$ in $\pInt^{\revised{E \times I\times\{+,-\}}}$. The latter is weakly compact since $\pInt^{\revised{E \times I\times\{+,-\}}}$ is reflexive (\cf \citealt[Theorem~6.25]{InfiniteDimensionalAnalysis}).

			Hence, as a weakly closed subset of a weakly compact set, $K$ is weakly compact as well.
			\qedhere
		\end{structuredproof}
	\end{proofClaim}

	\begin{claim}
		$\Gamma$ has non-empty and convex values.
	\end{claim}

	\begin{proofClaim}
		Fix any $\revised{f} \in K$. Then, we have:
		\begin{structuredproof}
			\proofitem{$\Gamma(f)$ non-empty} \revised{Since $\ROp_T(f)$ is assumed to be non-empty, let $r\in\ROp_T(f)$ be any element thereof.
			Then, the flow $g$ defined by $g^-\coloneqq \EdgeLoading_\horizon(f^+)$ and $g^+_{e,i} \coloneqq r_{e,i}\cdot\bigl(u_{v,i}+\sum_{e\revMinor{'}\in \in \edgesTo{v}}g^-_{e\revMinor{'},i}\bigr)$ is an element of $\Gamma(f)$.}
			
			\proofitem{$\Gamma(f)$ convex} \revised{This follows directly from the convexity of $\ROp_\horizon(f)$: Let $\alpha\in[0,1]$ and let $g,h \in \Gamma(f)$.
			Then, for the outflow rates, it holds $(\alpha g + (1-\alpha) h)^- = \EdgeLoading_T(f^+) = g^- = h^-$.
			Let $\usg_{v,i} \coloneqq u_{v,i} + \sum_{e\in\edgesFrom{v}}\EdgeLoading_T(f^+)_{\revMinor{e,i}}$.
			Then, for the inflow rates, there are $r^g, r^h\in \ROp_T(f)$ with $g_{e,i}^+ = r_{e,i}^g\cdot\usg_{v,i}$ and $h_{e,i}^+ = r_{e,i}^h \cdot \usg_{v,i}$.
			Since $\ROp(f)$ is assumed to be convex, $r^{\alpha} \coloneqq \alpha \cdot r^g + (1-\alpha) r^h$ is in $\ROp_T(f)$, and we have
			$(\alpha\cdot g + (1-\alpha)\cdot h)^+ = r_{e,i}^\alpha \cdot \usg_{v,i}$.
			This concludes the proof.}
				\qedhere
		\end{structuredproof}
	\end{proofClaim}

	\begin{claim}\label{claim:GammaWeaklyClosedGraph}
		$\Gamma$ has a weakly closed graph.
	\end{claim}

	\begin{proofClaim}
		Since $K \times K$ is a weakly closed set by \Cref{claim:PropertiesOfK}, it suffices to show that $\graph(\Gamma)$ is weakly closed in $K \times K$.
		By \Cref{prop:weakly-metrizable}, this is equivalent to showing that $\graph(\Gamma)$ is sequentially weakly closed.
		Hence, fix some sequence \revised{$(f^{(n)}, g^{(n)})_n$} weakly converging to \revised{$(f, g)$} in $K \times K$.

		This implies that $(f^{(n),+})_n$ is weakly converging to $f^+$ and, therefore, $\EdgeLoading_\horizon(f^{(n),+})$ is weakly converging to $\EdgeLoading_\horizon(f^+)$ by the sequential weak-weak continuity of $\EdgeLoading_\horizon$.
		As $g^-$ is the weak limit of $g^{(n),-} = \EdgeLoading_\horizon(f^{(n),+})$, we must have $g^- = \EdgeLoading_\horizon(f^+)$.
        \revised{
		Thus, the sequence $(\usg[n]_{v,i})_n \coloneqq (u_{v,i}+\sum_{e\in\edgesTo{v}} g_{e,i}^{(n),-})_n$ converges weakly to $\usg_{v,i} \coloneqq u_{v,i}+\sum_{e\in\edgesTo{v}} g_{e,i}^-$.
		Further, for all $n$, there exist $r^{(n)}$ in $\ROp_T(f^{(n)})$ such that $g^{(n),+}_{e,i} = r^{(n)}_{e,i} \cdot \usg[n]_{v,i}$.
		By regularity of $\ROp_T$, the graph of the operator in~\eqref{eq:regularity} is weakly sequentially closed.
		Therefore, there exists some $r\in\ROp_T(f)$ such that $g^+_{e,i} = r_{e,i} \cdot \usg_{v,i}$ holds for all $e\in E, i\in I$.
		This implies $g\in \Gamma(f)$.
		}\end{proofClaim}

	Using the previous three claims, \Cref{thm:KFGFixedPointThm} now guarantees the existence of a fixed point of $\Gamma$, \ie a \revised{flow $f \in K$ with $f\in\Gamma(f)$}.
	\revised{As argued above,} this flow then fulfils
	$f^-=\EdgeLoading_\horizon(f^+)$, implying that $f$ is consistent with $\EdgeLoading$ until $\horizon$\revMinor{, and}
	\revised{there is some $r\in\ROp_T(f)$ with} $f_{e,i}^+ = r_{e,i}\cdot\bigl(u_{v,i}+\sum_{e\revMinor{'}\in\edgesTo{v}}f_{e\revMinor{'},i}^-\bigr)$ \revised{for all $e=vw$}, which together imply that $f$ fulfils flow conservation until $\horizon$ and is consistent with the routing operator $\ROp$ until $\horizon$.
\end{proof}
}

	\proofThmExistenceFiniteTimeHorizon[Proof of \Cref{thm:existence-finite-time-horizon}]

The previous theorem shows existence of \coherent{} flows up to any \emph{finite} time horizon in continuous models.
For some models (like certain full information equilibria with finitely lasting network inflow) this is already enough to guarantee existence of a coherent flow for all times as one can just choose a large enough time horizon~$\horizon$ such that all coherent flows terminate before that time (\cf \Cref{cor:existence-termination}).
For other models, this may not be possible (\eg due to infinitely lasting network inflow rates or because coherent flows allow cycling behaviour{} -- see, for example, \cite[Theorem~6.1]{DynamicFlowswithAdaptiveRouteChoice}).
Here, under the additional assumption of causality \revised{(\Cref{def:causal-edge-loading,def:causal-routing-operator})}, we can still reach an infinite time horizon{} (\cf \Cref{thm:ExtensionIfContinuous}){} by following an extension-based approach similar to the one used in \cite[Section~4.1]{GrafThesis}.

\begin{revisedEnv}
	This extension-based approach requires that we can extend any coherent flow with a finite time horizon by some positive (but arbitrarily small) time period.
	Using Zorn's Lemma, we can then obtain a coherent flow until $\infty$:
\end{revisedEnv}

\begin{definition}\label{def:ext-prop}
	A pair of an edge loading operator $\EdgeLoading$ and a routing operator $\ROp$ fulfils the \emph{\extProp}, if for all $\horizon\in\IRnn$ and for every flow $f$ that is \coherent{} until time~$\horizon$, there exists some $\alpha\in\IRp$ and a flow $g$ that coincides with $f$ until time~$\horizon$ and is \coherent{} until $\horizon+\alpha$.
\end{definition}

\begin{lemma}\label{lem:existence-using-ext-prop}
	Assume that $\EdgeLoading$ is locally bounded \revised{(\Cref{def:locally-bounded})} and causal, $\ROp$ is causal \revised{and has non-empty values}, and they fulfil the \extProp.
	Then, there exists a flow that is \coherent{} until time $\infty$.
\end{lemma}
\newcommand{\proofLemExistenceUsingExtProp}[1][Proof]{
\begin{proof}[#1]
	\newcommand*{\pord}{\preceq}
	We call the pair $(f, \horizon)$ of a flow $f$ and a time $\horizon\in\IRnnInf$ a \emph{partially coherent flow} if $f$ is \coherent{} until time $\horizon$ and the essential support of $f$ is contained in $[0,\horizon]$.
	On the set of all partially coherent flows, we define the partial order $\pord$ such that $(f,\horizon)\pord (g,\horizon')$ is equivalent to $\horizon\leq\horizon'$ and $f = \CharF_{[0,\horizon]}\cdot g$.
	We now use Zorn's Lemma (\cite[Lemma~1.7]{InfiniteDimensionalAnalysis}) to show that there exists a \revMinor{$\pord$-maximal partially coherent flow $(f^*, H^*)$}.
	By the \extProp, the horizon \revMinor{$H^*$} of this maximal element must then be $\infty$\revMinor{, and thus $f^*$ is a coherent flow until $\infty$}.

	\revised{We call a subset $C$ of the set of partially coherent flows a \emph{$\pord$-chain} if $\pord$ is a total order on the set $C$.}
	To apply Zorn's Lemma, we need to show that any \revMinor{$\pord$-}chain $C$ of partially coherent flows has an upper bound.
	If $C$ is the empty \revMinor{set}, then $(0, 0)$ is an upper bound\revised{, since for the zero-flow $f\coloneqq 0$, $\ROp(f)$ is non-empty, and therefore, $f$ is consistent with $\ROp$ until time $0$, and thus $f$ is coherent until $0$.}
    Otherwise, let $(f^{(n)}, \horizon^{(n)})_{n\in\IN_{\geq 1}}$ be a non-decreasing (\wrt $\pord$) sequence in $C$ such that $\lim_{n\to\infty} \horizon^{(n)} = \horizon^*$ where $\horizon \coloneqq \sup_{(f,\horizon)\in C}\horizon$.
    We define \[
        f \coloneqq \sum_{n\in\IN_{\geq 1}} \CharF_{[\horizon^{(n-1)}, \horizon^{(n)})} \cdot f^{(n)}
	\]
	with $\horizon^{(0)}\coloneqq 0$.
    We first show that (every entry of) $f$ is locally $p$-integrable:
    If $\horizon=\infty$ holds, this is clear (on a bounded interval, $f$ is equal to some $f^{(n)}$ for large enough $n$).
    Otherwise $\horizon < \infty$ and by the local boundedness of $\EdgeLoading$ \revMinor{there is some $B'_{\horizon}\in\IRnn$ with}
	\begin{align*}
        B'_{\horizon}
		&\geq
        \lim_{n\to\infty} \norm{ \CharF_{[0, \horizon^{(n)}]} \cdot \EdgeLoading(f^{(n), +}) }_p
        \symoverset{1}{=} \lim_{n\to\infty} \norm{ \CharF_{[0, \horizon^{(n)}]} \cdot f^{(n),-} }_p
        = \norm{ \CharF_{[0, \horizon]} \cdot f^- }_p,
	\end{align*}
	where we use the fact that $(f^{(n)},\horizon^{(n)})$ is partially coherent at~\refsym{1}.
    Moreover, $f$ fulfils flow conservation until $\horizon$ (as all $f^{(n)}$ do) and, thus, $\smallnorm{f}_{\revMinor{p}}\leq \smallnorm{\restr{u}{[0,\horizon]}}_{\revMinor{p}} + \max_{v\in V} \abs{\edgesFrom{v}} \cdot \bound'_{\horizon}$.
    By the causality of $\EdgeLoading$ and $\ROp$, $f$ is consistent with $\EdgeLoading$ and $\ROp$ until $\horizon$ and therefore $(f, \horizon)$ is a partially coherent flow and, in particular, an upper bound of the chain $C$.
\end{proof}
}
\proofLemExistenceUsingExtProp

The extension existence property holds if the operators are continuous on the extension time period\revMinor{, as we show in the following}.
We define the following sets{} of functions:
For any (finite) time horizon $\horizon\in\IRnn$, $\alpha\in\IRp$ and flow $f\in\flowSet_{\horizon}$, we define the set of possible extensions of the inflow rates $f^+$ as
\begin{equation}\label{eq:extension-inflows}
	\Omega_{\horizon,\alpha}^+(f^+) \coloneqq \set{ g^+ \in \rateFcts_{\horizon+\alpha}^{E\times I} | \CharF_{[0,\horizon]} \cdot g^+ = \CharF_{[0,\horizon]}\cdot f^+ }
\end{equation}
and the set of flow extensions of $f$ as
\begin{equation}\label{eq:extension-flows}
	\Omega_{\horizon,\alpha}(f) \coloneqq \set{ g \in \flowSet_{\horizon +\alpha} | \CharF_{[0, \horizon]}\cdot g = \CharF_{[0,\horizon]}\cdot f}.
\end{equation}

\begin{theorem}\label{thm:ExtensionIfContinuous}
	Let $\EdgeLoading$ be a locally bounded and causal edge loading operator and let $\ROp$ be a causal routing operator \revised{with non-empty values} such that for every flow $f$ that is coherent until some time $\horizon\in\IRnn$ there exists some $\alpha\in\IRp$ such that
	\begin{enumerate}[label=(\roman*)]
		\item\label{EdgeLoadingContinuous} the map $\restr{\EdgeLoading_{\horizon+\alpha}}{\Omega^+_{\horizon,\alpha}(f^+)}$ is sequentially weak-weak continuous \revMinor{\wrt $\normFct_p$}, and
		\item\label{RoutingContinuous} the map $\restr{\ROp_{\horizon+\alpha}}{\Omega_{\horizon,\alpha}(f)}$ has convex
		values and \revised{the operator $\smallrestr{\tilde{\Gamma}}{\tilde{\Omega}_{\horizon,\alpha}(f)}$ has a sequentially weakly closed graph \wrt $\normFct_p$, where $\tilde{\Gamma}$ is given by~\eqref{eq:regularity} and $\tilde{\Omega}_{\horizon,\alpha}(f)$ is the product set of $\Omega_{\horizon,\alpha}(f)$ and $\{( u_{v,i} + \sum_{e\in\edgesFrom{v}} g_{e,i}^- )_{v,i}\mid g\in \Omega_{\horizon, \alpha}(f) \}$}.
	\end{enumerate}
	Then, $\EdgeLoading$ and $\ROp$ fulfil the \extProp{} and there exists a flow that is coherent until $\infty$.
\end{theorem}
\newcommand{\proofThmExtensionIfContinuous}[1][Proof]{
\begin{proof}[#1]
	Let $h$ be a flow that is \coherent{} until time $\horizon$.
	We show that there exists a flow $f$ that is \coherent{} until time $\horizon+\alpha$ and coincides with $h$ until $\horizon$.
	
	For this purpose, we adjust the proof of \Cref{thm:existence-finite-time-horizon}:
	We use the set~$K$ and correspondence~$\Gamma$ as defined there but with all occurrences of~$\horizon$ replaced by~$\horizon+\alpha$.
	We then add the \revised{constraint} $\CharF_{[0, \horizon]} \cdot (f^+, f^-) = \CharF_{[0,\horizon]}\cdot h$
	to the set $K$ (and restrict the domain and values of $\Gamma$ to the new $K$ and subsets of the new $K$, respectively).
	First, note that $K$ is still non-empty (since $\CharF_{[0,\horizon]}\cdot \revised{h}$ is in $K$), convex and weakly compact.
	To see the latter, note that $K$ is still strongly closed and convex and thus weakly closed.

	\revised{
	Every $\Gamma(f)$ remains non-empty:
	To show this, we first define $g^-\coloneqq \EdgeLoading_{\horizon + \alpha}(f^+)$.
	}By causality of $\EdgeLoading$ and as $f^+$ and $h^+$ coincide until $\horizon$, it follows that
	\[
		\CharF_{[0,\horizon]} \cdot g^-
		=\CharF_{[0,\horizon]} \cdot \EdgeLoading(f^+)
		= \CharF_{[0,\horizon]} \cdot \EdgeLoading(h^+)
		= \CharF_{[0,\horizon]} \cdot h^-.
	\]
	\revised{
	Similarly, by the causality of $\ROp$, we have $\ROp_\horizon(h) = \CharF[[0, \horizon]] \cdot \ROp(f)$.
	Since $h$ is consistent with $\ROp$ until $\horizon$, there exists some $r\in\ROp(f)$ such that we have $\CharF[[0, \horizon]] \cdot h^+_{e,i} = \CharF[[0, \horizon]] \cdot r_{e,i} \cdot (u_{v,i} + \sum_{e\revMinor{'}\in\edgesTo{v}} f_{e\revMinor{'},i}^-)$ for all $e=vw \in E$.
	Thus, we can define the witness $g\coloneqq (g^-, g^+)$ with $g^+_{e,i}\coloneqq r_{e,i}\cdot (u_{v,i} + \sum_{e'\in\edgesTo{v}} g_{e',i}^-)$ which then also fulfils the newly added constraint of $K$, since also $g^-$ and $h^-$ coincide on $[0, \horizon]$.
	Thus, $g\in \Gamma(f)$.
	}

	Clearly, the values of $\Gamma$ remain convex.
	Finally, with the same arguments as in the proof of \Cref{claim:GammaWeaklyClosedGraph}, the graph of $\Gamma$ is again weakly closed.
	
	Hence, by \Cref{thm:existence-finite-time-horizon} there exists a fixed point $\revised{f}$ of $\Gamma$\revised{, and there is some $r\in\ROp_{\horizon + \alpha}(f)$ such that $f$ is consistent with $\ROp$ until $\horizon + \alpha$.}
	The flow $f$ does not only fulfil the desired properties until time $\horizon+\alpha$, but as $f$ is an element of $K$, coincides with $h$ until time~$\horizon$.
	
	Therefore, \revMinor{$f$} is an extension of~\revMinor{$h$} proving the extension existence property for $\EdgeLoading$ and $\ROp$.
	The existence of a coherent flow until~$\infty$ then follows immediately by \Cref{lem:existence-using-ext-prop}.
\end{proof}
}
\proofThmExtensionIfContinuous

	If the model satisfies \emph{strict} causality, then the extension existence property holds even without the continuity assumptions of the previous \namecref{thm:ExtensionIfContinuous}:

\newcommand{\lemmaExistenceIfStrictlyCausal}{
\begin{lemma}\label{prop:existence-if-strictly-causal}
	Assume $\EdgeLoading$ and $\ROp$ are strictly causal and assume that $\ROp$ has non-empty values.
	Then, $\EdgeLoading$ and $\ROp$ fulfil the \extProp.
	
	If, additionally, $\EdgeLoading$ is locally bounded, then there exists a flow that is \coherent{} until time~$\infty$.
\end{lemma}
}
\lemmaExistenceIfStrictlyCausal
\newcommand{\proofPropExistenceIfStrictlyCausal}[1][Proof]{
	\begin{proof}[#1]
		Let $\horizon\in\IRnn$ and let $f$ be a flow that is \coherent{} until $\horizon$.
		Choose any element $r \in \ROp(f)$ witnessing the consistency of $f$ with $\ROp$ until~$\horizon$
		and define $g^- \coloneqq \EdgeLoading_{\horizon + 1}(f^+)$ and $g^+_{e,i} \coloneqq r_{e,i}\cdot(u_{v,i}+\sum_{e\revMinor{'}\in\edgesFrom{v}} g_{e\revMinor{'},i}^-)$
		for $e\in E$, $i\in I$.
		Since $f$ is consistent with~$\EdgeLoading$ until~$\horizon$, this implies that $f$ and $g$ conincide until time~$\horizon$. 
		Hency, by strict causality, there exists some $\alpha\in\IRp$ (without loss of generality assume $\alpha\leq 1$) such that $\CharF_{[0,\horizon+\alpha]} \cdot \ROp(g) = \CharF_{[0,\horizon+\alpha]}\cdot \ROp(f) \ni \CharF_{[0,\horizon+\alpha]}\cdot r$ and $\CharF_{[0,\horizon+\alpha]}\cdot \EdgeLoading(g^+) = \CharF_{[0,\horizon+\alpha]}\cdot \EdgeLoading(f^+) = g^-$ hold.
		Then, $g$ is \coherent{} until $\horizon+\alpha$.
		
		The additional part of the \namecref{prop:existence-if-strictly-causal} then follows immediately by \Cref{lem:existence-using-ext-prop}.
	\end{proof}
}
\proofPropExistenceIfStrictlyCausal

	For non-causal models, on the other hand, extensions of the form used in the proof of \Cref{lem:existence-using-ext-prop} do not work as extending a partially coherent flow can also affect whether or not the flow is coherent before the extension interval.\footnote{The thin-flow extensions used in \cite[Section~4]{CCLDynEquil} within the (non-causal) full information setting are a different type of extension as \iftrue these extensions \else they \fi determine the whole trajectory of the involved particles from source to sink.
	}
	Here, one instead often argues (e.g., \cite[Lemma~3]{CCLDynEquil}) that there exists some large but finite time horizon $\horizon$ such that for any coherent flow until $\horizon$ all flow has already left the network and, thus, it can be extended by the zero flow to a coherent flow until $\infty$.

\newcommand{\defTerminates}{
\begin{definition}
	A flow $f$ \emph{terminates until time $\horizon$} if the essential support of $f$ and the essential support of $\EdgeLoading(f^+)$ are contained in $[0, \horizon]$.
\end{definition}
}
\defTerminates
\newcommand{\corExistenceTermination}{
\begin{corollary}\label{cor:existence-termination}
	Suppose the assumptions of \Cref{thm:existence-finite-time-horizon} hold for some $\horizon\in\IRnn$ and assume the following two conditions:
	\begin{enumerate}[label=(\roman*)]
		\item The essential support of $u$ is contained in $[0, \horizon]$.
		\item\label{condition:finite-consistency} Any flow in $\flowSet_\horizon$ that is \coherent{} until time $\horizon$ and fulfils flow conservation until time~$\infty$ is already consistent with $\EdgeLoading$ until time $\infty$.
	\end{enumerate}
	Then, there exists a flow which is \coherent{} until time $\infty$.
	
	Condition~\ref{condition:finite-consistency} is fulfilled if every flow in $\rateFcts_T^{E\times I\times\{+,-\}}$ that is \coherent{} until time $\horizon$ and fulfils flow conservation until time $\infty$ terminates until time $\horizon$.
\end{corollary}
}
\corExistenceTermination
\newcommand{\proofCorExistenceTermination}[1][Proof]{
	\begin{proof}[#1]
		\Cref{thm:existence-finite-time-horizon} gives us the existence of a flow $f \in \flowSet_{\horizon}$ that fulfils the desired properties until time~$\horizon$.
		Clearly, the flow also fulfils flow conservation and consistency with $\ROp$ after time $\horizon$ as nothing flows into or out of any node then.
		Hence, due to condition~\ref{condition:finite-consistency}, it is consistent with $\EdgeLoading$ until time $\infty$ as well.
		
		Assume now that every flow in $\flowSet_\horizon$ that is consistent with $\EdgeLoading$ and $\ROp$ until time $\horizon$ and fulfils flow conservation until time $\infty$ terminates until time $\horizon$.
		Let $f$ be a flow in $\flowSet_\horizon$ that is consistent with $\EdgeLoading$ and $\ROp$ until time $\horizon$ and fulfils flow conservation until time $\infty$.
		Then, $f$ terminates until time $\horizon$.
		Hence, both $f$ and $\EdgeLoading(f^+)$ vanish outside $[0,\horizon]$ and, therefore, $f$ is consistent with $\EdgeLoading$ until time $\infty$.
	\end{proof}
}
\proofCorExistenceTermination

\subsection{Uniqueness}\label{sec:uniqueness}

In many models, \coherent{} flows are not unique:
For example, neither dynamic Nash equilibria nor instantaneous dynamic equilibria (see \Cref{sec:deterministic-dpe}) are unique in general{} as the simple example of two parallel edges with the same flow-independent travel times shows.
However, for the linear-edge delay model and a class of prescriptive \revised{(\Cref{def:prescriptive})} and causal routing operators, \textcite[Theorem~3.4]{Bayen2019} state\footnote{While their proof is based on an incorrect lemma (see \ifnoappendix\Cref{EXT:sec:bayen-lemma} in the supplementary material\else\Cref{sec:bayen-lemma}\fi), 
we show here that the statement itself is correct and their idea to use the Banach Fixed-Point Theorem as underlying machinery is valid.} that if the routing operator fulfils a Lipschitz condition, the resulting flow is in fact unique.
We generalize this uniqueness result to abstract physical models and a more general class of prescriptive and causal routing operators.
Similar to the last subsection, we use an extension-based approach.

\revised{
Similar as for the extension-based \emph{existence} result in \Cref{lem:existence-using-ext-prop}, we here define the \emph{unique} extension property, which states that coherent extensions of a coherent flow with finite horizon are unique on some positive (but arbitrarily small) extension interval.
}

\begin{definition}
	A pair of an edge loading operator $\EdgeLoading$ and a routing operator $\ROp$ fulfils the \emph{\uniqueExtProp}, if for any two flows $f$ and $g$ that coincide until some time $\horizon$ and that are coherent until some later time $\horizon' > \horizon$,  there exists an $\alpha>0$ such that $f$ and $g$ coincide until time $\horizon + \alpha$.
\end{definition}

\revised{
	Compared to the \extProp{} (\Cref{def:ext-prop}), the \uniqueExtProp{} guarantees uniqueness of coherent extensions as opposed to existence of coherent extensions.
	An elementary proof shows that this property already implies uniqueness of coherent flows until $\horizon$, for any $\horizon\in\IRnnInf$:
}

\begin{lemma}\label{lem:unique-ext-prop-implies-uniqueness}
	Assume $\EdgeLoading$ and $\ROp$ fulfil the \uniqueExtProp.
	Then, for every $\horizon\in\IRnnInf$, any two flows that are \coherent{} until time $\horizon$ already coincide until time $\horizon$.
	
	In particular, there is at most one flow that is \coherent{} until $\infty$.
\end{lemma}
\newcommand{\proofLemUniqueExtPropImpliesUniqueness}[1][Proof]{
\begin{proof}[#1]
	Assume to the contrary that there are two flows $f$ and $g$ which are both coherent up to some time $\horizon'$ but do not coincide up to that time.
	We then have
	\[
		\horizon \coloneqq \inf\set{\theta \leq \horizon' | \CharF_{[0,\theta)}\cdot f \neq \CharF_{[0,\theta)} \cdot g} < \horizon'.
	\]
	Note that, we have $\horizon \geq 0$ as all flows trivially coincide until time~$0$. Now, $f$ and $g$ coincide until $\horizon$ and are coherent until $\horizon' > \horizon$.
	Hence, the \uniqueExtProp\ guarantees the existence of some $\alpha > 0$ such that $f$ and $g$ coincide until $\horizon+\alpha$. This, however, is then a contradiction to the choice of~$\horizon$.
\end{proof}
}
\proofLemUniqueExtPropImpliesUniqueness

If the operators are \emph{strictly} causal \revised{(\Cref{def:causal-edge-loading,def:causal-routing-operator}), \ie they only depend on the past evolution of the flow separated from the present by a positive time period,} and \revMinor{if} $\ROp$ is prescriptive, the \uniqueExtProp{} can be shown \revMinor{directly} without requiring a fixed-point theorem:
\newcommand{\lemmaUniqueExtensionIfStrictlyCausal}{
\begin{lemma}\label{lem:unique-extension-if-strictly-causal}
	Assume that $\EdgeLoading$ and $\ROp$ are strictly causal and that $\ROp$ is prescriptive. Then $\EdgeLoading$ and $\ROp$ fulfil the \uniqueExtProp.

	In particular, if, additionally, $\EdgeLoading$ is locally bounded, then there exists a unique flow that is \coherent{} until time $\infty$. 
\end{lemma}
}\lemmaUniqueExtensionIfStrictlyCausal \newcommand{\proofLemUniqueExtensionIfStrictlyCausal}[1][Proof]{
\begin{proof}[#1]
	Let $\horizon\in\IRp$ and let $f$ and $g$ be two flows coinciding until $\horizon$ and that are coherent until some later time $\horizon' > \horizon$.
	Let $\alpha$ be the minimum of the two values for $\alpha$ obtained by applying strict causality of $\EdgeLoading$ and $\ROp$.
	This implies $\CharF_{[0,\horizon+\alpha]}\cdot r(f) = \CharF_{[0,\horizon+\alpha]} \cdot r(g)$ and $\CharF_{[0,\horizon+\alpha]}\cdot \EdgeLoading(f) = \CharF_{[0,\horizon+\alpha]}\cdot \EdgeLoading(g)$.
	Then we have $g^-(\theta) = f^-(\theta)$ and, consequently, 
	\begin{align*}
		f_{e,i}^+(\theta)
		&= \ROe_{e,i}(f)(\theta) \cdot \CharF_{v\neq \dest_i} \cdot f_{v,i}^+(\theta)
		 = \ROe_{e,i}(g)(\theta) \cdot \CharF_{v\neq \dest_i} \cdot g_{v,i}^+(\theta)
		 = g_{e,i}^+(\theta)
	\end{align*}
	for almost all $\theta < \min\{\horizon+\alpha,\horizon'\}$ and all $e\in E$, $i\in I$.
	Thus, $f$ and $g$ coincide until~$\min\{\horizon+\alpha,\horizon'\} > \horizon$.
	
	{} The additional part of the \namecref{lem:unique-extension-if-strictly-causal} follows immediately by \Cref{lem:unique-ext-prop-implies-uniqueness} together with \Cref{prop:existence-if-strictly-causal}.
\end{proof}
}\proofLemUniqueExtensionIfStrictlyCausal If, however, our operators are only causal (but not strictly causal) we can instead require a stronger continuity property on the mapping $\Contr$ that maps (for some extension interval) inflow rates~$g^+$ to new inflow rates that are induced by the flow $(g^+, \EdgeLoading(g^+))$ and the routing operator~$\ROp$.
This allows us to apply the Banach Fixed-Point Theorem (\cite[Theorem~3.48]{InfiniteDimensionalAnalysis}) which guarantees the existence of a \emph{unique} fixed point.
Afterwards we will show that $\Contr$ satisfies this property whenever both the edge loading and routing operator satisfy a certain Lipschitz-continuity on small extension intervals.

\newcommand{\banachFixedPointTheorem}{
\begin{theorem}[Banach Fixed-Point Theorem]\label{thm:BanachFixedPoint}
	A map $\Contr: X \to X$ on a non-empty, complete metric space $(X,\mathrm d)$ has a unique fixed point $x^*$ if $\Contr$ is a \emph{contraction}, which means that there exists some $\kappa<1$ such that $\dist{\Contr(x)}{\Contr(y)} \leq \kappa\cdot \dist{x}{y}$ for all $x,y\in X$.
\end{theorem}
}
\banachFixedPointTheorem

\newcommand{\LemmaUniqueExtIFContraction}{
\begin{lemma}\label{thm:uniqueExtIfContraction}
	For a pair $(\EdgeLoading, \ROp)$ of a causal edge loading operator and a prescriptive and causal routing operator,
	we introduce the following so-called \emph{contraction property}~\ref*{contraction-prop}:
	{
	\begin{enumerate}[label=(C)]
		\item\label{contraction-prop} For any \coherent{} flow $f$ until $\horizon$, there exists an $\alpha> 0$ such that the following map is a contraction \revised{\wrt $\normFct_p$}:
	\end{enumerate}}\begin{align*}
		\Contr_{\horizon,\alpha}^f :
		\Omega^+_{\horizon,\alpha}(f^+) \to \Omega^+_{\horizon,\alpha}(f^+),
		~
		\Contr_{\horizon,\alpha}^f(g^+)_{e,i} \coloneqq \CharF_{[0,\horizon+\alpha]}\cdot \ROe_{e,i}(g^+, \EdgeLoading(g^+)) \cdot (u_{v,i} + \sum_{e'\in\edgesTo{v}} \EdgeLoading(g^+)_{e',i}),
	\end{align*}
	\revised{where $\Omega^+_{\horizon,\alpha}(f^+)$, as defined in~\eqref{eq:extension-inflows}, is the set of extensions of $f^+$ until $T+\alpha$.}

	If \ref{contraction-prop} is fulfilled, then $\EdgeLoading$ and $\ROp$ satisfy both the \uniqueExtProp{} and the \extProp.
	In particular, if $\EdgeLoading$ is additionally locally bounded, there exists a unique flow that is \coherent{} until time $\infty$.
	{}
\end{lemma}
}
\LemmaUniqueExtIFContraction
\newcommand{\proofThmUniqueExtIfContraction}[1][Proof]{
\begin{proof}[#1]
	Note that the mapping $\Contr_{\horizon,\alpha}^f$ is well-defined for every $\alpha > 0$ and flow $f$ that is coherent until time $\horizon$ due to the causality of $\EdgeLoading$ and $\ROp$.
	Therefore, for any such $f$, there exists an $\alpha$ such that the contraction $\Contr_{\horizon,\alpha}^f$ admits a unique fixed point by the Banach Fixed-Point Theorem.

	We first show the \uniqueExtProp.
	Let $f$ and $g$ be two flows that coincide until time $\horizon$ and that are coherent until $\horizon' > \horizon$.
	Let $\alpha\in\IRp$ be given by the contraction property~\revMinor{\ref*{contraction-prop}} (\wrt $f$ and $\horizon$).
	Then, for any $\alpha'\in(0,\alpha]$, the mapping $\Contr_{\horizon,\alpha'}^f$ is a contraction \revMinor{as well}.
	Thus, we may assume $\alpha \leq \horizon' - \horizon$.
	Clearly, both $\CharF_{[0,\horizon+\alpha]}\cdot g^+$ and $\CharF_{[0,\horizon+\alpha]}\cdot f^+$ are in $\Omega^+_{\horizon,\alpha}(f^+)$ and, as both $f$ and $g$ are coherent until $\horizon+\alpha$, they are both fixed points of $\Contr_{\horizon,\alpha}^f$.
	Since the fixed point is unique, $f$ and $g$ must coincide until time $\horizon + \alpha$.

	We now show the \extProp.
	Let $\horizon\in\IRnn$ and let $f$ be a flow that is \coherent{} until time $\horizon$.
	Let $\alpha$ be given as described in the contraction property and let $g^+$ be a fixed point of $\Contr$.
	Then, $g\coloneqq (g^+, \EdgeLoading(g^+))$ is \coherent{} until $\horizon + \alpha$ and coincides with $f$ until time $\horizon$ by the causality of $\EdgeLoading$ and $\ROp$. 
	
	The additional part of the \namecref{thm:uniqueExtIfContraction} now follows again immediately using \Cref{lem:unique-ext-prop-implies-uniqueness,prop:existence-if-strictly-causal}.{}
\end{proof}
}
	\proofThmUniqueExtIfContraction

\begin{theorem}\label{thm:unique-existence-for-lipschitz-operators}
	Let $\EdgeLoading$ be a causal edge loading operator and let $\ROp$ be a causal, prescriptive routing operator, and let all network inflow rates $u_{v,i}$ be contained in $\pLocInt[\infty][\IR]$ for $v\in V, i\in I$.
	\revMinor{If, for every time $\horizon \in \IRnn$ and every flow $f$ which is coherent until~$\horizon$,} there exists \revMinor{an} $\alpha > 0$ such that
	\begin{enumerate}[label=(\roman*)]
		\item\label{cond:EdgeLoading-essentially-bounded} the map $\restr{\EdgeLoading_{\horizon+\alpha}}{\Omega^+_{\horizon,\alpha}(f^+)}$ is essentially bounded, \ie $\sup_{g^+\in\Omega^+_{\horizon,\alpha}(f^+)} \norm{\EdgeLoading_{\horizon+\alpha}(g^+)}_\infty <\infty$, and
		\item\label{cond:EdgeLoading-Lipschitz} the map $\restr{\EdgeLoading_{\horizon+\alpha}}{\Omega^+_{\horizon,\alpha}(f^+)}$ is Lipschitz continuous \revMinor{\wrt $\normFct_p$}, and
		\item\label{cond:RoutingOperator-Lipschitz} the map $\restr{\ROe_{\horizon+\alpha}}{\Omega_{\horizon,\alpha}(f)}$ is Lipschitz continuous \wrt $(\Omega_{\horizon,\alpha}(f), \normFct_1)$ and $(\splitFcts_{\horizon+\alpha}^{E\times I}, \normFct_\revised{1})$,
	\end{enumerate}
	then $\EdgeLoading$ and $\ROp$ satisfy the contraction property \ref{contraction-prop}.

	In particular, if additionally $\EdgeLoading$ is locally bounded, there exists a unique flow that is \coherent{} until time $\infty$.
\end{theorem}

Note that the conditions \ref{cond:EdgeLoading-Lipschitz} and \ref{cond:RoutingOperator-Lipschitz} of \Cref{thm:unique-existence-for-lipschitz-operators} are trivially fulfilled if $\EdgeLoading$ and $\ROp$, respectively, are uniformly strictly causal since, for small enough~$\alpha$, the mappings $\restr{\EdgeLoading_{\horizon+\alpha}}{\Omega^+_{\horizon,\alpha}(f^+)}$ and $\restr{\ROe_{\horizon+\alpha}}{\Omega_{\horizon,\alpha}(f)}$, respectively, are then constant \revMinor{on the restricted domains}.

\newcommand{\proofThmUniqueExistenceForLipschitzOperators}[1][Proof]{
\begin{proof}[#1]
	\revMinor{Given a coherent flow~$f$ until~$T$, l}et $\alpha$ be given as in the \revMinor{\namecref{thm:unique-existence-for-lipschitz-operators}} assumptions and choose $M\in\IRp$ such that $M$ is an upper bound of $\smallnorm{\CharF_{[0,\horizon+\alpha]}\cdot (u_{v,i} + \sum_{e \in \edgesFrom{v}}\EdgeLoading(g^+)_{e,i})}_\infty$ for all $g^+\in \Omega^+_{\horizon+\alpha}(f)$ and $v \in V$, $i \in I$.
	Let $L_1$ and $L_2$ be Lipschitz constants of $\restr{\EdgeLoading_{\horizon+\alpha}}{\Omega^+_{\horizon,\alpha}(f^+)}$ and $\restr{\ROe_{\horizon+\alpha}}{\Omega_{\horizon,\alpha}(f)}$, respectively.
	Without loss of generality, assume $L_1 \geq 1$, $L_2 > 0$.

	We prove that $\Contr\coloneqq \Contr^f_{\horizon,\tilde\alpha}$ is a contraction \wrt any fixed $\tilde\alpha\leq\alpha$ satisfying $\tilde\alpha^q < (M\cdot L_1\cdot L_2^\revised{1/p})^{-1}$ where $q$ is the conjugate of $p$ with $\nicefrac{1}{p}+\nicefrac{1}{q}=1$.
	Let $h^+$ and $g^+$ be arbitrary in $\Omega^+_{\horizon+\tilde\alpha}(f)$.
	We deduce
	\begin{align*}
		\norm{\Contr_{e,i}(h^+) - \Contr_{e,i}(g^+)}_p
		&\leq M\cdot \norm{
			\CharF_{[0,\horizon+\tilde\alpha]}\cdot (
				r_{e,i}(h^+,\EdgeLoading(h^+)) - r_{e,i}(g^+, \EdgeLoading(g^+))
			)
		}_p
		\\
		& = M\cdot \lVert
			\ROe_{\horizon+\tilde\alpha}(h^+, \EdgeLoading(h^+))_{e,i}
			- \ROe_{\horizon+\tilde\alpha}(g^+, \EdgeLoading(g^+))_{e,i}
		\rVert_p
		\\
		& \revised{\leq M\cdot \lVert
			\ROe_{\horizon+\tilde\alpha}(h^+, \EdgeLoading(h^+))_{e,i}
			- \ROe_{\horizon+\tilde\alpha}(g^+, \EdgeLoading(g^+))_{e,i}
		\rVert_1^{1/p},}
	\end{align*}
	\revised{where the last inequality uses that the range of $\ROe_{T+\tilde\alpha}(g)$ is contained in $[0,1]$ for all $g$ (\cf \Cref{prop:p-integrable-bounded-codomain}).}
	Applying the Lipschitz continuity of $\restr{\ROe_{\horizon+\alpha}}{\Omega_{\horizon,\alpha}(f)}$ yields
	\begin{align*}
		\norm{\Contr_{e,i}(h^+) - \Contr_{e,i}(g^+)}_p
		&\leq
		M\cdot L_2^\revised{1/p} \cdot\max( \norm{h^+ - g^+}_1, \norm{\EdgeLoading_{\horizon+\tilde\alpha}(h^+) - \EdgeLoading_{\horizon+\tilde\alpha}(g^+)}_1 ).
	\end{align*}
	As $h$ and $g$ coincide until time $\horizon$ and as $\EdgeLoading$ is causal, it suffices to only consider the interval $[\horizon,\horizon+\tilde\alpha]$ for the norms.
	By applying Hölder's inequality, we obtain
	\begin{align*}
		\norm{\Contr_{e,i}(h^+) - \Contr_{e,i}(g^+)}_p
		& \leq 
		M\cdot L_2^\revised{1/p} \cdot \tilde\alpha^q \cdot \max(
			\norm{h^+ - g^+}_p,
			\norm{\EdgeLoading_{\horizon+\tilde\alpha}(h^+) - \EdgeLoading_{\horizon+\tilde\alpha}(g^+)}_p
		).
	\end{align*}
	Finally, with the Lipschitz-continuity of $\EdgeLoading_{\horizon+\tilde\alpha}$ (and $L_1\geq 1$) we obtain
	\begin{equation*}
		\norm{\Contr_{e,i}(h^+) - \Contr_{e,i}(g^+)}_p
		 \leq  
		M\cdot L_1 \cdot L_2^\revised{1/p} \cdot \tilde\alpha^q \cdot \norm{h^+ - g^+}_p
	\end{equation*}
	Since we have $M\cdot L_1\cdot L_2^\revised{1/p}\cdot \tilde\alpha^q < 1$, this shows that $\Contr$ is a contraction on $\Omega^+_{\horizon,\alpha}(f^+)$.
	
	The additional part of the \namecref{thm:unique-existence-for-lipschitz-operators} follows directly from \Cref{thm:uniqueExtIfContraction}.{}
\end{proof}
}
\proofThmUniqueExistenceForLipschitzOperators[Proof of \Cref{thm:unique-existence-for-lipschitz-operators}]

\section{Applications}\label{sec:Applications}

In this section, we illustrate how our general existence and uniqueness results can be applied to various concrete models from the literature.
We first present two widely used physical models in \Cref{sec:physical-model-examples} and then discuss in \Cref{sec:behavioral-model-examples} how several well-known behavioral models from the literature fit into the framework of routing operators.
We show how our results can be used to obtain existence and uniqueness results for various classes of dynamic flows.

\subsection{Physical Models}\label{sec:physical-model-examples}

The two physical models, which we introduce as examples here, are \revMinor{Vickrey's queuing} model and the affine-linear volume delays.
As we will show, these two physical models fulfill all the required properties necessary to apply our results.

\subsubsection{Vickrey's Queuing Model}\label{sec:vickrey}

In \emph{Vickrey's queuing model}, as proposed by \textcite{VickreyQueuingModel}, sometimes also called the \emph{deterministic fluid queuing model},
every edge $e\in E$ is assigned a \emph{capacity} $\capa_e\in\IRp$ as well as a \emph{free-flow travel time} $\ffttime_e\in\IRnn$.
Whenever the inflow into an edge exceeds the capacity, a point queue builds up in front of the edge.
Particles entering an edge must first wait in this queue, which is processed at rate $\capa_e$, before they can traverse the edge which then takes another $\ffttime_e$ time units (see \Cref{fig:VickreyQueue} for a graphical depiction).

\begin{figure}[ht]
	\centering
	\begin{tikzpicture}
	\node[vertex](v)at(0,0){$v$};
	\node[vertex](w)at(4,0){$w$};

\draw[flowCol,line width=10pt]($(v)+(.55,.1)$)--++(0,1);
	\draw[flowCol,line width=10pt]($(v)+(.8,0)$)--(w);
	\draw[thick]($(v)+(.8,5.4pt)$) -- node[above]{$e=vw$} ($(w.west)+(0,5.4pt)$) ($(v)+(.8,-5.4pt)$) -- ($(w.west)+(0,-5.4pt)$);
	\draw[edge] (v)--(w);
	\draw[<->]($(v)!.5!(w)+(0,-5pt)$) --node[right=2pt,fill=flowCol,inner sep=1pt]{$\nu_e$} ++(0,10pt);

\draw[<->]($(v)+(.8,-10pt)$) --node[below]{$c_e^0$} ($(w.west)+(0,-10pt)$);
	\draw[<->]($(v)+(.55,.1)$)-- node[left=2pt]{$z_e$}++(0,1);

\draw[flowCol,line width=16pt]($(v)+(-.8,0)$)--(v);
	\draw[thick]($(v)+(-.8,8.4pt)$) -- ($(v.west)+(0,8.4pt)$) ($(v)+(-.8,-8.4pt)$) -- ($(v.west)+(0,-8.4pt)$);
	\draw[edge]($(v)+(-.8,0)$)--(v);
	\draw[thick,dotted]($(v)+(-.8,8.4pt)$) -- ++(-.45,0);
	\draw[thick,dotted]($(v)+(-.8,-8.4pt)$) -- ++(-.45,0);
	\draw[thick,dotted]($(v)+(-.8,0)$) -- ++(-.45,0);

\draw[flowCol,line width=10pt](w)--++(.8,0);
	\draw[thick]($(w.east)+(0,8.4pt)$) -- ($(w)+(.8,8.4pt)$) ($(w.east)+(0,-8.4pt)$) -- ($(w)+(.8,-8.4pt)$);
	\draw[thick](w)--++(.8,0);
	\draw[thick,dotted]($(w)+(.8,8.4pt)$)--++(.45,0);
	\draw[thick,dotted]($(w)+(.8,-8.4pt)$)--++(.45,0);
	\draw[edge,dotted]($(w)+(.8,0)$)--++(.45,0);
\end{tikzpicture}
 	\caption{A graphical depiction of an edge $e=vw$ with a Vickrey queue of length $z_e$ at its start.}\label{fig:VickreyQueue}
\end{figure}

More formally, for a flow $f$, the \emph{queue length} $\qLen_e$ at edge $e$ is defined as
 \[
	\qLen_e(\theta, f)\coloneqq \int_0^\theta f^+_e \diff\leb - \int_0^{\theta + \ffttime_e}  f^-_e \diff\leb,
\]
where $f_e^+\coloneqq \sum_{i\in I} f_{e,i}^+$ and $f_e^-\coloneqq \sum_{i\in I} f_{e,i}^-$ denote the aggregate in- and outflow rate, respectively.
The travel time for a particle entering at time $\theta$ is then given by \[
	\ttime_e(\theta, f) \coloneqq \frac{\qLen_e(\theta, f)}{\capa_e} + \ffttime_e,
\]
and its exit time is $\exitTime_e(\theta, f)\coloneqq \theta + \ttime_e(\theta, f)$.

A flow $f$ is called a \emph{Vickrey flow} until time $\horizon$ if it fulfils the following two conditions:
First, its queue \emph{operates at capacity} until $\horizon$, \ie for almost all $\theta < \horizon$ we have \[
	f_e^-(\theta + \ffttime_e) = \begin{cases}
		\capa_e, &\text{if $\qLen_e(\theta, f) > 0$}, \\
		\min\{ f_e^+(\theta), \capa_e \}, &\text{otherwise}.
	\end{cases}
\]
Second, its queue \emph{operates fairly} until $\horizon$, \ie for almost all $\theta < \horizon$ we have \[
	f_{e,i}^-(\theta + \ffttime_e) = \begin{cases}
		f_e^-(\theta + \ffttime_e) \cdot \frac{f_{e,i}^+(\xi_\theta)}{f_e^+(\xi_\theta)}, &\text{if $f_e^+(\xi_\theta) > 0$}, \\
		0, &\text{otherwise},
	\end{cases}
\]
where $\xi_\theta$ is chosen such that $\xi_\theta + c_e(\xi_\theta) = \theta + \ffttime_e$.

It is a well-known result that for any given set of edge inflow rates $f^+\in\rateFcts^{E\times I}$ there exists a unique set of outflow rates $f^-\in\rateFcts^{E\times I}$ such that $(f^+, f^-)$ is a Vickrey flow, and moreover, this mapping is causal (\cf \citealt[Corollaries~3.43~and~3.46]{GrafThesis}):
\begin{lemma}\label{lem:Vickrey-properties}
	For any inflow rates vector $f^+\in\rateFcts^{E\times I}$, there exists a unique vector $f^-\in\rateFcts^{E\times I}$ such that $(f^+, f^-)$ is a Vickrey flow.
	Then, \revMinor{the resulting edge loading operator} $\EdgeLoading$ is causal and, if all free-flow travel times are strictly positive, even uniformly strictly causal.
\end{lemma}

Clearly, due to the capacity constraints, $\EdgeLoading$ is also locally bounded.
Moreover, it is a well-known result that $\EdgeLoading_T$ is sequentially weak-weak continuous for every $\horizon\in\IRnn$ even if we allow $\ffttime_e=0$.

\begin{lemma}[{\citealt[Corollaries~3.45,~3.46]{GrafThesis}}]
	For the edge loading operator $\EdgeLoading$ of Vickrey's model, $\EdgeLoading_\horizon$ is sequentially weak-weak continuous \revMinor{\wrt $\normFct_p$}.
	Furthermore, the induced travel time function $f\mapsto \ttime_e(\emptyarg, f)$ is sequentially weak-strong continuous \wrt $(\flowSet_\horizon,\revMinor{\normFct_p})$ and $(C([0,\horizon],\IRnn), \revMinor{\normFct_\infty})$ for arbitrary $\horizon\in\IRnn$.
\end{lemma}

\subsubsection{Affine-Linear Volume-Delay Dynamics}

A closely related and also frequently used physical model is that of \emph{affine-linear volume-based edge dynamics}.
This model is studied in more detail, e.g., by \textcite{ZhuM00} and also used by \textcite{Bayen2019}.
Here, the travel time on a link depends affine-linearly on the flow volume on the edge at the time the edge is entered.

More specifically, for each edge $e\in E$, we are given a positive free-flow travel time $\ffttime_e \in\IRp$ and a positive capacity $\capa_e\in\IRp$.
The travel time on $e$ at time $\theta$ is then given by \begin{equation}\label{eq:affine-linear-volume-delay}
	\ttime_e(\theta, f) \coloneqq \ffttime_e + \frac{X_e(\theta, f)}{\capa_e},
\end{equation}
where $\flowVol_e(\theta, f)\coloneqq \int_0^\theta f_e^+ \diff\leb -  \int_0^\theta f_e^-\diff\leb$ is the flow volume on edge $e$ at time $\theta$.
The exit time when entering $e$ at time $\theta$ is again given by $\exitTime_e(\theta, f)\coloneqq \theta + \ttime_e(\theta, f)$.
Note that, for any flow $f$, both $\ttime_e(\emptyarg, f)$ and $\exitTime_e(\emptyarg, f)$ are absolutely continuous and thus differentiable almost everywhere.

\begin{definition}\label{def:respects-affine}
	We say that a flow $f$ \emph{respects affine-linear travel times} until time~$\horizon$ if 
	\[
		\int_0^\theta f_{e,i}^+ \diff\leb = \int_0^{\exitTime_e(\theta,f)} f_{e,i}^-\diff\leb
	\]
	holds for all $\theta < \horizon$ with $\exitTime_e(\theta, f)\coloneqq \theta+\ttime_e(\theta, f)$ where $\ttime_e$ is given by~\eqref{eq:affine-linear-volume-delay}.
\end{definition}

\begin{remark}
	Note that for \revMinor{a flow $f$ that respects affine-linear travel times} and an edge $e$, the exit time function $\exitTime_e(\emptyarg,f)$ is strictly increasing on $(-\infty, \horizon)$ and therefore the inverse function $\exitTime_e(\emptyarg, f)^{-1}$ of its restriction $\exitTime_e(\emptyarg, f): (-\infty, \horizon) \to (-\infty, \exitTime_e(\horizon, f))$ exists, and both functions are differentiable with positive derivative almost everywhere on their respective domains.
	Further, a flow $f$ respects travel times until $\horizon$ if and only if $f_{e,i}^+ (\theta) = f_{e,i}^-(\exitTime_e(\theta, f))\cdot (\exitTime_e(\emptyarg, f))'(\theta)$ holds for almost all $\theta< \horizon$.
	This means that $f$ respects travel times until $\horizon$ if and only if $\exitTime_e(\emptyarg, f)$ is strictly increasing on $(-\infty, \horizon)$ and \begin{align*}
		f_{e,i}^-(\theta)
		&
		= f_{e,i}^+(\exitTime_e(\emptyarg, f)^{-1}(\theta)) \cdot (\exitTime_e(\emptyarg, f)^{-1})'(\theta)
	\end{align*}
	holds for all $\theta < \exitTime_e(\horizon, f)$.

	Therefore, given some edge inflow rate vector $f^+\in\rateFcts^{E\times I}$, the set of outflow rate vectors $f^-\in\rateFcts^{E\times I}$ such that $f=(f^+, f^-)$  respects these travel times (until $\infty$) correspond one-to-one to the function vectors $(X_{e,i})_{e,i}: \IR \to\IR^{E\times I}$ for which $\exitTime_e[X] : \theta \mapsto \theta + \ffttime_e + \frac{1}{\capa_e}\cdot \sum_{i\in I} X_{e,i}(\theta)$ is strictly increasing and which solve the following system of delay differential equations (DDE) for all $e\in E$, $i\in I$ and almost all $\theta\in\IR$:
	\begin{align*}
		\flowVol_{e,i}'(\theta) &= f_{e,i}^+(\theta) - f_{e,i}^-[\flowVol](\theta) \\
		f_{e,i}^-[\flowVol](\theta) &\coloneqq f_{e,i}^+( \exitTime_e[\flowVol]^{-1}(\theta) ) \cdot  (\exitTime_e[\flowVol]^{-1})'( \theta ) \\
		\flowVol_{e,i}(0)           &= 0
	\end{align*}

	This system of DDE can be solved in a stepwise fashion (as was done in~\citealt[Theorem~2.9]{Bayen2019}, and \citealt[Theorem~2.2]{ZhuM00}) which yields the following theorem.
\end{remark}

\begin{theorem}\label{thm:linear-edge-delay-well-defined}
	For each edge inflow rate vector $f^+\in\rateFcts^{E\times I}$, there exists a unique outflow rate vector $f^-\in\rateFcts^{E\times I}$ such that the flow $(f^+, f^-)$ respects \revMinor{affine-linear} travel times \revMinor{(\Cref{def:respects-affine})}.
	Hence, the \revMinor{edge loading operator} $\EdgeLoading$ that maps inflow rates to such outflow rates is well-defined.
\end{theorem}

By analysing the proof of \Cref{thm:linear-edge-delay-well-defined}, we can extract the following causality property:
\begin{proposition}\label{prop:uniqueness-for-continuity}
	Let $f_{e,\emptyarg}=(f^+_{e,\emptyarg}, f^-_{e,\emptyarg})$ and $g_{e,\emptyarg}=(g^+_{e,\emptyarg},g^-_{e,\emptyarg})$ be two vectors in $\rateFcts^{I\times\{+,-\}}$ that respect affine-linear travel times until time $\horizon$ \revMinor{(\Cref{def:respects-affine})}.
	If $f^+_{e,\emptyarg}$ and $g^+_{e,\emptyarg}$ coincide until $\horizon$, then $f^-_{e,\emptyarg}$ and $g^-_{e,\emptyarg}$ coincide until $\tau_e(\horizon, f) = \tau_e(\horizon, g) \geq \horizon + \ffttime$.

	In particular, $\EdgeLoading$ is uniformly strictly causal.
\end{proposition}

We continue by showing that $\EdgeLoading_\horizon$ is sequentially weak-weak continuous \revMinor{\wrt{} $\normFct_p$}.
Note that a similar continuity result was presented in \cite[Theorem~3.2]{ZhuM00} where the edge inflow rates are assumed to be essentially bounded and instead of weak convergence, a convergence in an integral sense is used.

\begin{theorem}
	Let $\EdgeLoading$ denote the edge loading operator \revMinor{from \Cref{thm:linear-edge-delay-well-defined}} induced by a set of affine-linear volume-delay functions.
	Then, for every $\horizon\in\IRnn$, $\EdgeLoading_\horizon$ is sequentially weak-weak continuous \revMinor{\wrt $\normFct_p$}.
	Furthermore, the map $f\mapsto \ttime_e(\emptyarg, f)$ is sequentially weak-strong continuous \wrt $(\flowSet_\horizon, \revMinor{\normFct_p})$ and $(C([0,\horizon],\IRnn), \revMinor{\normFct_\infty})$.
\end{theorem}
\begin{proof}
	We show the sequential weak-weak continuity for $(f^+_{e,\emptyarg})\mapsto (\EdgeLoading_\horizon(f^+))_{e,\emptyarg}$ for each edge $e\in E$ separately.
	To reduce noise, we will omit the edge index $e$ for all relevant functions and constants.

	Let $(f^{n,+})_n$ be a sequence of vectors in $\rateFcts_\horizon^I$ that converges weakly to $f^+$.
	We show that $f^{n,-}\coloneqq\EdgeLoading_\horizon(f^{n,+})$ converges weakly to $f^-\coloneqq \EdgeLoading_\horizon(f^+)$.
	Note that all $f^{n,-}$ are contained in the weakly compact set $\set{ g\in\rateFcts_\horizon^{I} | g_{i} \leq \capa_e}$.
	Thus, the sequence $f^{n,-}$ has a weakly convergent subsequence, and we pass to this subsequence with weak limit $f^{*,-}$.
	Let $f^n$ denote the flow $(f^{n,+}, f^{n,-})$ and let $f^*$ denote the flow $(f^+, f^{*,-})$.
	
	Because the operator $J: L^p([0,\horizon]) \to C([0,\horizon]), f\mapsto (\theta\mapsto\int_0^\theta f\diff\leb)$ is compact, the cumulative flow rate functions $F_i^{n,+}: \theta\mapsto \int_0^\theta f_i^{n,+}\diff\leb$ and $F_i^{n,-}: \theta\mapsto \int_0^\theta f_i^{n,-}\diff\leb$ converge uniformly as functions in $C(\IR)$.
	Therefore, also the sequence of edge volume functions $X_i(\emptyarg, f^n)$ and the sequences of travel time functions $\ttime(\emptyarg, f^n)$ and exit time functions $\exitTime(\emptyarg, f^n)$ converge uniformly to $X(\emptyarg, f^*)$, $\ttime(\emptyarg, f^*)$ and $\exitTime(\emptyarg, f^*)$, respectively, in $C(\IR)$.
	Thus, the composition $F^{n,-}\circ \exitTime(\emptyarg, f^n)$ converges pointwise to $F^{*,-}\circ \exitTime(\emptyarg, f^*)$.
	
	Note that the function $\exitTime(\emptyarg, f^n)$ coincides with $\exitTime(\emptyarg, (f^{n,+}, \EdgeLoading(f^{n,+})))$ on $(-\infty, \horizon]$.
	Therefore, $\exitTime(\emptyarg, f^n)$ is invertible as a function from $(-\infty, \horizon]$ to $(-\infty, \exitTime(\horizon, f^n)]$ and $f^n$ respects travel times until time $\exitTime(\emptyarg, f^n)^{-1}(\horizon)$.
	We define $\horizon'\coloneqq \min\set{ \theta | \exitTime(\theta,f^*) = \horizon}$ and let $\theta < \horizon'$.
	Then, for large enough $n$, we have $\exitTime(\theta, f^{n}) < \horizon$, and we can deduce $F_{i}^{+,n}(\theta) = F_{i}^{-,n}(\exitTime(\theta, f^n))$.
	In the limit, this shows that $f^*$ respects travel times until $\horizon'$.
	From \Cref{prop:uniqueness-for-continuity} it follows that $f^{*,-}$ and $f^-$ must coincide on $(-\infty, \horizon]$.
\end{proof}

\subsection{Behavioral Models}\label{sec:behavioral-model-examples}

We continue by discussing how known behavioral models from the literature fit into our framework and how our existence and uniqueness results apply to them.
We begin with deterministic prediction equilibria, which subsume several well-known equilibrium concepts such as dynamic Nash equilibria (also known as Nash flows over time) and instantaneous dynamic equilibria.
Then, we discuss the applications to prescriptive routing models such as those introduced by \textcite{Bayen2019}.

\subsubsection{Deterministic Prediction Equilibrium}\label{sec:deterministic-dpe}

In a deterministic prediction equilibrium, particles adaptively choose their path by minimizing the predicted travel cost.
Here, all particles of a commodity are assumed to use a common travel cost predictor which may depend on the (past and/or future) evolution of the flow.
We reintroduce the model proposed in~\cite{PredictionEquilibria}, and show that it fits into the concept of routing operators.

For each pair of nodes $v,w\in V$, we denote the finite set of simple $v$-$w$-paths by $\PathSet_{v,w}$.
Each commodity is associated with a \emph{cost predictor} $\pCost_{i,\Path}: \IR \times \flowSet \to \IRnn$ for each simple path $p$ where $\pCost_{i,\Path}(\theta, f)$ is the predicted cost of traversing path $\Path$ when entering the path at time~$\theta$ as predicted at the same time $\theta$ using the flow $f$.
We assume that $\pCost_{i,\Path}(\emptyarg, f)$ is measurable for any cost predictor $\pCost_{i,\Path}$ and flow~$f$.

For a commodity $i\in I$, an edge $e=vw$ with $v\neq \dest_i$ is called \emph{active at time $\theta$} for commodity~$i$, if $e$ is the first edge of an optimal path in $\PathSet_{v,\dest_i}$ when starting at $v$ at time $\theta$ as predicted at the same time $\theta$ according to $\pCost_i$; in formulas:
\[
    \exists\, \Path\in\PathSet_{v,\dest_i} : \quad
	\Path_1 = e
    \,\land\,
    \Path\in\argmin_{\PathAlt\in\PathSet_{v,\dest_i}} \pCost_{i,\PathAlt}(\theta, f).
\]
We denote the set consisting of all active edges of commodity $i$ at time $\theta$ by $\pEdges_i(\theta,f)${} and the set of active times of an edge~$e$ and a commodity~$i$ by $\pActive_{e,i}(f)$.
\revised{Note that since we assumed in \Cref{sec:Model} that $\dest_i$ is reachable from every node, every node $v\neq \dest_i$ has at least one active outgoing edge at any time, \ie $\pEdges_i(\theta, f)\cap \edgesFrom{v}$ is non-empty.}
\begin{definition}
	A flow $f$ \emph{only uses active edges \wrt $\hat C$} until time $\horizon\in\IRnnInf$ if for all $e\in E$, $i\in I$ and almost all $\theta < \horizon$ it holds that \[
		f^+_{e,i}(\theta) > 0
		\implies
		e\in\pEdges_i(\theta,f).
	\]
\end{definition}

\begin{definition}[DPE]\label{def:dpe}
	A flow $f$ is a \emph{dynamic prediction equilibrium} (DPE) until time $\horizon$ if $f$ is consistent with $\EdgeLoading$, fulfils flow conservation, and only uses active edges \wrt $\hat C$ until time~$\horizon$.
\end{definition}

The cost predictors $\hat C_{i,\Path}$ are often based on the (actual) travel time functions $\ttime_e: \IR \times \flowSet \to \IRnn$ on the edges $e\in E${} of the network, where $\ttime_e(\theta,f)$ is the travel time of edge $e$ as induced by the flow~$f$ when entering~$e$ at time~$\theta$.
Based on these, we can define $\exitTime_e(\theta, f)\coloneqq \theta + \ttime_e(\theta, f)$ as the exit time when entering edge $e$ at time $\theta$ given flow $f$.
Usually, the travel time functions $\ttime_e$ are induced by the underlying physical model. 

As described by \textcite{PredictionEquilibria}, several well-studied equilibrium concepts fall into the class of \DPEs{}\iftrue{} by choosing appropriate cost predictors\fi.
These concepts include \emph{dynamic (Nash) equilibria} (DE) and \emph{instantaneous dynamic equilibria} (IDE).

\begin{example}
	A \emph{dynamic Nash equilibrium}  (DE) models the situation in which particles have full access to (future) information, which here means that they may use the future evolution of $\ttime_e$ to minimize their \emph{actual} travel time.
	This is achieved by the use of the so-called \emph{perfect predictor}, \ie
	\[
		\hat C_p(\theta, f) \coloneqq \exitTime_p(\theta, f) - \theta = (\exitTime_{e_k}(\emptyarg, f) \circ \,\cdots\, \circ \exitTime_{e_1}(\emptyarg, f))(\theta) - \theta 
	\]
	which predicts the travel time when entering path $p=(e_1, \ldots, e_k)$ at time $\theta$ exactly as it will constitute \wrt the flow $f$.
\end{example}

\begin{example}\label{ex:ide}
	An \emph{instantaneous dynamic equilibrium} (IDE) reflects the situation in which particles only have instantaneous information on the current travel times on every edge{} (at the time the particle takes any decision).
	Here, the particles choose an edge that minimizes their instantaneous travel time by using the so-called \emph{constant predictor}
	\[
		\hat C_p(\theta , f) \coloneqq   \sum_{e\in p} \ttime_e(\theta, f),
	\]
	where the predictions no longer depend on the future evolution of $f$.
\end{example}

The flexibility of the DPE framework also allows the use of more advanced cost predictors that are based on ML-algorithms, as it is often the case for today's navigation devices.
These ML-based predictors (as well as the constant predictor) are then examples of \emph{causal predictors}, \ie predictors $\pCost$ for which whenever two flows $f$ and $g$ coincide until some time $\horizon$, then $\pCost(\emptyarg, f)$ and $\pCost(\emptyarg, g)$ also coincide until~$\horizon$.

\smallskip

We can characterize the \DPE{} as the coherent flows of a suitably chosen routing operator:

\begin{lemma}\label{lem:routing-operator-prediction-equilibrium}
	Let $f$ be a flow, $\horizon\in\IRnnInf$ and let $\pCost$ be cost predictors.
	Then, $f$ only uses active edges \wrt $\hat C$ until time $\horizon$ if and only if $f$ is consistent until $\horizon$ with the routing operator 
	\begin{equation}\label{eq:RoutingOperatorPredictionEquilibrium}
		\newcommand{\conds}{
			\begin{array}{r l}
				\forall v, i:  & \sum_{e \in \edgesFrom{v}} r_{e,i}=\CharF_{v\neq \dest_i},\\
				\forall e, i, \foraall \theta: & r_{e,i}(\theta) > 0 \implies e\in\pEdges_i(\theta,f)
			\end{array} }
		\ROp(f) \coloneqq \Set{
			r \in \revised{\splitFcts}^{E\times I}
			| \conds
		}.
	\end{equation}
\end{lemma}
Recall, that $\splitFcts$ (\cf \Cref{not:function-spaces}) is the set of measurable functions from $\IR$ to $[0,1]$.
\newcommand{\proofLemRoutingOperatorPredictionEquilibrium}[1][Proof]{
\begin{proof}[#1]
	Assume $f$ only uses active edges up to time $\horizon$.
	For a commodity $i\in I$ and an edge $e=vw$ with $v\neq \dest_i$ we define \[
		r_{e,i} (\theta) \coloneqq \begin{cases}
			\frac{f^+_{e,i}(\theta)}{\sum_{e\revMinor{'}\in\edgesFrom{v}}f^+_{e\revMinor{'},i}(\theta)}, & \text{if } {\displaystyle \sum_{e\revMinor{'}\in\edgesFrom{v}}} f^+_{e\revMinor{'},i}(\theta) > 0 \land \theta < \horizon,\\
			\frac{\CharF_{\pActive_{e,i}(f)}(\theta)}{\smallabs{\edgesFrom{v}\cap \pEdges_{i}(\theta,f)}}, & \text{otherwise},
		\end{cases}
	\]
	and $r_{e,i}(\theta) \coloneqq 0$ for $v=\dest_i$.
	Note that $r_{e,i}$ is a measurable function since $\smallabs{\edgesFrom{v}\cap\pEdges_{i}(\theta,f)} = \sum_{e\revMinor{'}\in\edgesFrom{v}} \CharF_{\pActive_{e\revMinor{'},i}(f)}(\theta)$ holds. 
	Clearly, $r_{e,i}(\theta)$ is in the interval~$[0,1]$ and we have $\sum_{e\in\edgesFrom{v}} r_{e,i} (\theta) = \CharF_{v\neq \dest_i}$ for all $e\in E,v\in V, i\in I$ and almost all $\theta\in\IR$.
	Finally, $r_{e,i}(\theta) > 0$ can only hold if $\theta\in\pActive_{e,i}(f)$ or if $f_{e,i}^+(\theta) > 0$ with $\theta < \horizon$ hold.
	The latter (also) implies $e\in\pEdges_i(\theta,f)$ as $f$ only uses active edges up to $\horizon$.
	In conclusion, this shows that $r\in\ROp(f)$.
	By construction of $r$, we also have $f^+_{e,i}(\theta) = r_{e,i}(\theta)\cdot \sum_{e\revMinor{'}\in\edgesFrom{v}}f^+_{e\revMinor{'},i}(\theta)$ for all $i\in I, e=vw\in E$ and almost all $\theta<\horizon$, and thus $f$ is consistent with $\ROp$ up to time $\horizon$.

	For the other direction, assume that $f$ is consistent with $\ROp$ up to time $\horizon$ and let $r\in\ROp(f)$ be such that $f^+_{e,i}(\theta) = r_{e,i}(\theta)\cdot \sum_{e\revMinor{'}\in\edgesFrom{v}} f^+_{e\revMinor{'},i}(\theta)$ for all $i\in I, e=vw\in E$ and almost all $\theta\in[0,\horizon)$.
	Then, $f_{e,i}^+(\theta)$ can only be positive whenever $r_{e,i}$ is positive, in which case $e\in\pEdges_i(\theta,f)$ holds.
	Therefore, $f$ only uses active edges up to time $\horizon$.
\end{proof}
}\proofLemRoutingOperatorPredictionEquilibrium

We continue by showing that \revised{this routing operator $\ROp$ is decomposable (\Cref{def:decomposable})} under certain conditions.
This allows us to apply the existence results \iftrue from the previous section \fi to \DPEs.

\newcommand{\lemmaAltContinuityPredictionEquilibrium}{
\begin{lemma}\label{lem:alt-continuity-prediction-equilibrium}
	Let $\horizon\in\IRnn$ and let $\hat C_{i,p}$ be a collection of cost predictors.
	If $\hat C_{i,p}(\theta,\emptyarg)$ is sequentially weak-strong continuous \wrt $(\flowSet_\horizon, \revMinor{\normFct_p})$ and $\IR$ for all $i\in I, p\in\PathSet$ and almost all $\theta<\horizon$, then \revised{the routing operator $\ROp_\horizon$ defined in \eqref{eq:RoutingOperatorPredictionEquilibrium} is decomposable}.
\end{lemma}
}
\lemmaAltContinuityPredictionEquilibrium
\newcommand{\proofLemAltContinuityPredictionEquilibrium}[1][Proof]{
\begin{proof}[#1]
	\revised{
	The first equality conditions in~\eqref{eq:RoutingOperatorPredictionEquilibrium} can be transformed into two inequality constraints for each node-commodity pair of the required shape in~\eqref{eq:decomposability} for decomposability (with $H_k$ and $b_k$ constant in $f$, and thus trivially sequentially weak-strong continuous).
	For every $e\in E, i\in I$, the second condition is equivalent to $r_{e,i} \leq \CharF_{[0,\horizon]\cap \pActive_{e,i}(f)}$.
	We show that this right-hand side is pointwise sequentially weak-strong upper-semicontinuous in $f$; since $\CharF_{[0, \horizon]\cap\pActive_{e,i}(f)}$ is uniformly bounded by $1$, any sequence is trivially dominated by $\CharF_{[0, \horizon]}$.
	}

	Let $f^{(n)}$ be a sequence weakly converging to $f$ in $\flowSet_\horizon$ and consider only times~$\theta$ where $\hat C_{i,p}(\theta, f^{(n)})$ converges to $\hat C_{i,p}(\theta, f)$.
	If $\theta\in\hat\Theta_{e,i}(f^{(n)})$ holds for infinitely many $n\in \IN$, then there exists some path $p$ starting with $e=vw$ that minimizes $\hat C_{i,p}(\theta, f^{(n)})$ over all paths in $\PathSet_{v,\dest_i}$ for infinitely many $n\in\IN$.
	Therefore, $\hat C_{i,p}(\theta, f^{(n)})$ converges to $\min_{q\in\PathSet_{e,i}} \hat C_{i,q}(\theta, f)$ and, thus, $\theta\in \hat\Theta_{e,i}(f)$.
	Clearly, this implies \revised{$\limsup_{n\to\infty} \CharF_{[0,\horizon]\cap \hat\Theta_{e,i}(f^{(n)})}(\theta) \leq \CharF_{[0,\horizon]\cap \hat\Theta_{e,i}(f)}(\theta)$}. 
\end{proof}
}\proofLemAltContinuityPredictionEquilibrium

\begin{corollary}\label{cor:prediction-equilibria-existence}
	Let $\EdgeLoading$ be a locally bounded edge loading operator and let $\hat C_{i,p}$ be a set of cost predictors.
	\begin{enumerate}[label=(\roman*)]
		\item\label{cor:prediction-equilibria-existence:finite} If, for some $\horizon\in\IRnn$, $\EdgeLoading_\horizon$ is sequentially weak-weak continuous \revMinor{\wrt $\normFct_p$} and all $f\mapsto \hat C_{i,p}(\theta, f)$ are sequentially weak-strong continuous \wrt $(\flowSet_\horizon, \revMinor{\normFct_p})$ and $\IR$ for almost all $\theta <\horizon$, then there exists a dynamic prediction equilibrium until $\horizon$.

		\item\label{cor:prediction-equilibria-existence:infinite}
		\revMinor{If the assumptions of~\ref{cor:prediction-equilibria-existence:finite} hold for all $\horizon\in\IRnn$} and if $\EdgeLoading$ and all $\hat C_{i,p}$ are causal,
		then there exists a dynamic prediction equilibrium until $\infty$.
	\end{enumerate}
\end{corollary}

\newcommand{\proofCorPredictionEquilibriaExistence}[1][Proof]{
\begin{proof}[#1]
	Note first, that the routing operator~$\ROp$ defined in \eqref{eq:RoutingOperatorPredictionEquilibrium} always has non-empty \revMinor{values: This is witnessed by the flow split $r$ defined by} $r_{vw,i}\coloneqq 0$, for $v=\dest_i$, and $r_{vw,i}\coloneqq \CharF_{\hat\Theta_{vw,i}(f)} / \sum_{e\in\edgesFrom{v}}\CharF_{\hat\Theta_{e,i}}(f)$, for $v\neq \dest_i$\revised{, which is well-defined since $v$ has at least one active outgoing edge at all times.}
	Therefore, \revMinor{thanks to} \Cref{lem:alt-continuity-prediction-equilibrium,lem:decomposable-implies-regular}, we can apply \Cref{thm:existence-finite-time-horizon} to get existence of a $\EdgeLoading$-$\ROp$-coherent flow until time~$\horizon$. \Cref{lem:routing-operator-prediction-equilibrium} then ensures that this is also a dynamic prediction equilibrium until~$\horizon$.
	
	Under the additional assumption that all predictors $\hat C_{i,p}$ are causal, we immediately get causality of~$\ROp$. Together with the assumption that $\EdgeLoading$ is causal as well, we can apply \Cref{thm:ExtensionIfContinuous} to obtain the existence of a coherent flow (and, hence, a dynamic prediction equilibrium) until~$\infty$.
\end{proof}
}\proofCorPredictionEquilibriaExistence

\begin{remark}
	If the map $f \mapsto \ttime_e(\emptyarg, f)$ is sequentially weak-strong continuous \wrt $(\flowSet_\horizon, \revMinor{\normFct_p})$ and $(C([0,\horizon],\IRnn),\revMinor{\normFct_\infty})$ for all $e\in E$ (which is the case for both Vickrey's model and the affine-linear volume-delay dynamics), then both the perfect predictor and the constant predictor become continuous in the sense of \Cref{cor:prediction-equilibria-existence}.
	Therefore, \Cref{cor:prediction-equilibria-existence}~\ref{cor:prediction-equilibria-existence:finite}{} subsumes the existence results of dynamic Nash equilibria for Vickrey's model (\cite[Theorem~8]{CCLDynEquil}) and the affine-linear volume-delay model (\cite[Theorem~4.2]{ZhuM00}). Similarly, 
	\Cref{cor:prediction-equilibria-existence} generalizes the existence results for instantaneous dynamic equilibria, which can be found for Vickrey's model in \cite[Theorem~5.6]{DynamicFlowswithAdaptiveRouteChoice}, and the existence results for dynamic prediction equilibria in \cite[Theorem~15]{PredictionEquilibria}.
		Further, the above result can also be applied to new predictors obtained by combining the predictors mentioned so far,
		for example, to a perfect predictor with limited time horizon where perfect predictions are used until some finite time horizon and constant predictions afterwards.
\end{remark}

\subsubsection{Prescriptive Routing Operators}\label{subsec:prescriptive-routing-operators}

Let us now summarize our results on prescriptive routing operators.
\revised{Recall that we call $\ROp$ prescriptive if $\ROp(f)$ consists of a single element for all flows $f$; we denote this single element by $\ROe(f)$.}
\revised{For these operators, it is straightforward to see that they fulfill decomposability (\Cref{def:decomposable}) if $\ROe_\horizon$ is sequentially weak-strong continuous:}

\newcommand{\lemPrescriptiveOperatorsClosedGraphTheorem}{
\begin{lemma}\label{lem:prescriptive-operators-closed-graph-theorem}
    Let $\ROp$ be a prescriptive routing operator.
    Then, for every $\horizon\in\IRnn$, \revised{$\ROp_\horizon$ is decomposable} if $\ROe_\horizon$ is sequentially weak-\revised{strong} continuous \revised{\wrt $\normFct_p$ and $\normFct_1$},
    where we define
	\[
		\ROe_\horizon : \flowSet_\horizon \to \revised{\splitFcts}_\horizon^{E\times I},
		\quad f \mapsto \CharF_{[0,\horizon]}\cdot \ROe(f).
	\]
\end{lemma}
\begin{proof}
	\revised{
		Clearly, $\ROp_\horizon(f)$ can be written as the set of flow splits $\tilde{r}\in\splitFcts_\horizon^{E\times I}$ with $\tilde{r}_{e,i} \leq \ROe_\horizon(f)_{e,i}$ and $-\tilde{r}_{e,i} \leq -\ROe_\horizon(f)_{e,i}$, and thus fits the form of~\eqref{eq:decomposability}.
		As $\ROe_\horizon(f)$ is bounded by $1$ for all $f$, $\ROe_\horizon$ is also weak-strong continuous \wrt $\normFct_p$ and $\normFct_q$ with $q$ defined as $\nicefrac{1}{p} + \nicefrac{1}{q} = 1$ (\cf \Cref{prop:p-integrable-bounded-codomain}).
		This implies that $\ROp_\horizon$ is decomposable.
	}
\end{proof}
}

\lemPrescriptiveOperatorsClosedGraphTheorem
This, together with \revMinor{our} results from \Cref{sec:ExistenceUniqueness}, gives us the following existence and uniqueness results for prescriptive operators.

\begin{corollary}\label{cor:prescriptive-operators}
	Let $\EdgeLoading$ be a locally bounded edge loading operator and let $\ROp$ be a prescriptive routing{} {}operator.
	\begin{enumerate}[label=(\roman*)]
		\item\label{cor:prescriptive-operators:finite} If, for some $\horizon\in\IRnn$, $\EdgeLoading_\horizon$ is sequentially weak-weak continuous \revMinor{\wrt $\normFct_p$} and $\ROe_\horizon$ is sequentially weak-\revised{strong} continuous \revMinor{\wrt $\normFct_p$ and $\normFct_\revised{1}$}, then there exists a coherent flow until $\horizon$.
		\item\label{cor:prescriptive-operators:infinite-causal} If \revMinor{the assumptions of~\ref{cor:prescriptive-operators:finite} hold for all $\horizon\in\IRnn$,} and $\EdgeLoading$ and $\ROe$ are causal, then there exists a coherent flow until $\infty$.
		\item\label{cor:prescriptive-operators:strictly-causal} If both $\EdgeLoading$ and $\ROe$ are strictly causal, there exists a unique coherent flow until $\infty$.
		\item\label{cor:prescriptive-operators:lipschitz} If $\EdgeLoading$ is uniformly strictly causal and essentially bounded, $\ROe$ is causal, and $\ROe_\horizon$ is Lipschitz-continuous \wrt $\normFct_1$ and \revMinor{$\normFct_{\revised{1}}$}, for all $\horizon\in\IRnn$, and the network inflow rates vector $u$ is in $\pLocInt[\infty]^{V\times I}$, then there exists a unique coherent flow until~$\infty$.
	\end{enumerate}
\end{corollary}
\newcommand{\proofCorPrescriptiveOperators}[1][Proof]{
\begin{proof}[#1]
	For statements~\ref{cor:prescriptive-operators:finite} and~\ref{cor:prescriptive-operators:infinite-causal}, we apply \Cref{thm:existence-finite-time-horizon} and \Cref{thm:ExtensionIfContinuous}, respectively,  together with \Cref{lem:prescriptive-operators-closed-graph-theorem,lem:decomposable-implies-regular}.
	For statement~\ref{cor:prescriptive-operators:strictly-causal}, we apply \Cref{prop:existence-if-strictly-causal} and \Cref{lem:unique-extension-if-strictly-causal}.
	Finally, for statement~\ref{cor:prescriptive-operators:lipschitz}, we apply \Cref{thm:unique-existence-for-lipschitz-operators}.
\end{proof}
}
\proofCorPrescriptiveOperators

\begin{remark}
	This corollary generalizes the existence and uniqueness results for prescriptive operators presented by \citet{Bayen2019}:
	They consider causal operators and show (a) existence for ``continuous'' such operators on finite horizons \revised{where they assume strong-strong continuity from edge flow volumes $(\theta\mapsto \int_0^{\theta} f^+ - f^- \diff\leb)$ in $(C([0, \horizon])^{E\times I}, \normFct_\infty)$ to flow splits in $(\pInt[1][[0, \horizon]]^{E\times I}, \normFct_1)$, a similar} continuity assumption as the one specified for $\ROe_T$ in \Cref{cor:prescriptive-operators}~\ref{cor:prescriptive-operators:finite}{} and~\ref{cor:prescriptive-operators:infinite-causal}{} \revised{(see also \Cref{remark:cumulative-flow})}, (b) unique existence for so-called ``delay-type'' routing operators that are special cases of uniformly strictly causal operators and, thus, included in \ref{cor:prescriptive-operators:strictly-causal}, and (c) unique existence for ``Lipschitz-continuous'' operators with a condition similar to the Lipschitz condition on $\ROe$ in \ref{cor:prescriptive-operators:lipschitz}.
\end{remark}

One tempting question is whether, for a given set of cost predictors, there exists a \emph{prescriptive} routing operator for which the set of coherent flows \revMinor{is exactly the set of \DPEs{} \revised{(\Cref{def:dpe})} (or at least a subset of the \DPEs)}.
Of course, for any given \DPE{} $f$, we can ``a posteriori'' construct a prescriptive routing operator as $r_{e,i} = f_{e,i}^+ / (\sum_{e'\in\edgesFrom{v}} f_{e,i}^+)$ (whenever well-defined) such that (only) $f$ is coherent with respect to this operator.
However, the use of the equilibrium solution $f$ in the definition of the behavioural model is not satisfying as usually we want to find $f$ as the solution to the model.

Assuming all cost predictors are causal, another attempt would be to choose for every node and every possible set of outgoing edges some fixed flow-split over the active outgoing edges.
Such a routing operator would then still be prescriptive, and any coherent flow would be a \DPE.
In fact, for IDE \revised{(\Cref{ex:ide})}, this is exactly what is discussed in \textcite[Routing 4.10]{Bayen2019}:
They define a routing operator that equally distributes flow over all outgoing edges starting a currently shortest path.
Any coherent flow would then be an IDE.

However, as already noted in \cite{Bayen2019}, this routing operator does not satisfy the continuity conditions{} introduced there (nor the ones used in this paper) and, hence, they cannot show existence of such flows.{} 
In fact, it is not hard to see that coherent flows \wrt such a routing operator need not exist in general:

\newcommand{\exampleIdeNonPrescriptive}{
	\begin{example}\label{ex:NonExistenceNaiveIDE}
		Consider a single-source, single-sink network with two parallel paths consisting of two edges each.
		We use Vickrey's queuing model \revised{(\Cref{sec:vickrey}) as} the physical model.
		All edges have a free flow travel time of $1$.
		The edges on \revMinor{the bottom} path have a capacity of \revMinor{$2$} and \revMinor{$1$}, respectively, while \revMinor{both} edges on the \revMinor{top} path have a capacity of \revMinor{$1$} (\cf \Cref{fig:NonExistenceNaiveIDE}).
		
		\begin{figure}[ht]
			\centering
			\begin{tikzpicture}
	\node[vertex](s) at (0,0) {$s$};
	\node[vertex](t) at (6,0) {$\dest$};
\node[vertex](v) at (3,-0.5) {$w$};
	\node[vertex](w) at (3, 0.5) {$v$};
	\draw[edge](s) to[out=-35,in=180]node[below]{2,1} (v);
	\draw[edge](v) to[out=0,in=215]node[below]{1,1} (t);
	\draw[edge](s) to[out=35,in=180]node[above]{1,1} (w);
	\draw[edge](w) to[out=0,in=145]node[above]{1,1} (t);
	
	\draw[line width=7pt,<-,flowCol] (s) -- +(-1,0)coordinate(inflow);
	\node[left,flowColLabel]at(inflow){$u_s \equiv 3$};
\end{tikzpicture} 			\caption{A network in Vickrey's queuing model with edge labels $(\nu_e, \ffttime_e)$ with capacity $\nu_e$ and free-flow travel time $\ffttime_e$.}\label{fig:NonExistenceNaiveIDE}
		\end{figure}
		
		For a constant network inflow rate of \revMinor{$3$}, this network has a unique IDE \revised{(see \Cref{fig:NonExistenceNaiveIDEAnimation})}.
		It starts with a flow split in a ratio of \revMinor{$1:2$} between the upper and the lower paths\revised{, since this is the unique ratio that avoids the buildup of any queues at the outgoing edges of $s$ and keeps the instantaneous travel times of both paths minimal (during this initial phase).}
		Starting at time~$1$ a queue builds up on the second edge of the \revMinor{lower} path at a rate of~\revMinor{$1$} and the only possible flow split \revMinor{at~$s$} is \revMinor{$2:1$}\revised{, which then also leads to an emerging queue at the first edge of the upper path.}
		\revMinor{Starting a}t time~$2$ the queue length on the second edge of the \revMinor{lower} path remains constant at~\revMinor{$1$} while the first edge of the \revMinor{upper} path \revMinor{maintains} a \revMinor{constant} queue length of~\revMinor{$1$}.
		Thus, the split changes back to \revMinor{$1:2$} until time~$3$.
		\revised{This periodic behavior continues forever, where the queues of the first edge on the upper path and of the second edge on the lower path grow in lockstep every other unit time interval at a rate of~$1$.}
		
		\begin{figure}[ht]
			\centering
			\subfloat[Four snapshots from the evolution of the unique IDE.]{
				\begin{minipage}{\textwidth}

					\includeinkscape[width=0.49\textwidth]{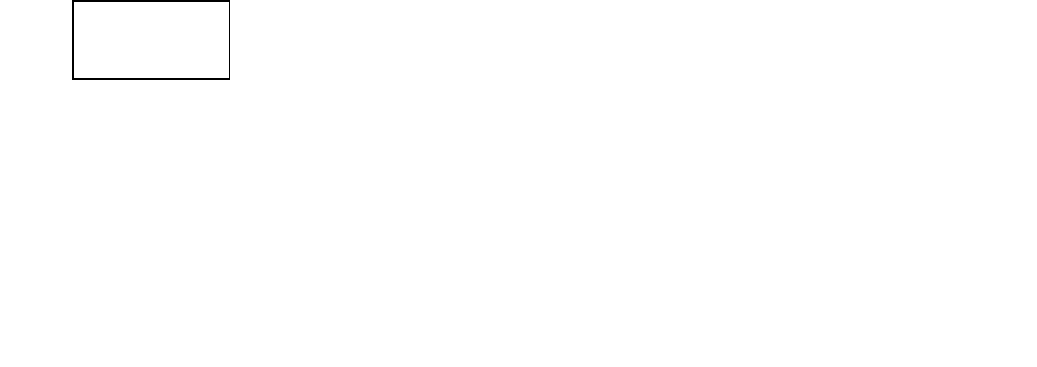}
					\includeinkscape[width=0.49\textwidth]{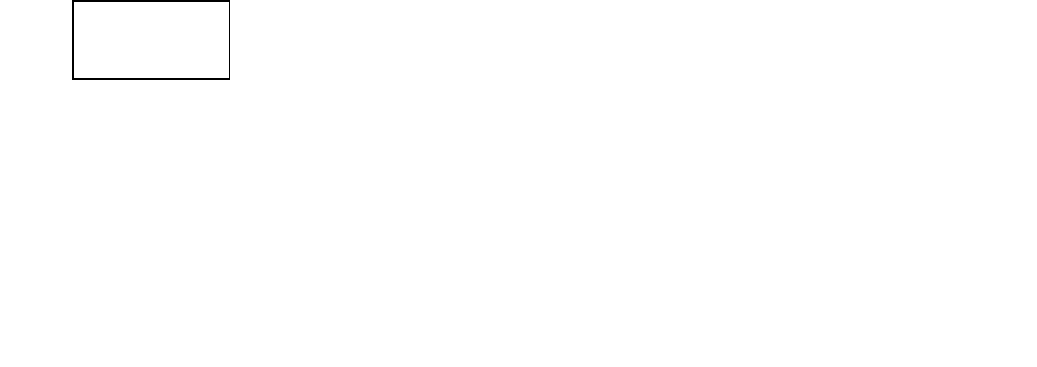}

					\includeinkscape[width=0.49\textwidth]{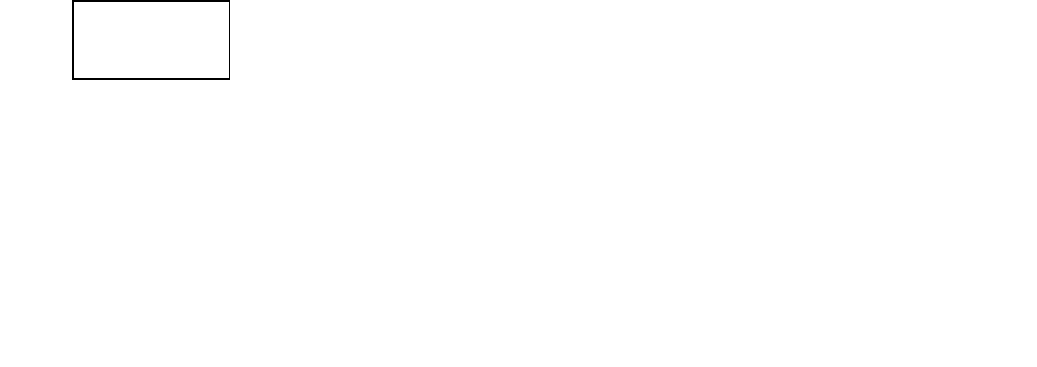}
					\includeinkscape[width=0.49\textwidth]{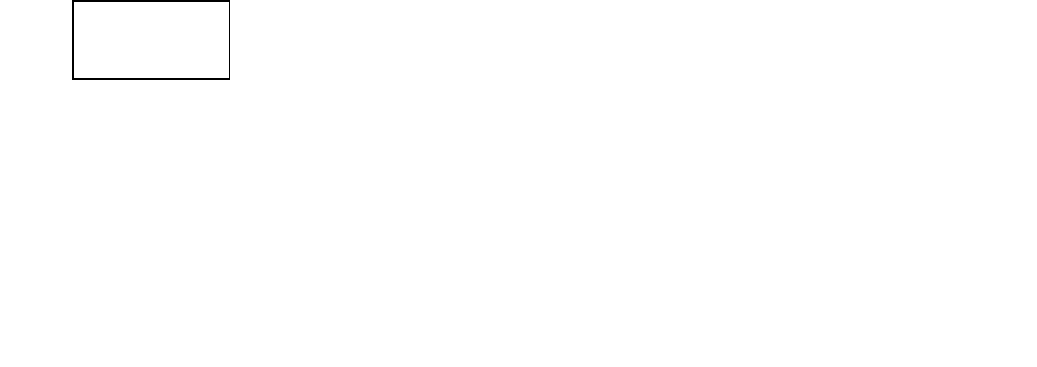}
				\end{minipage}
			}
			\\
			\subfloat[The (coinciding) instantaneous path travel times (left), and the flow split $\nicefrac{f_e^+}{u_s}$ at node $s$ into the lower and upper edge (right).]{
				\begin{tikzpicture}
					\begin{groupplot}[group style={group size=2 by 1, horizontal sep=2.5cm}, width=0.45\textwidth, height=5cm]

						\nextgroupplot[xlabel={Time}, ylabel={instantaneous travel time}, grid=major, legend style={anchor=north west, at={(0.0, 1.0)},
							legend image code/.code={
								\draw [mark repeat=2,mark phase=2,##1]
									plot coordinates {
									(0cm,0cm)
									(8pt,0cm)
									(16pt,0cm)
									};
							}
						}]

							\addplot[thick, cyan, restrict x to domain=-inf:6, dashed, dash pattern=on 4pt off 4pt on 4pt off 4pt] table[x=x,y=highpath_cost,col sep=comma]{py/2_ide.csv};
							\addlegendentry{upper path}

							\addplot[thick, orange, restrict x to domain=-inf:6, dashed, dash pattern=on 4pt off 4pt on 4pt off 4pt, dash phase=-4pt] table[x=x,y=lowpath_cost,col sep=comma]{py/2_ide.csv};
							\addlegendentry{lower path}

						\nextgroupplot[xlabel={Time}, ylabel={flow split at $s$}, grid=major, ymax=1.0, ymin=0.0, legend style={anchor=north east, at={(1.0, 1.0)}}, ytick={0, 1/3, 2/3, 1}]
							\addplot[thick, cyan, restrict x to domain=-inf:6] table[x=x,y=edge2,col sep=comma]{py/2_ide.csv};
							\addlegendentry{upper edge}
							\addplot[thick, orange, restrict x to domain=-inf:6] table[x=x,y=edge0,col sep=comma]{py/2_ide.csv};
							\addlegendentry{lower edge}
					\end{groupplot}
				\end{tikzpicture}
			}

			\marginnote{}
			\caption{The unique IDE for the network from \Cref{fig:NonExistenceNaiveIDE} as described in \Cref{ex:NonExistenceNaiveIDE}.
			Here, each infinitesimal particle, that starts at some time $\theta$ in $s$, chooses a path minimizing the instantaneous travel time.}\label{fig:NonExistenceNaiveIDEAnimation}
		\end{figure}
		
		Note that the flow \revMinor{split} at node~$s$ changes multiple times even though at all times \revised{the instantaneous travel times of both paths coincide, and thus} both outgoing edges are \revMinor{always} active.
		Hence, no choice of a fixed flow \revMinor{split} over active edges can lead to the existence of a flow \revMinor{consistent with} the resulting routing operator.
	\end{example}
}
\exampleIdeNonPrescriptive

\section{Stochastic Prediction Equilibrium}\label{sec:SPE}

\begin{revisedEnv}
	Let us take a step back and consider \emph{static} traffic assignment for a moment.
	In static traffic assignment, a stochastic user equilibrium reflects the fact that each agent may have a slightly different perception of travel cost.
	The standard model proposed by \citet{Daganzo1977} thus assumes that path travel times are flow-dependent random variables.
	A flow is then a stochastic user equilibrium if and only if every path is used by (at most) the proportion of demand that perceives this path as their best option, according to the distribution of the random variables.
	This models exactly the situation in which every infinitesimal agent is assigned a (random) vector of path travel times, and chooses a path minimizing travel time.
	Using a specific distribution of path travel times then yields the well-known logit assignment model.
\end{revisedEnv}

In this section, we use our framework to introduce \revised{these} stochastic effects into the concept of prediction equilibria \revised{for \emph{dynamic} traffic assignment}.
For this, we define the new concept of a \emph{stochastic prediction equilibrium}.
The idea here is to augment \revised{path cost} predictions with random noise so that the effective
cost predictor becomes a random variable. 
As our main result, we show existence and uniqueness of stochastic prediction equilibria under mild assumptions.
\revised{In particular, we show that under natural assumptions, the resulting routing operator for stochastic prediction equilibria becomes prescriptive.
Thus, this analysis explains how the incorporation of stochastic errors in the perception of the infinitesimal agents can lead to some of the previously studied prescriptive routing operators, and  highlights their underlying stochastic assumptions.}

\newcommand*{\randDist}{U}

\revised{Compared to \DPE{} (\Cref{def:dpe}),} we model \revMinor{the stochastic} noise by assuming that the cost predictor $\hat C_{i,p}$ is no longer \revMinor{a} fixed \revMinor{function $\IR\times\flowSet\to \IR$} but rather a random variable \revMinor{of such functions}.
More specifically, a particle of commodity $i$ that arrives at a node $v\neq \dest_i$ at some time $\bar\theta$ is assigned a cost predictor \[ 
	\hat C: \IR \times \flowSet \to \IR^\PathSet,
	\quad
	(\theta, f)\mapsto (\hat C_{p}(\theta, f))_{p\in\PathSet}
\] according to the random distribution of a probability measure $\Prob_i$ over such prediction functions.
The particle then evaluates the cost $\hat C_p(\bar\theta)$ for every $v$-$\dest_i$-path $p$ and chooses to enter an edge that lies on a path minimizing this value.

Let $\tilde E_i(\theta, f, \hat C)$ denote this set of perceived active edges at time $\theta$ \wrt flow $f$ and predictor~$\hat C$\revised{, \ie $\tilde E_i(\theta, f, \hat C)$ is the set of edges $e=vw\in E$ for which there exists a $v$-$d_i$-path starting with $e$ that minimizes $\hat C_p(\theta, f)$ among all $v$-$\dest_i$-paths.}
Then, for any time $\theta \in \IR$, any node $v \in V$, any commodity $i \in I$ and any subset of edges $M \subseteq \edgesFrom{v}$ leaving~$v$ the probability that $M$ is perceived as the set of active edges leaving~$v$ by particles of commodity~$i$ at time~$\theta$ is defined as 
\[
    \pi_{v,M,i}(\theta, f) \coloneqq \Prob_i(
        \set{
            \hat C
            |
            M = \tilde E_i(\theta, f, \hat C) \cap \edgesFrom{v}
        }
    ).
\]

We assume that all particles of commodity~$i$ carry out the same experiment independently of each other and, thus, among all particles of commodity~$i$ entering node $v$, the proportion of particles of that commodity perceiving the set $M$ as the set of active outgoing edges \revMinor{of $v$} at time $\theta$ is exactly $\pi_{v,M,i}(\theta,f)$.
\revised{For this proportion of flow, any choice of edge in $M$ is an optimal choice (as in the perception of these infinitesimal agents, all edges in $M$ lie on a shortest $v$-$d_i$-path).
Thus, any partition $(r_{M,e,i})_{e\in M} \subset \IRnn^M$ of $\pi_{v,M,i}(\theta, f)$ into the edges in $M$ is a valid ``subrouting'' for the proportion $\pi_{v,M,i}(\theta, f)$;} here, $r_{M,e,i}$ is the share of the particles entering node $v$ that perceive $M$ as their set of active edges and choose the edge $e$ as their next edge.
\revised{Summing over all these subroutings, yields the overall flow split.
More precisely, let} $\mathcal E$ denote the set of all pairs $(M, e)$ of the form $M\subseteq \edgesFrom{v}$ and $e\in M$ for any $v\in V$.
We can then define the following routing operator
\begin{equation}\label{eq:stochastic-routing-operator}
	\newcommand{\conds}{\begin{aligned}
		\forall v, M\neq\emptyset, i: &\quad \textstyle \sum_{e \in M}r_{M,e,i} = \pi_{v,M,i}(\emptyarg, f)
		\end{aligned}}
    \ROp(f) \coloneqq \Set{
        \Bigl(\sum_{M \ni e}r_{M,e,i}\Bigr)_{e,i}
        |
        \begin{array}{c}
            (r_{M,e,i})_{(M,e,i)} \in \revised{\splitFcts}^{\mathcal E \times I}: \\[1ex]
			\conds
        \end{array}
	}.
\end{equation}

\begin{definition}[SPE]\label{def:SPE}
	\revised{Given probability measures $\Prob_i$ over prediction functions $\hat C$,} a \emph{stochastic prediction equilibrium (SPE)} until time $\horizon$ is a flow that is coherent (until $\horizon$) with respect to a given \revMinor{edge loading operator} $\EdgeLoading$ and the routing operator in \eqref{eq:stochastic-routing-operator}.
\end{definition}

Throughout this section, we assume that the mapping $(\hat C, \theta)\mapsto \hat C_p(\theta, f)$ is $(\Prob_i \otimes \leb)$-measurable for every flow $f$, path $p$ and commodity $i$.
This ensures that $\pi_{v,M,i}(\emptyarg, f)$ is measurable and, thus, that $\ROp$ is well-defined:

\newcommand{\propMeasurabilityStochasticRoutingOperator}{
\begin{proposition}\label{prop:measurability-stochastic-routing-operator}
	If $(\hat C, \theta)\mapsto \hat C_p(\theta, f)$ is $(\Prob_i \otimes \leb)$-measurable for all paths $p$, then $\pi_{v,M,i}(\emptyarg, f)$ is measurable.
\end{proposition}
}
\newcommand{\proofPropMeasurabilityStochasticRoutingOperator}[1][Proof]{
\begin{proof}[#1]
	The function $(\hat C,\theta) \mapsto \min_{p\in P} \hat C_p(\theta, f)$ is $(\Prob_i\otimes \leb)$-measurable for any $P\subseteq \PathSet_{v, \dest_i}$.
	Let $\PathSet_e$ denote the subset of paths $p\in\PathSet_{v,\dest_i}$ that use $e$ as their first edge.
	Then, the function $h_e(\hat C, \theta)\coloneqq \min_{p\in\PathSet_e} \hat C_p(\theta, f) - \min_{p\in\PathSet_{v,\dest_i}} \hat C_p(\theta,f)$ is measurable as well as the set $\Theta_{e} \coloneqq h_e^{-1}(0)$ (which is the set of pairs $(\hat C, \theta)$ for which $e$ is perceived active at time $\theta$).
	
	Now, for any specific $\theta$ we have \begin{equation*}
		\set{\hat C | M = \tilde E_i(\theta,f, \hat C) \cap \edgesFrom{v}}
		=
		\Big(
			\big(\bigcap_{e\in M} \Theta_e\big)\setminus\bigcup_{e\in\edgesFrom{v}\setminus M} \Theta_e
		\Big)_{\theta}
		\eqqcolon
		(\Theta_M)_\theta,
	\end{equation*}
	where we denote $(\Theta_M)_\theta \coloneqq \set{ \hat C | (\hat C, \theta)\in \Theta_M}$.
	As $\Theta_M$ is measurable, the function $\pi_{v,M,i}(\emptyarg, f) = \theta\mapsto \Prob_i((\Theta_M)_\theta)$ is also measurable by Fubini's Theorem (\cf \citealt[Theorem~14.19]{Klenke2020}).
\end{proof}
}
\propMeasurabilityStochasticRoutingOperator\proofPropMeasurabilityStochasticRoutingOperator

\begin{observation}
	If all $\Prob_i$ are trivial probability measures assigning probability 1 to a single prediction function $\hat C_i$,
	then the \SPE{} routing operator in \eqref{eq:stochastic-routing-operator} coincides with the \DPE{} routing operator \eqref{eq:RoutingOperatorPredictionEquilibrium} induced by $(\hat C_i)_i$.
\end{observation}

Before diving into the properties of \SPEs, we give a natural example of a distribution over the predictors $\hat C$ inducing a stochastic version of \IDEs:

\begin{example}[Stochastic IDE]\label{ex:StochasticIDE}
	We introduce a stochastic variant of instantaneous dynamic equilibria \revised{from \Cref{ex:ide}}:
	In this model,
	every time a particle arrives at an intermediate node, it retrieves the current, instantaneous travel times of each edge disturbed by a stochastic measurement error.
	Let $\ttime_e(\theta, f)$ denote the (actual) travel time of edge $e$ when entering $e$ at time $\theta$.
	Then, the particle retrieves the value $\hat \ttime_e(\theta, f, \varepsilon_{e,i})\coloneqq \ttime_e(\theta, f) + \varepsilon_{e,i}$ where the measurement error $\varepsilon_{e,i}$ is distributed according to some random distribution $U_{e,i}$.
	The particles now aim to minimize their travel time according to the instantaneous, erroneous path costs $\hat C_{p,i}(\theta, f, \boldeps) = \sum_{e\in p} \hat \ttime_e(\theta, f, \varepsilon_{e,i})$ where $\boldeps$ denotes the error vector $(\varepsilon_{e,i})_e$.
	Accordingly, the distributions $(U_{e,i})_e$ induce the probability measure $\Prob_i$.
	To ensure that $\ROp$ is well-defined, we assume that $\ttime_e(\emptyarg, f)$ are measurable functions.
	A flow that is coherent \wrt a given \revMinor{edge loading operator} $\EdgeLoading$ and this routing operator is called a \emph{stochastic IDE}.

	\begin{revisedEnv}
	Let us reconsider the concrete network from \Cref{ex:NonExistenceNaiveIDE} and compare the (deterministic) IDE depicted in \Cref{fig:NonExistenceNaiveIDEAnimation} with the stochastic IDE model.
	For this, we assume that the measurement errors $\varepsilon_{e,i}$ are independent and normally distributed around $0$ with standard deviation $\ffttime_e / 4$.
	Unlike deterministic IDE, stochastic IDE are \emph{always} unique under natural assumptions on the error destributions, as we will show in \Cref{cor:stochastic-ide}.
	The unique stochastic IDE for this setup is displayed in \Cref{fig:example-stochastic-ide}.

	\begin{figure}[ht]
		\centering
		\subfloat[Four snapshots of the evolution of the unqiue stochastic IDE.]{
			\begin{minipage}{\textwidth}
				\centering
				\includeinkscape[width=0.49\textwidth]{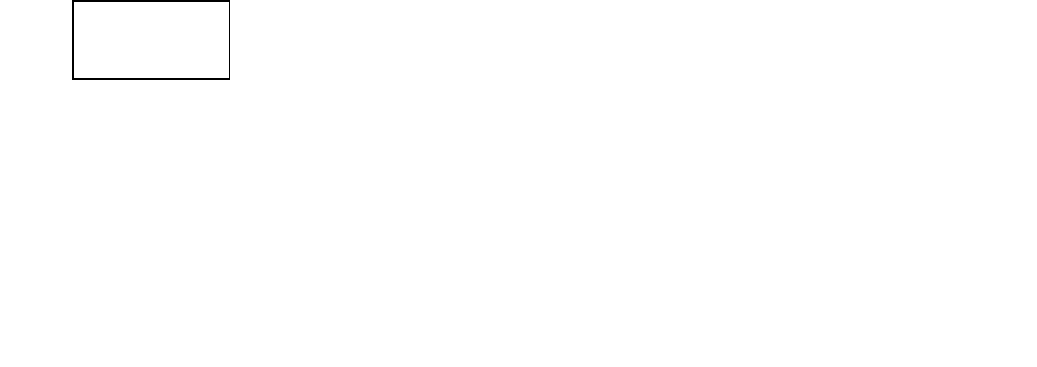}
				\includeinkscape[width=0.49\textwidth]{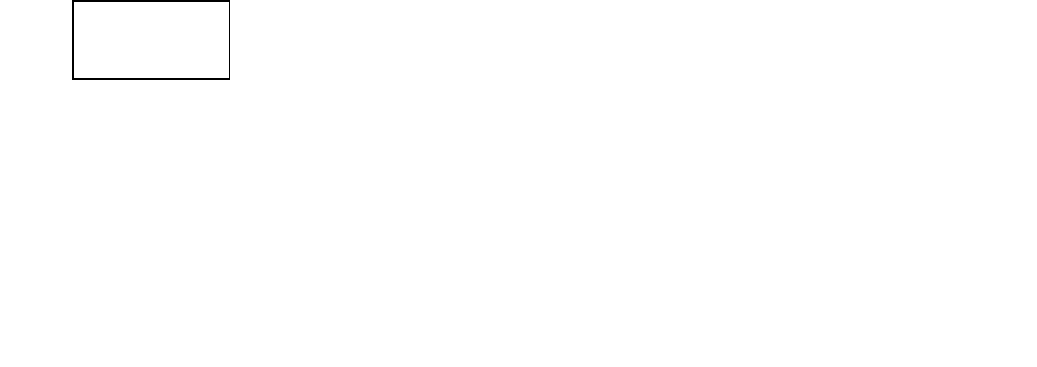}

				\includeinkscape[width=0.49\textwidth]{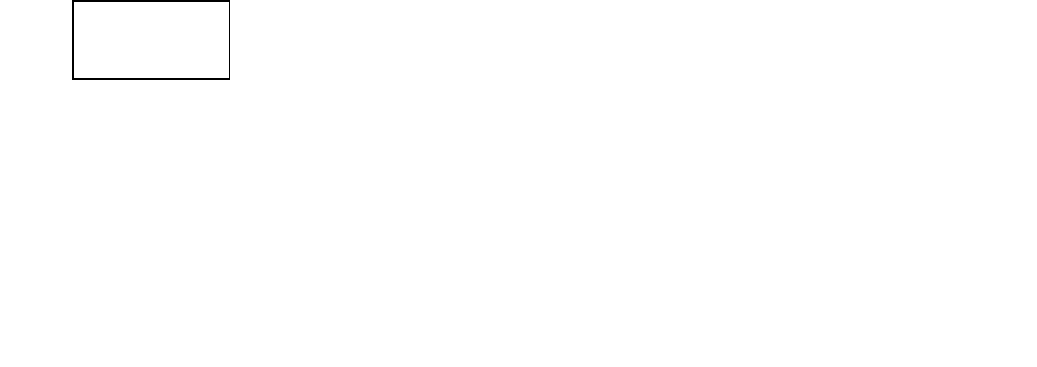}
				\includeinkscape[width=0.49\textwidth]{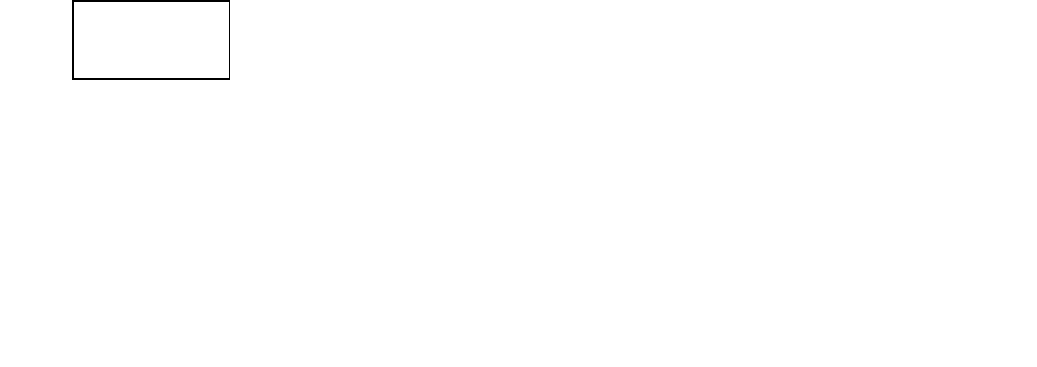}
			\end{minipage}
		}
		\\
		\subfloat[The instantaneous path travel times and the band within the standard deviation of the measured instantaneous travel times (left), and the flow split $\nicefrac{f_e^+}{u_s}$ at node $s$ into the lower and upper edge (right).]{
			\centering
			\begin{tikzpicture}
				\begin{groupplot}[group style={group size=2 by 1, horizontal sep=2.5cm}, width=0.45\textwidth, height=5cm]

					\nextgroupplot[xlabel={Time}, ylabel={\begin{tabular}{c}instantaneous travel time,\\band within standard deviation\end{tabular}}, grid=major, legend style={anchor=north west, at={(0.0, 1.0)}}]
\def\epsBand{0.3535533905932738}

						\addplot [
							draw=none, 
							name path=highpath_upper, 
							restrict x to domain=-inf:6,
							forget plot
						] table [x=x, y expr=\thisrow{highpath_cost} + \epsBand,col sep=comma]{py/2_stochastic_ide_0.25.csv};
						\addplot [
							draw=none, 
							name path=highpath_lower,
							restrict x to domain=-inf:6,
							forget plot
						] table [x=x, y expr=\thisrow{highpath_cost} - \epsBand,col sep=comma]{py/2_stochastic_ide_0.25.csv};
						\addplot [fill=cyan, opacity=0.15, forget plot] fill between[of=highpath_upper and highpath_lower];

						\addplot [
							draw=none, 
							name path=lowpath_upper, 
							restrict x to domain=-inf:6,
							forget plot
						] table [x=x, y expr=\thisrow{lowpath_cost} + \epsBand,col sep=comma]{py/2_stochastic_ide_0.25.csv};
						\addplot [
							draw=none, 
							name path=lowpath_lower,
							restrict x to domain=-inf:6,
							forget plot
						] table [x=x, y expr=\thisrow{lowpath_cost} - \epsBand,col sep=comma]{py/2_stochastic_ide_0.25.csv};
						\addplot [fill=orange, opacity=0.15, forget plot] fill between[of=lowpath_upper and lowpath_lower];

						\addplot[thick, cyan, restrict x to domain=-inf:6] table[x=x,y=highpath_cost,col sep=comma]{py/2_stochastic_ide_0.25.csv};
						\addlegendentry{upper path}

						\addplot[thick, orange, restrict x to domain=-inf:6] table[x=x,y=lowpath_cost,col sep=comma]{py/2_stochastic_ide_0.25.csv};
						\addlegendentry{lower path}

					\nextgroupplot[xlabel={Time}, ylabel={flow split at $s$}, grid=major, legend style={anchor=north east, at={(1.0, 1.0)}}, extra y ticks={1/3, 1/2, 2/3}]
						\addplot[thick, cyan] table[x=x,y=edge2,col sep=comma, restrict x to domain=-inf:6]{py/2_stochastic_ide_0.25.csv};
						\addlegendentry{upper edge}
						\addplot[thick, orange] table[x=x,y=edge0,col sep=comma, restrict x to domain=-inf:6]{py/2_stochastic_ide_0.25.csv};
						\addlegendentry{lower edge}
				\end{groupplot}
			\end{tikzpicture}
		}

		\marginnote{}
		\caption{The unique stochastic IDE for the network in \Cref{fig:NonExistenceNaiveIDE} as described in \Cref{ex:StochasticIDE}. In this example, each infinitesimal particle, that starts at some time $\theta$ in $s$, measures the instantaneous travel times of every edge subject to a normally distributed error with a standard deviation of $\ffttime_e / 4$, and chooses the path that minimizes this cost.}\label{fig:example-stochastic-ide}
	\end{figure}

	This flow begins similarly as the deterministic IDE in the sense that during the interval $[0,1]$, more flow enters the lower edge.
	However, the flow split starts at time $0$ at a ratio of exactly $1.5 : 1.5$ between the upper and lower edge, and then gradually shifts towards the lower edge as a queue starts to build up in front of the first edge of the upper path.
	At time $1$, the flow split reaches a ratio of almost $1 : 2$.
	Then a queue starts to build up at edge $w\dest$
	which in turn leads to a gradual shift of the split towards the upper edge.
	Roughly at time $1.47$, the flow split is equal, and at time $2.13$, the flow split reaches its next local extremum at a ratio of roughly $1.82: 1.18$.
	These oscillations continue with the next extrema roughly at times $3.46$ and $4.84$, however its amplitudes decrease with time, and the flow split ratio converges to an equal split.

	Unlike for the deterministic IDE, the instantaneous travel times of both paths do not coincide, especially in the beginning of the observation period.
	Due to the measurement errors, small changes in the instantaneous travel times only affect the routing decisions of a small portion of the flow, and hence these errors act as a regularization of the routing behaviour.
	Instead of the observed periodic jumps of the flow split of the deterministic IDE (\cf \Cref{fig:NonExistenceNaiveIDEAnimation}), that are necessary to keep the instantaneous travel times of both paths equal at all times, the regularization here prevents such jumps which leads to the described damped oscillations both in the difference in travel times of both paths, and in the flow split.
	\end{revisedEnv}
\end{example}

\begin{revisedEnv}
In both static and dynamic traffic assignment, a commonly used route-choice model is the so-called logit model (\eg \cite{Daganzo1977,Han2003}).
The following example shows that this logit-based route choice (with instantaneous path travel times) emerges as an instantiation of SPE when assuming independent reverse Gumbel distributions for the path cost predictors.
It coincides with the routing operator defined in \cite[Routing 4.12]{Bayen2019}.

\begin{example}[Instantaneous logit-based route choice]
	Consider the following instantiation of stochastic prediction equilibria where we assume that the path cost predictors $\hat C_{p}(\theta, f)$ are distributed in $\Prob_i$ according to independent reverse Gumbel distributions\footnote{
		\revised{The cumulative distribution function of the reverse Gumbel distribution for location parameter $\mu$ and scale parameter $\beta$ is given by $F(x) = 1 - \exp(-\exp( ( x - \mu ) / \beta ))$.
	}}
	with parameters $(\mu_{p}, \beta_i)$ where the location parameter $\mu_p \coloneqq \sum_{e\in p} c_e(\theta, f)$ is the path's instantaneous path travel time, and the scale parameter $\beta_i> 0$ is a commodity-specific constant.
	It should be noted that it is unrealistic to assume that the random variables $\hat C_{p}(\theta, f)$ are independent among all paths $p$ since usually, many paths overlap and thus, their perceived travel times correlate.
	Nevertheless, the independence allows us to use the softmax representation of Gumbel distributions \cite[Lemma~6]{Yellott1977}.
	Therefore, it is easy to see that for a $v$-$\dest_i$-path~$p$, the probability that $p$ minimizes $\hat C_{p}(\theta, f)$ among all $v$-$\dest_i$-paths $\PathSet_{v,\dest_i}$ is given by \[
		\Prob_i\left(\set{ \hat C | \hat C_{p}(\theta, f) = \min_{q\in\PathSet_{v, \dest_i}} \hat C_{q}(\theta, f) }\right)
		= \frac{
			\exp(- \frac{1}{\beta_i} \sum_{e\in p} c_e(\theta, f))
		}{
			\sum_{q\in\PathSet_{v, \dest_i}} \exp(- \frac{1}{\beta_i} \sum_{e\in q} c_e(\theta, f))
		}
	\]
	As we will show in \Cref{prop:stochastic-dpe-prescriptive}, the corresponding routing operator is also prescriptive and, for $e=vw$, given by \[
		\ROe(f)_{e,i} = 
		\sum_{\substack{p\in \PathSet_{v,\dest_i} \\ p_1 = e}}
		\frac{
			\exp(- \frac{1}{\beta_i} \sum_{e\in p} c_e(\theta, f))
		}{
			\sum_{q\in\PathSet_{v, \dest_i}} \exp(- \frac{1}{\beta_i} \sum_{e\in q} c_e(\theta, f))
		}.
	\]
\end{example}
\end{revisedEnv}

We now analyse the properties of \revMinor{general} \SPEs.
\revised{Since the natural definition of the \SPE{} routing operator $\ROp$ in \eqref{eq:stochastic-routing-operator} is not of the form of a decomposable routing operator as in \eqref{eq:decomposability}, we first provide} an equivalent description of $\ROp$\revised{.}
\revised{This allows} us to show that $\ROp_\horizon$ \revised{is in fact decomposable, and thus regular,} under certain additional conditions.
In turn, this enables us to prove existence of \SPEs{} with the developed theory from \Cref{sec:ExistenceUniqueness}.
\newcommand{\notationProbSets}{\begin{notation}We abbreviate the notation of sets of predictors~$\hat C$ by writing $\set{ A(\hat C) }$ instead of $\set{ \hat C | A(\hat C) }$ for any predicate $A$.
	\end{notation}
}\notationProbSets

\begin{lemma}\label{lem:CharacterizationStochasticIDERouting}
	We denote the probability that \revised{$M$ contains} all edges in~$\edgesFrom{v}$ that are perceived as active by particles of commodity~$i$ at time~$\theta$ by
	\begin{equation}\label{eq:spe-rho}
		\rho_{v,M,i}(\theta,f)\coloneqq \Prob_i(\set{ \tilde E_i(\theta,f, \hat C) \cap \edgesFrom{v} \subseteq M })
	\end{equation}
	for all $v\in V$, $M\subseteq \edgesFrom{v}$, $i\in I$, $\theta\in\IR$, and $f\in\flowSet$.
	Then, we have \begin{equation}\label{eq:stochastic-routing-operator-alt-description}
		\newcommand{\conds}{
			\begin{aligned}
			\textstyle \forall v, i: &\textstyle \quad \sum_{e \in \edgesFrom{v}} r_{e,i} = \CharF_{v\neq \dest_i}, \\
			\textstyle \forall i, v\neq \dest_i, M\subseteq \edgesFrom{v}: & \textstyle \quad \sum_{e \in M}r_{e,i} \geq \rho_{v,M,i}(\emptyarg, f) \\
			\end{aligned}
		}\ROp(f) = \Set{
			(r_{e,i})_{(e,i)} \in \splitFcts^{E \times I}
			| \conds
		}.
	\end{equation}
\end{lemma}
\newcommand{\proofPropCharacterizationStochasticIDERouting}[1][Proof]{
\begin{proof}[#1]
	Let $r\in\ROp(f)$.
	Then, for every $e\in E$, $i\in I$, the entry $r_{e,i}$ is of the form $\sum_{M\ni e} r_{M,e,i}$ with $(M,e)\in \mathcal E$.
	For all $v\in V$ it holds that \begin{multline*}
		\sum_{e\in\edgesFrom{v}} r_{e,i}
		= \sum_{e\in\edgesFrom{v}} \sum_{\substack{M\subseteq\edgesFrom{v}\\M\ni e}} r_{M,e,i}
		= \sum_{\substack{M\subseteq\edgesFrom{v}\\M\neq\emptyset}} \sum_{e\in M} r_{M,e,i}
		= \sum_{\substack{M\subseteq\edgesFrom{v}\\M\neq\emptyset}} \pi_{v,M,i}(\emptyarg, f)
		= \CharF_{v\neq \dest_i}.
	\end{multline*}
	Furthermore, for $v\neq \dest_i$ and for every $M\subseteq \edgesFrom{v}$ we have
	\begin{align*}
		\sum_{e \in M} r_{e,i}
		&= \sum_{e\in M} \sum_{\substack{M'\subseteq \edgesFrom{v}\\M'\ni e}} r_{M',e,i}
		\geq \sum_{e\in M} \sum_{\substack{M'\subseteq M\\M'\ni e}} r_{M',e,i}
		= \sum_{M'\subseteq M} \sum_{e\in M'} r_{M',e,i}
		= \sum_{\substack{M'\subseteq M\\ M'\neq\emptyset}} \pi_{v,M',i}(\emptyarg, f)
		\\
		&= \Prob_i(\set{
			\hat C	
			|
			\tilde E_{i}(\theta, f, \hat C) \cap \edgesFrom{v} \subseteq M
		})
		= \rho_{v,M,i}(\theta,f)
	\end{align*}
	where the second to last equation holds because $\tilde E_{i}(\theta, f, \hat C) \cap \edgesFrom{v}$ cannot be the empty set. 

	For the other direction, let $r$ be an element of the right-hand side of~\eqref{eq:stochastic-routing-operator-alt-description}.
	We use the max-flow min-cut theorem to show that a suitable decomposition of each $r_{e,i}$ exists.
	For this, fix some $\theta\in\IR$, $v\in V$, and $i\in I$.
	We need to determine suitable rates $r_{M,e,i}(\theta)$ for $M\subseteq \edgesFrom{v}$ and $e\in M$ such that $r_{e,i}(\theta)=\sum_{M\ni e} r_{M,e,i}(\theta)$, for every $e\in\edgesFrom{v}$, and $\sum_{e\in M}r_{M,e,i} = \pi_{v,M,i}(\theta,f)$, for every $M\subseteq \edgesFrom{v}$ with $M\neq \emptyset$, hold.
	If $v=\dest_i$, we simply set $r_{M,e,i}\coloneqq 0$. 
	This is possible as $r_{e,i} = 0$ holds for every $e\in \edgesFrom{v}$, and for every $M\neq \emptyset$ we have $\pi_{v,M,i}(\theta,f)=0$.

	Otherwise, $v\neq \dest_i$.
	We build an artificial network $(W,A)$ (see \Cref{fig:NetworkMaxFlowArgument}) to apply the max-flow min-cut theorem \revised{(\cf \cite[Theorem~3.10]{Schrijver2003})}.
	To avoid confusion with the ``natural'' edges and nodes of our traffic network, we refer to the directed edges $A$ as \emph{arcs} and to the nodes $W$ as \emph{vertices}.
	The network is a directed acyclic graph consisting of four layers of vertices:
	In the first layer, there is only a single artificial source vertex $s$.
	The second layer consists of a vertex for each subset $M$ of $\edgesFrom{v}$; each such vertex has an incoming arc from $s$ with capacity $\pi_{v,M,i}(\theta,f)$.
	The third layer consists of a vertex for each edge in $\edgesFrom{v}$; for each subset $M\subseteq \edgesFrom{v}$ there is an arc from $M$'s vertex to $e$'s vertex for every $e\in M$.
	These arcs have infinite capacity (and thus never occur in a min-cut).
	Finally, the fourth layer consists only of a single artificial sink vertex $\dest$.
	For every edge $e$ there is an arc from $e$ to $\dest$ with capacity $r_{e,i}(\theta)$.
	
	\begin{figure}[ht]\centering
		\begin{adjustbox}{max width=\linewidth}
			\begin{tikzpicture}[y=0.7cm]
	\begin{scope}[xshift=0cm]
		\node[vertex](v)at(0,0){$v$};
		\node[vertex](v2)at(2,1){};
		\node[vertex](v3)at(2,-1){};
		\draw[edge](v) --node[sloped,above]{$e_1$} (v2);
		\draw[edge](v) --node[sloped,above]{$e_2$} (v3);
		
		\draw[edge,dashed,<-](v)--++(-.75,.4);
		\draw[edge,dashed,<-](v)--++(-.75,-.4);
		
		\draw[edge,dashed](v2)--++(.75,.3);
		\draw[edge,dashed](v2)--++(.75,-.3);
		\draw[edge,dashed](v3)--++(.75,0);
	\end{scope}
	
	\begin{scope}[xshift=6cm]
		\node[vertex](s) at (-2,0) {$s$};
		
		\node[vertex](M1) at (3,3) {$\emptyset$};
		\node[vertex](M2) at (3,1) {$\{e_1\}$};
		\node[vertex,inner sep=0pt](M4) at (3,-1) {$\{e_1,e_2\}$};
		\node[vertex](M3) at (3,-3) {$\{e_2\}$};
		
		\node[vertex](e1) at (6,1.5) {$e_1$};
		\node[vertex](e2) at (6,-1.5) {$e_2$};
		
		\node[vertex](t) at (9,0) {$\dest$};
		
		\draw[edge] (s) --node[above,sloped,pos=.65]{$\pi_{v,\emptyset,i}(\theta,f)$}(M1);
		\draw[edge] (s) --node[above,sloped,pos=.65]{$\pi_{v,\{e_1\},i}(\theta,f)$}(M2);
		\draw[edge] (s) --node[above,sloped,pos=.65]{$\pi_{v,\{e_1,e_2\},i}(\theta,f)$}(M4);
		\draw[edge] (s) --node[above,sloped,pos=.65]{$\pi_{v,\{e_2\},i}(\theta,f)$}(M3);
		
		\draw[edge] (M2) --node[above,sloped]{$\infty$} (e1);
		\draw[edge] (M4) --node[above,sloped]{$\infty$} (e1);
		\draw[edge] (M4) --node[above,sloped]{$\infty$} (e2);
		\draw[edge] (M3) --node[above,sloped]{$\infty$} (e2);
		
		\draw[edge] (e1) --node[above,sloped]{$\ROe_{e_1,i}(\theta)$} (t);
		\draw[edge] (e2) --node[above,sloped]{$\ROe_{e_2,i}(\theta)$} (t);
	\end{scope}
\end{tikzpicture} 		\end{adjustbox}
		\caption{An example for the construction of the auxiliary network (right) used in the proof of \Cref{lem:CharacterizationStochasticIDERouting} for a node~$v$ with two outgoing edges (left).
		The labels on the arcs of the auxiliary network indicate their capacities.}\label{fig:NetworkMaxFlowArgument}
	\end{figure}

	We show that the maximum flow value is $1$; then setting the value $r_{M,e,i}(\theta)$ to the flow on the arc from $M$ to $e$ of any maximum flow would satisfy the requirements:
	As the maximum flow value is $1 = \sum_{M\subseteq \edgesFrom{v}} \pi_{v,M,i}(\theta,f)$ and $\pi_{v,\emptyset,i}=0$, every vertex of a subset $M$ must have incoming (and outgoing) flow of value $\pi_{v,M,i}(\theta,f)$.
	Therefore, we must have $\sum_{e\in M} r_{M,e,i}(\theta) = \pi_{v,M,i}(\theta,f)$ for every $M\neq \emptyset$.
	Furthermore, due to the capacity constraints, we must have $r_{e,i}(\theta)\geq \sum_{M\ni e} r_{M,e,i}(\theta)$ for each $e\in\edgesFrom{v}$.
	As, additionally, $\sum_{e} r_{e,i}(\theta) = 1 = \sum_{M}\sum_{e} r_{M,e,i}(\theta)$, we must have equality in the previous sentence.

	\newcommand*{\val}{\mathrm{cap}}
	We now show that the minimum $s$-$\dest$-cut of this artificial network has value $1$.
	Clearly, $1=\sum_{M}\pi_{v,M,i}(\theta,f)$ is an upper bound(by considering the cut $S=\{s\}$).
	Let $S$ be a minimal cut.
	Note, that $M\in S$ implies $e\in S$ for every $e\in M$ (as the arc from $M$ to $e$ has infinite capacity).
	We consider the set $M\coloneqq \{ e \in\edgesFrom{v} \mid e\in S \}$.
	Note that the previous observation shows that every $M'\subseteq\edgesFrom{v}$ with $M'\not\subseteq M$ is not contained in $S$ (as there is an edge $e$ in $M'$ which is not in $S$).
	By possibly removing all subsets of~$M$ and all edges in~$M$ from~$S$, we obtain the cut $\{s\}$ and the following inequality (where $\val(S)$ denotes the capacity of cut~$S$)
	\begin{equation*}\label{cuts-equal}
		\val(S) \leq \val(S\setminus ( \{ M' \mid M'\subseteq M \} \cup M ))
		\leq \val(S) + \sum_{M'\subseteq M}\pi_{v,M',i}(\theta,f) - \sum_{e\in M} r_{e,i},
	\end{equation*}
	or, equivalently, $\sum_{e\in M} r_{e,i} \leq \sum_{M' \subseteq M} \pi_{v,M',i}(\theta, f)$. As the latter equals $\rho_{v,M,i}(\theta,f)$ and as the flipped inequality is fulfilled, the inequalities in \eqref{cuts-equal} must hold with equality.
	Furthermore, as the cut in the middle of the chain is $\{ s\}$ with value $1$, the value of $S$ must also be $1$.
	
	We use the Measurable Maximum Theorem \cite[Thm. 18.19]{InfiniteDimensionalAnalysis} to show that there exists a measurable selector of the maximum $s$-$\dest$-flows, and thus, measurable functions $(r_{M,e,i})_{M,e,i}$ satisfying the constraints in $\ROp(f)$.
	To apply this theorem, we show that the map $\phi:\IR \rightrightarrows [0,1]^{\abs{A}}$, which maps a time $\theta$ to the polytope of the feasible $s$-$\dest$-flows \wrt the constraints at time $\theta$, is a measurable correspondence.
	Note that, for every $\theta$, the constraints of the polytope are of the form $g_k(x) \leq c_k(\theta)$ where $g_k$ is a continuous function from $[0,1]^{\abs{A}}$ to $\IR$ and $c_k$ is a measurable function on $\IR$.
	For each such constraint, we define the correspondence $\phi_k(\theta)=\set{ x\in[0,1]^{\abs{A}} | g_k(x) \leq c_k(\theta) }$.
	Then, $\phi_k$ is a measurable correspondence, as for every closed set $S\subseteq [0,1]^{\abs{A}}$ the set $\{ \theta \mid \exists x\in S: g_k(x) \leq c_k(\theta) \} = \{ \theta \mid \min_{x\in S} g_k(x) \leq c_k(\theta) \}$ is measurable due to the measurability of $c_k$.
	Then, $\phi$ is a measurable correspondence as it is the intersection of the correspondences $\phi_k$ (\cf \citealt[Lemma 18.4]{InfiniteDimensionalAnalysis}).
	Thus, we obtain $(r_{M,e,i})_{M,e,i}$ witnessing $r\in\ROp(f)$.
\end{proof}
}
\proofPropCharacterizationStochasticIDERouting

\newcommand{\lemmaStochasticOperatorClosedGraph}{
\begin{lemma}\label{lemma:StochasticOperatorClosedGraph}
    Let $\ROp$ be the routing operator defined in~\eqref{eq:stochastic-routing-operator} and let $\horizon\in\IRnn$.
	If the mapping $f\mapsto \hat C_p(\theta, f)$ is sequentially weak-strong continuous \wrt $(\flowSet_\horizon,\revMinor{\normFct_p})$ and $\IR$ for $\Prob_i$-almost all $\hat C$, almost all $\theta < \horizon$ and all $p \in \PathSet$, then $\ROp_\horizon$ \revised{is decomposable (\Cref{def:decomposable})}.
\end{lemma}
}
\lemmaStochasticOperatorClosedGraph
\newcommand{\proofLemStochasticOperatorClosedGraph}[1][Proof]{
\begin{proof}[#1]
	\revised{
	After restricting the constraints in~\eqref{eq:stochastic-routing-operator-alt-description} to $[0, \horizon]$, they are now in the shape of decomposable routing operators as in \eqref{eq:decomposability}.
	For the second type of constraints in \eqref{eq:stochastic-routing-operator-alt-description}, we only have to verify that the right-hand side, namely $\CharF[[0,\horizon]]\rho_{v,M,i}(\emptyarg, f)$ with $\rho$ defined in~\eqref{eq:spe-rho}, is pointwise sequentially weak-strong lower\footnote{Since the inequality is flipped, we need to verify lower- instead of upper-semicontinuity.}-semicontinuous in $f$; since $\rho_{v,M,i}$ is uniformly bounded by $1$, any sequence $(\CharF[[0,\horizon]]\rho_{v,M,i}(f^{(n)}))_n$ is trivially dominated by $\CharF[[0,\horizon]]$.
	}
    
    \begin{claim}\label{claim:RhoLowerSemicont}
    	For almost every $\theta<\horizon$, the mappings $\rho_{v,M,i}(\theta,\emptyarg): \flowSet_\horizon \to [0,1]$ are sequentially \revMinor{weak-strong} lower-semicontinuous, \ie for every sequence~$f^{(n)}$ converging weakly to some~$f$ it holds that $\liminf_{n\to\infty} \rho_{v,M,i}(\theta, f^{(n)}) \geq \rho_{v,M,i}(\theta, f)$.
    \end{claim}
    
    \begin{proofClaim}
    	Let $f^{(n)} \to f$ be a weakly convergent sequence in $\flowSet_\horizon$.
		By the \namecref{lemma:StochasticOperatorClosedGraph}'s assumption, we have $\hat C_p(\theta,f^{(n)}) \to \hat C_p(\theta,f)$ for $\Prob_i$-almost all $\hat C$ and all $p\in \PathSet$.
		Therefore, for $\Prob_i$-almost all $\hat C$, there exists some $n_{\hat C}$ such that $\tilde E_i(\theta, f^{(m)}, \hat C) \subseteq \tilde E_i(\theta,f, \hat C)$ holds for all $m\geq n_{\hat C}$.
		This means, for every $M\subseteq\edgesFrom{v}$, $n\in\IN$, and $m\geq n$ we have \begin{equation*}
			\Prob_i(\set{
				n_{\hat C} = n
				\,\land\,
				\tilde E_i(\theta, f^{(m)}, \hat C) \cap \edgesFrom{v} \subseteq M
			})
			\geq 
			\Prob_i(\set{
				n_{\hat C} = n
				\,\land\,
				\tilde E_i(\theta,f, \hat C) \cap \edgesFrom{v} \subseteq M
			}).
		\end{equation*}
		By taking the limit inferior on the left side and summing over all $n\in\IN$ we get
		\begin{align*}
			\sum_{n\in\IN} \liminf_{m\to\infty} \Prob_i(\set{
				n_{\hat C} = n
				\,\land\,
				\tilde E_i(\theta,f^{(m)}, \hat C) \cap \edgesFrom{v} \subseteq M
			})
			\geq
			\rho_{v,M,i}(\theta, f).
		\end{align*}
		The superadditivity of the limit inferior yields our claim.{}
    \end{proofClaim}
    \revised{Thus, all constraints are valid for a decomposable routing operator.}
\end{proof}
}
\proofLemStochasticOperatorClosedGraph

With this \iftrue\namecref{lemma:StochasticOperatorClosedGraph}\else lemma\fi{} we can now apply \Cref{thm:existence-finite-time-horizon} or \Cref{thm:ExtensionIfContinuous} to derive existence results for \SPEs.
For an additional uniqueness result, we show that if $\hat C_p(\theta, f)$ and $\hat C_q(\theta, f)$ coincide almost never for any two distinct $p,q\in\PathSet_{v,\dest_i}$,
\iftrue the probabilities $\pi_{v,M,i}(\theta)$ are almost always zero for any $M$ containing more than one edge and, hence, \fi
the resulting routing-operator is prescriptive.
Thus, we can apply \Cref{cor:prescriptive-operators}~\ref{cor:prescriptive-operators:lipschitz} once we show that $\ROe_\horizon$ fulfils a Lipschitz condition (under natural assumptions on $\Prob_i$).

\newcommand{\propStochasticDPEPrescriptive}{
\begin{proposition}\label{prop:stochastic-dpe-prescriptive}
	Assume $\Prob_i(\set{ \hat C_p(\theta, f)=\hat C_q(\theta, f)  })=0$ holds for all flows $f$, $p,q\in\PathSet_{v,\dest_i}$ with $p\neq q$, $i\in I$, $v\in V$ and  almost all $\theta\in\IR$.
	Then the routing operator defined in \eqref{eq:stochastic-routing-operator} is prescriptive \revised{and, for all $e=vw$, has the form}
	\[
		\revised{
		\ROe(f)_{e,i} = \pi_{v,\{e\},i}(\emptyarg, f) = \sum_{\substack{p\in \PathSet_{v,\dest_i}\\p_1 = e}} \Prob_i\big(\set{ \hat C_p(\emptyarg, f) = \min_{q\in\PathSet_{v,\dest_i}} \hat C_q(\emptyarg, f) }\big).
		}
	\]
\end{proposition}
}
\propStochasticDPEPrescriptive
\newcommand{\proofPropStochasticDPEPrescriptive}[1][Proof]{
\begin{proof}[#1]\label{proof:prop:stochastic-dpe-prescriptive}
	Under the given assumption, the set of active outgoing edges $\tilde E_i(\theta,f,\hat C)$ has exactly one element almost surely for all $i\in I$, $v\in V$ and almost all $\theta\in\IR$.
	Therefore, $\pi_{v,M,i}(\theta,f) = 0$ holds whenever $\abs{M}\neq 1$ and we have $r_{e,i}(\theta) = \pi_{v,\{e\},i}(\theta, f)$ for all $r\in \ROp(f)$.
	\revised{Further, $\pi_{v,\{e\}, i}(\theta, f)$ equals 
	\[
		\Prob_i\Big(\bigcup_{\substack{p\in\PathSet_{v,\dest_i}\\ p_1 = e}}\set{ \hat{C} | \hat C_{p}(\theta, f) = \min_{q\in\PathSet_{v,\dest_i}} \hat C_{q}(\theta, f) }\Big)
		=
		\sum_{\substack{p\in\PathSet_{v,\dest_i}\\ p_1 = e}}\Prob_i\big(\set{ \hat C_{p}(\theta, f) = \min_{q\in\PathSet_{v,\dest_i}} \hat C_{q}(\theta, f) }\big),
	\]
	since the set of predictors $\hat C$ for which two different paths attain the minimum at the same time is, by assumption, a $\Prob_i$-null set.}
\end{proof}
}
\proofPropStochasticDPEPrescriptive

\newcommand*{\ErrCost}{\hat C}
\newcommand*{\ErrZ}{\hat Z}

In particular, this means that $\ROp$ is prescriptive if for all flows $f$, paths $p\neq q$ and almost all $\theta$, the random variable
\begin{equation}\label{eq:spe-errz}
	\ErrZ_{p,q}(\theta, f)\coloneqq \hat C_p(\theta, f) - \hat C_q(\theta, f)
\end{equation}
has a probability density function (p.d.f.) \wrt every $\Prob_i$\iftrue, that is, there exists a function $\delta$ such that $\Prob_i(\set{\ErrZ_{p,q}^f(\theta, f)\in S}) = \int_S \delta \diff\leb$ holds for every measurable $S\subseteq \IR$\fi.

\newcommand{\lemmaStochasticDPELipschitz}{
\begin{lemma}\label{lemma:StochasticDPELipschitz}
	Let $\horizon \in\IRp$ and assume that there exists some $B\in\IRp$ such that $\ErrZ_{p,q}(\theta, f)$ has a p.d.f. essentially bounded by $B$ for all $f\in \flowSet_\horizon$, $i\in I$, $v\in V$ and $p,q\in\PathSet_{v,\dest_i}$ with $p\neq q$ and almost all $\theta<\horizon$.
	Furthermore, assume that for $\Prob_i$-almost all $\hat C$ and almost all $\theta\in[0,\horizon]$, the map $f\mapsto \CharF_{[0,\horizon]}\cdot \hat C(\theta, f)$ is Lipschitz continuous with common constant $L$ \wrt $(\flowSet_\horizon, \normFct_1)$ and~$\IR$. 

	Then, $\ROe_{\horizon}$ is Lipschitz continuous \wrt \revMinor{$\normFct_1$} and \revMinor{$\normFct_\revised{1}$}.
\end{lemma}
}
\lemmaStochasticDPELipschitz
\newcommand{\proofLemStochasticDPELipschitz}[1][Proof]{
\begin{proof}[#1]
	We show Lipschitz-continuity in every coordinate of $\ROe_{\horizon}$.
	Thus, let $e=vw\in E$ and $i\in I$ be given.
	For $v=\dest_i$, we have $\ROe(f)_{e,i}\equiv 0$ for all $f\in\flowSet_\horizon$.
	Hence, going forward, we assume $v\neq \dest_i$.
	\begin{claim}\label{claim:path-differences-lipschitz}
		There exists $L' > 0$ such that for all $f, g\in \flowSet_\horizon$, $p,q\in\PathSet_{v, \dest_i}$ and almost all $\theta\in[0,\horizon]$ it holds \optDisplay{
			\Prob_i\left(
				\set{ \ErrCost_{p}(\theta, f) \leq \ErrCost_{q}(\theta, f)}
				\symDiff
				\set{ \ErrCost_{p}(\theta, g) \leq \ErrCost_{q}(\theta, g)}
			\right)
			\leq L'\cdot \norm{f - g}_1,
		}
		\revised{where $\symDiff$ denotes the symmetric difference, defined as $A\symDiff B\coloneqq (A\setminus B) \cup (B\setminus A)$.}
	\end{claim}
	\begin{proofClaim}
		The statement is trivial for $p=q$.
		By expanding the definition of the symmetric difference, the left-hand side equals
		\begin{align}\label{spe:claim-lipschitz-split}
			\Prob_i\left(
				\begin{aligned}
				\set{ \ErrZ_{p,q}(\theta, f) \leq 0 < \ErrZ_{p,q}(\theta, g) }
				\\
				\cup
				\set{ \ErrZ_{p,q}(\theta, f) > 0 \geq \ErrZ_{p,q}(\theta, g) }
				\end{aligned}
			\right).
		\end{align}
		We bound the probability of each of these sets separately (but analogously).
		Thus, assume \wlg that $\ErrZ_{p,q}(\theta, f) \leq 0 < \ErrZ_{p,q}(\theta, g)$ holds.
		Clearly, $h\mapsto \ErrZ_{p,q}(\theta, h)$ is Lipschitz continuous with constant $2\cdot L$.
		Therefore, $\ErrZ_{p,q}(\theta, g) - \ErrZ_{p,q}(\theta, f) \leq 2\cdot L \cdot \norm{f - g}_1$ and thus $0 < \ErrZ_{p,q}(\theta, g) \leq 2\cdot L\cdot \norm{f - g}_1$.
		In particular, $\{ \ErrZ_{p,q}(\theta, f) \leq 0 < \ErrZ_{p,q}(\theta, g)\}\subseteq \{ \ErrZ_{p,q}(\theta, g) \in [0,2L\norm{f-g}_1 \}$ and thus, the probability of the former set is bounded by $2\cdot B\cdot L\cdot \norm{f-g}_1$.
		Doing this for both sets in \eqref{spe:claim-lipschitz-split}, the considered probability of the claim is bound by $L'\cdot \norm{f-g}_1$ with $L'\coloneqq 4\cdot B\cdot L$.
	\end{proofClaim}

	\begin{claim}\label{claim:path-minimum-lipschitz}
		There exists $L'' > 0$ such that for $f,g\in \flowSet_\horizon$, $p\in\PathSet_{v,\dest_i}$ and almost all~$\theta$ it holds that 
		\begin{equation*}
			\Prob_i\left(
				\begin{aligned}
					\set{ \forall q\in \PathSet_{v,\dest_i}: \hat C_{p}(\theta, f) \leq \ErrCost_{q}(\theta, f)}
					\symDiff
					\set{ \forall q\in \PathSet_{v,\dest_i}: \hat C_{p}(\theta, g) \leq \hat C_{q}(\theta, g)}\end{aligned}\right)
			\leq L''\cdot \norm{f-g}_1.
		\end{equation*}
	\end{claim}
	\begin{proofClaim}
		Let $L'$ denote the constant given by \Cref{claim:path-differences-lipschitz}.
		Because the symmetric difference of intersections $(\bigcap_{j\in J} A_j) \symDiff (\bigcap_{j\in J} B_j)$ is a subset of $\bigcup_{j\in J} (A_j\symDiff B_j)$ (for arbitrary $J$, $A_j$, $B_j$), the described probability is bounded from above by
		\begin{equation*}
			\sum_{q\in \PathSet_{v,\dest_i}} \Prob_i\left(
					\set{ \hat C_{p}(\theta, f) \leq \hat C_{q}(\theta, f)}
					\symDiff \set{ \hat C_{p}(\theta, g) \leq \hat C_{q}(\theta, g)}
			\right)
			\leq \abs{\PathSet_{v,\dest_i}} \cdot L'\cdot \norm{f-g}_1. \qedhere
		\end{equation*}
	\end{proofClaim}
	
	By (the proof of) \Cref{prop:stochastic-dpe-prescriptive}, we know $\pi_{v,M,i}(\theta,f)=0$ for $\abs{M}\neq 1$ and therefore $r_{e,i}(f)(\theta)$ equals $\pi_{v,\{e\}, i}(\theta, f)$, \ie the probability that $\{e\}$ is the set of perceived active outgoing edges of $v$.
	This is exactly the probability that there exists a path in $\PathSet_{v,\dest_i}$ starting with $e$ that minimizes the perceived cost $\ErrCost_{p,i}(\theta, f)$ over all alternatives in $\PathSet_{v,\dest_i}$:
	\begin{align*}
		r_{e,i}(f)(\theta) = \Prob_i\bigl(
			\bigcup_{\substack{p\in\PathSet_{v,\dest_i}\\ p_1=e}} \set{ \forall q\in\PathSet_{v,\dest_i}: \hat C_{p}(\theta, f) \leq \hat C_{q}(\theta, f)}
		\bigr).
	\end{align*}
	Now, let $f, g\in M$ be arbitrary.
	Using the fact that $\abs{\Prob(A) - \Prob(B)} \leq \Prob(A\symDiff B)$ holds for arbitrary $A$ and $B$, \Cref{claim:path-minimum-lipschitz} as well as the fact that $(\bigcup_{j\in J} A_j) \symDiff (\bigcup_{j\in J} B_j) \subseteq \bigcup_{j\in J} (A_j \symDiff B_j)$ holds for arbitrary $J$, $A_j$, $B_j$, we deduce (for almost all $\theta\in[0,\horizon]$) that
	\begin{equation*}
		\abs{r_{e,i}(f)(\theta) - r_{e,i}(g)(\theta)}
		\leq \sum_{\substack{p\in\PathSet_{v,\dest_i}\\ p_1=e}} L'' \cdot \norm{f-g}_1
		\leq \abs{\PathSet_{v,\dest_i}}\cdot L'' \cdot \norm{f-g}_1
		\eqqcolon L'''\cdot \norm{f-g}_1.
	\end{equation*}
	Clearly, this implies \optDisplay{
		\norm{\CharF_{[0,\horizon]}\cdot r_{e,i}(f) - \CharF_{[0,\horizon]}\cdot r_{e,i}(g)}_\revised{1}
		\leq \horizon\cdot L'''\cdot \norm{f-g}_1. \qedhere
	}
\end{proof}
}
\proofLemStochasticDPELipschitz

\revised{
Let us now apply the theory from~\Cref{sec:ExistenceUniqueness} and summarize our findings for \SPEs{} in the following theorem:
}

	\begin{theorem}\label{cor:stochastic-prediction-equilibrium}
		Let $\EdgeLoading$ be a locally bounded edge loading operator and $\Prob_i$ be probability measures on cost predictors.
		\begin{enumerate}[label=(\roman*)]
			\item\label{cor:stochastic-prediction-equilibrium:finite-existence}
			If, for some $\horizon\in\IRnn$, $\EdgeLoading_\horizon$ is sequentially weak-weak continuous \revMinor{\wrt $\normFct_p$} and $f\mapsto \hat C_p(\theta, f)$ is sequentially weak-strong continuous \wrt $(\flowSet_\horizon,\revised{\normFct_p})$ and $\IR$ for $\Prob_i$-almost all $\hat C$ for almost all $\theta<\horizon$, then there exists a \SPE{} until time $\horizon$.
			\item\label{cor:stochastic-prediction-equilibrium:infinite-existence}
			If \revMinor{the assumptions of~\ref{cor:stochastic-prediction-equilibrium:finite-existence} hold for all $\horizon\in\IRnn$} and if $\EdgeLoading$ as well as $\Prob_i$-almost all $\hat C$ are causal, then there exists a \SPE{} until time $\infty$.
			\item\label{cor:stochastic-prediction-equilibrium:unique-existence}
			If $\EdgeLoading$ is uniformly strictly causal and essentially bounded, \revMinor{the network inflow rates} $u$ \revMinor{are contained in} $\pLocInt[\infty][\IR]^{V\times I}$, $\hat C$ is causal for $\Prob_i$-almost all $\hat C$, there exists some $B\in\IRp$ such that for every two distinct paths $p, q$, flow $f$, almost all $\theta$, the random variable $\ErrZ_{p,q}(\theta, f)$ \revised{in~\eqref{eq:spe-errz}} has a p.d.f. bounded by $B$ (\wrt every $\Prob_i$),
			and if for all $\horizon\in\IRnn$ there exists some $L\in\IRp$ such that for $\Prob_i$-almost all $\hat C$, the mapping $f\mapsto \CharF_{[0,\horizon]}\cdot \hat C(\theta, f)$ is Lipschitz continuous with constant $L$ \wrt $(\flowSet_\horizon, \revMinor{\normFct_1})$ and $\IR$ for almost all $\theta\in[0,\horizon]$, then there exists a unique \SPE{} until~$\infty$.
		\end{enumerate}
	\end{theorem}
\newcommand{\proofCorStochasticPredictionEquilibrium}[1][Proof]{
\begin{proof}[#1]
	\revised{First, note that the corresponding routing operator $\ROp$ in~\eqref{eq:stochastic-routing-operator} has non-empty values:
	For this, we can consider the witness $r$, defined by $r_{vw,i} = 0$, if $v=\dest_i$ and \[
		r_{vw,i}\coloneqq \sum_{M\subseteq \edgesFrom{v}: e\in M} \frac{ \pi_{v,M,i}(\cdot, f) }{ {\abs{M}} },
	\]
	if $v\neq \dest_i$, which is well-defined since $\dest_i$ is assumed to be reachable from $v$, and the sum $\sum_{e\in\edgesFrom{v}} r_{e,i}$ is $1$ at all times, since for all times $\theta$ and $\Prob_i$-almost all predictors $\hat{C}$, the perceived set of active edges $M$ is non-empty,subset $M$ of $\edgesFrom{v}$ that is the perceived set of active edges, and thus $\sum_{M\subseteq\edgesFrom{v}: M\neq \emptyset} \pi_{v,M,i}(\theta, f) = 1$.}

	For \ref{cor:stochastic-prediction-equilibrium:finite-existence}{} and~\ref{cor:stochastic-prediction-equilibrium:infinite-existence} {} we \revMinor{can now} apply \Cref{thm:existence-finite-time-horizon} and \Cref{thm:ExtensionIfContinuous}, respectively.
	In both cases, \Cref{lemma:StochasticOperatorClosedGraph} shows that $\ROp_\horizon$ \revised{is decomposable}.
	For \ref{cor:stochastic-prediction-equilibrium:unique-existence} note that the routing operator is prescriptive (\Cref{prop:stochastic-dpe-prescriptive}) and thus~\ref{cor:stochastic-prediction-equilibrium:unique-existence} follows from \Cref{cor:prescriptive-operators}~\ref{cor:prescriptive-operators:lipschitz} where \Cref{lemma:StochasticDPELipschitz} proves the required Lipschitz property.
\end{proof}
}
\proofCorStochasticPredictionEquilibrium

	We conclude by applying the gained insights to the stochastic IDE model.

	\begin{corollary}\label{cor:stochastic-ide}
		Let $\EdgeLoading$ be a locally bounded edge loading operator and let $\varepsilon_{e,i}$ denote the random variables of the stochastic IDE model.
		\begin{enumerate}[label=(\roman*)]
			\item\label{cor:stochastic-ide:finite-existence} If, for some $\horizon\in\IRnn$, $\EdgeLoading_\horizon$ is sequentially weak-weak continuous \revMinor{\wrt $\normFct_p$} and $f\mapsto C_e(\theta, f)$ is sequentially weak-strong continuous \wrt $(\flowSet_\horizon, \revMinor{\normFct_p})$ and $\IR$ for almost all $\theta<\horizon$, then there exists a stochastic IDE until time $\horizon$.
			\item\label{cor:stochastic-ide:infinite-existence} If \revMinor{the assumptions of~\ref{cor:stochastic-ide:finite-existence} hold for all $\horizon\in\IRnn$} and if $\EdgeLoading$ and $C_e$ are causal, then there exists a stochastic IDE until time $\infty$.
			\item\label{cor:stochastic-ide:unique-existence} If $\EdgeLoading$ is uniformly strictly causal and essentially bounded, $u\in\pLocInt[\infty][\IR]^{V\times I}$, $C_e$ is causal, and all $\varepsilon_{e,i}$ are independent with bounded probability densities, and the mapping $f \mapsto \CharF_{[0, \horizon]}\cdot C_e(\theta, f)$ is Lipschitz continuous with constant $L$ \revMinor{\wrt $\normFct_1$ and $\IR$} for almost all $\theta\in[0,\horizon]$, then there exists a unique stochastic IDE until $\infty$.
		\end{enumerate}
	\end{corollary}
\newcommand{\proofCorStochasticIDE}[1][Proof]{
\begin{proof}[#1]
	Statements~\ref{cor:stochastic-ide:finite-existence}{} and~\ref{cor:stochastic-ide:infinite-existence} follow{} directly from \Cref{cor:stochastic-prediction-equilibrium}~\ref{cor:stochastic-prediction-equilibrium:finite-existence}{} and \Cref{cor:stochastic-prediction-equilibrium}~\ref{cor:stochastic-prediction-equilibrium:infinite-existence}, respectively.
	For statement~\ref{cor:stochastic-ide:unique-existence}, we need to show that for distinct paths $p,q$ the random variable $\ErrZ_{p,q}(\theta, f)$ has a p.d.f. bounded by some fixed $B$ (\wrt every $\Prob_i$) in order to apply \Cref{cor:stochastic-prediction-equilibrium}~\ref{cor:stochastic-prediction-equilibrium:unique-existence}.
	However, this is true since $\ErrZ_{p,q}(\theta, f)$ is the sum of a constant and a non-trivial linear combination of independent random variables with bounded probability densities.
\end{proof}
}
\proofCorStochasticIDE

\begin{remark}\label{remark:cumulative-flow}
	Since the prediction functions (as well as the routing operators) are often defined in terms of the cumulative flow functions (\ie $F^{\pm}_{e,i}: \theta \mapsto \int_0^\theta f^{\pm}_{e,i}\diff\leb$) instead of the flow rates, it can be helpful to observe that the continuity assumption on the mapping $f \mapsto \CharF_{[0, \horizon]}\cdot \ttime_e(\theta, f)$ in \Cref{cor:stochastic-ide}~\ref{cor:stochastic-ide:unique-existence} can be replaced by the following (slightly stronger) assumption:
	The map $F \mapsto \ttime_e(\theta, \dot F)$ is Lipschitz continuous from the space of absolutely continuous functions (with the uniform norm) to $\IR$ for almost all $\theta \in [0,\horizon]$.
	This is, because for any two functions $f,g \in L^1([a,b])$ we have 
		\begin{align*}
			\norm{\int f\diff\leb - \int g \diff\leb}_\infty
			&= \sup_{\theta \in [a,b]}\abs{\int_a^\theta f\diff\leb - \int_a^\theta g\diff\leb}
			\leq \sup_{\theta \in [a,b]}\int_a^\theta \abs{f - g}\diff\leb
			\\
			&\leq \int_a^b \abs{f - g}\diff\leb
			= \norm{f-g}_1.
		\end{align*}
	Similar adjustments can be made in \Cref{cor:stochastic-prediction-equilibrium}~\ref{cor:stochastic-prediction-equilibrium:unique-existence}, \revMinor{in} \Cref{cor:prescriptive-operators}~\ref{cor:prescriptive-operators:lipschitz} and \revMinor{in} \Cref{thm:unique-existence-for-lipschitz-operators}~\ref{cond:RoutingOperator-Lipschitz}.
\end{remark}

\section{Conclusion}

In this paper we introduced a general framework for the analysis of dynamic traffic assignment with adaptive route choice that incorporates both the well-known descriptive behavioural models stemming from game-theoretic perspectives, such as the (full-information) Nash equilibrium and the dynamic prediction equilibrium models, as well as prescriptive models including those that assume a logit-based route choice.
We showed that, under certain continuity assumptions on the edge loading and routing operators, a coherent flow exists up to any finite time horizon and, under the additional assumption of causality, it exists on the whole $\IR$.
Moreover, we can guarantee uniqueness of the coherent flow if the operators are uniformly strictly causal or fulfil a Lipschitz-condition.

Since our framework uses abstract edge loading and routing operators, that can be instantiated with various models from the literature, our results generalize established findings including the existence of dynamic Nash flows, dynamic prediction equilibria, and the existence and uniqueness of coherent flows with prescriptive routing operators.
The logit-based operators of the latter model class are motivated by stochastic noise in the predictions of the perceived travel times that may vary from agent to agent.
We modeled these stochastic effects descriptively as what we call the stochastic prediction equilibrium and showed that such an equilibrium exists under mild assumptions on the continuity of the prediction functions.
If we assume that the distribution of the noise has a bounded probability density function, then stochastic prediction equilibria are even unique.

\paragraph{Open Questions}

Concerning the special case of stochastic IDE{} it is not hard to see that if the distribution of the measurement errors is supported on a small interval around zero, then the resulting stochastic IDE is also an $\varepsilon$-approximate IDE in the sense of \cite{PredictionEquilibria}.
Hence, a natural question is whether a sequence of such stochastic IDE converges to an exact IDE if the support of the measurement error distribution shrinks to zero.
If this is the case and if stochastic IDE are unique, then this limit point would, in some sense, be a canonical choice within the set of IDE for a given instance. Furthermore, our existence result for stochastic IDE even provides a natural way of numerically computing such equilibria, as the Banach Fixed Point theorem used for the extension step also guarantees that the standard fixed point iteration converges to a fixed point.

Finally, even though our framework is already quite general, there are still certain, more complex phenomena in traffic assignment for which it is not clear upfront whether they can be incorporated in our model. One such aspect would be spillback effects, where congestion on one road section leads to congestion on a previous road section (see, e.g., \cite{SpillbackSeringKoch}). Other potential extensions would be allowing departure time choice or elastic demand scenarios. 

\paragraph{Acknowledgements} We thank Julian Schwarz and Alexander Keimer for helpful discussions on the topics of this paper. We further thank Julian Schwarz for proofreading. Finally, we thank the anonymous reviewers for their feedback, and one reviewer in particular for spotting an error in the initially submitted version.

\paragraph{Funding} This work was partially funded by the German Federal Ministry of Research, Technology and Space under Grant No.~05M22WPA.

\clearpage

	\printbibliography[heading=bibintoc]

\clearpage

\appendix

\ifnoappendix\else

\section{Some Technical Results from Topology}

\begin{revisedEnv}
The following results regarding topological properties of $L^p$-spaces are well-known. For the convenience of the reader, we still restate them here in the form and using the notation as needed in this paper. For completeness we also provide short proofs.

	\begin{proposition}\label{prop:p-integrable-bounded-codomain}
		Let $E\subset \IR$ be a set of finite measure, and let $p\in (1,\infty)$.
		Then, on the set~$S$ of measurable functions with domain $E$ and codomain $[0,1]$, the topologies induced by $\normFct_p$ and $\normFct_1$ coincide.

		Further, $\norm{f - g}_1 \leq \norm{f - g}_p\cdot \leb(E)^{1/q}$ and $\norm{f-g}_p \leq \norm{f - g}_1^{1/p}$ hold for all $f,g\in S$, where $q$ is the conjugate of $p$ satisfying $\nicefrac{1}{p} + \nicefrac{1}{q} = 1$.
	\end{proposition}
	\begin{proof}
		Since $E$ is of finite measure, $S$ is a subset of both $L^1(E)$ and $L^p(E)$, and thus, both $(S, \normFct_1)$ and $(S, \smallnorm{\emptyarg_p})$ are metric spaces.
		Therefore (\cf \cite[Theorem~2.40]{Aliprantis2006a}), their topologies can be described by sequences.

		Assume a sequence $(f^{(n)})_n$ converges in $(S,\normFct_p)$ to $f$.
		Then, by Hölder's inequality, we have $\smallnorm{f^{(n)} - f}_1 \leq \smallnorm{f^{(n)} - f}_p \cdot \smallnorm{\CharF[E]}_q = \smallnorm{f^{(n)} - f}_p \cdot \leb(E)^{1/q}$.
		Thus, $(f^{(n)})_n$ also converges in $(S,\normFct_1)$ to $f$.

		Now, assume $(f^{(n)})_n$ converges in $(S, \normFct_1)$ to $f$.
		Then, since $\abs{f^{(n)} - f}$ is bounded almost everywhere by $1$, we have 
		$\abs{f^{(n)} - f}^p \leq \abs{f^{(n)} - f}$.
		This implies $\smallnorm{f^{(n)} - f}_p \leq \smallnorm{f^{(n)} - f}_1^{1/p}$ which converges to $0$.
	\end{proof}

	\begin{proposition}\label{prop:product-of-weak-and-strong}
		Let $p,q\in(1,\infty)$ with $\nicefrac{1}{p}+\nicefrac{1}{q}=1$.
		Let $(f^{(n)})_n \subseteq \pInt[q]$ and $(g^{(n)})_n \subseteq \pInt[p]$ be two sequences and $f \in \pInt[q]$ and $g \in \pInt[p]$ two elements.
		If $(f^{(n)})_n$ converges strongly to $f$ (in $\pInt[q]$) and $(g^{(n)})_n$ converges weakly to $g$ (in $\pInt[p]$), then $(f^{(n)} \cdot g^{(n)})_n$ converges weakly to $f\cdot g$ in $\pInt[1]$.
	\end{proposition}
	\begin{proof}
		Let $h \in \pInt[\infty]$ be arbitrary. We need to show that then
			\[\int f^{(n)} \cdot g^{(n)} \cdot h \diff\lambda \to \int f\cdot g \cdot h \diff\lambda.\]
		We have
			\begin{align*}
				&\abs{\int f^{(n)} \cdot g^{(n)} \cdot h \diff\lambda - \int f\cdot g \cdot h \diff\lambda} \\
				&\quad \leq \abs{\int f^{(n)} \cdot g^{(n)} \cdot h - f\cdot g^{(n)}\cdot h\diff\lambda} + \abs{\int f\cdot g^{(n)}\cdot h\diff\lambda - \int f\cdot g \cdot h \diff\lambda}.
			\end{align*}
		Here, then, the latter term goes to $0$ since $g^{(n)}$ converges weakly to $g$ and $f\cdot h \in L^q$. The former term we can further bound using Hölder's inequality:
			\begin{align*}
				\abs{\int f^{(n)} \cdot g^{(n)} \cdot h - f\cdot g^{(n)}\cdot h\diff\lambda} 
					\leq \int \abs{f^{(n)} - f}\cdot\abs{g^{(n)}}\cdot\abs{h}\diff\lambda 
					\leq \norm{f^{(n)}-f}_q \cdot \norm{g^{(n)}}_p \cdot \norm{h}_\infty.
			\end{align*}
		Here, then, the first factor goes to zero since $f^{(n)}$ converges strongly to $f$, the second is bounded as $g^{(n)}$ is a weakly converging sequence (see \cite[Chapter~8, Theorem~7]{Royden2010}), and the third term is a constant.
	\end{proof}

	\begin{proposition}\label{prop:inequalities-weakly-continuous}
		Let $p\in[1,\infty)$, $E\subseteq \IR$ measurable, and let $f^{(n)}$, $g^{(n)}$ be two sequences that converge weakly in $L^p(E)$ to $f$ and $g$, respectively, such that $f^{(n)} \leq g^{(n)}$ holds almost everywhere for all $n\in \IN$.
		Then, $f\leq g$ holds almost everywhere.
	\end{proposition}
	\begin{proof}
		Note that $h^{(n)}\coloneqq g^{(n)} - f^{(n)}$ converges weakly in $L^p(E)$ to $h\coloneqq g - f$ and $h^{(n)}$ is non-negative almost everywhere.
		We have to show that $h$ is also non-negative almost everywhere.
		Thus, assume for contradiction that the set $A \coloneqq \{ \theta \in E \mid h(\theta) < 0 \}$ has positive measure, and let $B$ be a subset of $A$ of positive, but finite measure.
		Since $\CharF[B]$ is in $L^q(E)$ (where $q$ is the conjugate of $p$ satisfying $\nicefrac{1}{p} + \nicefrac{1}{q} = 1$), $h^{(n)}$ converges weakly, we have \[
			0 \leq \lim_{n\to \infty} \int_E h^{(n)} \cdot \CharF[B] \diff\leb = \int_B h \diff\leb < 0,
		\]
		a contradiction.
	\end{proof}
\end{revisedEnv}

\begin{proposition}\label{prop:weakly-metrizable}
	Let $X$ be a reflexive, separable Banach space over $\IR$.
	Let $K\subseteq X$ be a norm bounded subset, \ie there is some $B\in\IR$ with $\norm{x} \leq B$ for all $x\in K$.
	Then, $K$ is weakly metrizable, \ie the subspace topology on $K$ induced by the weak topology on $X$ is metrizable.
	In particular:
	\begin{enumerate}[label=(\roman*)]
		\item\label{prop:weakly-metrizable:weakly-closed-iff-seq-weakly-closed} A set $K'\subseteq K$ is weakly closed in $K$ if and only if it is sequentially weakly closed in $K$.
		\item\label{prop:weakly-metrizable:weakly-closed-in-X-iff-seq-weakly-closed-in-X} If $K$ is closed, then $K'\subseteq K$ is weakly closed in $X$ if and only if it is sequentially weakly closed in $X$.
	\end{enumerate}
\end{proposition}
\begin{proof}
	\newcommand*{\ball}{U}
	Clearly, as $X$ is reflexive, its continuous dual $X'$ is reflexive as well.
	As the dual of $X'$ is the separable space $X$, $X'$ is also separable (\cf \citealt[Theorem~4.6-8]{Kreyszig1989}).
	Therefore, the closed unit ball $\ball_1(0)$ in $X$ is weakly metrizable (\cf \citealt[Theorem~6.31]{InfiniteDimensionalAnalysis}, \citealt[Theorem~3.35]{Aliprantis2006a}).
	Then, the closed ball around $0\in X$ with radius $B$ is also weakly metrizable (by scaling the arguments of the metric with $1/B$).
	Restricting this metric to $K$ gives the desired metric.

	For statement \ref{prop:weakly-metrizable:weakly-closed-iff-seq-weakly-closed}, see \cite[Lemma~3.3]{InfiniteDimensionalAnalysis};
	\ref{prop:weakly-metrizable:weakly-closed-in-X-iff-seq-weakly-closed-in-X} follows directly from  \ref{prop:weakly-metrizable:weakly-closed-iff-seq-weakly-closed}.
\end{proof}

\clearpage

\clearpage
\section{Comments on \cite[Lemma 3.3]{Bayen2019}}\label{sec:bayen-lemma}

As mentioned in \Cref{sec:uniqueness}, we briefly discuss counterexamples for the statement in \cite[Lemma~3.3]{Bayen2019} which serves the authors as a building block for their uniqueness and existence proof of coherent flows with Lipschitz-continuous, prescriptive routing operators (see \citealt[Proof of Theorem~3.4]{Bayen2019}).
The lemma is stated as follows:

\begin{lemma}[{\citealt[Lemma 3.3]{Bayen2019}}]\label{lem:BayenUniquenessLemma}
	Let $n\in\IN$, $L,\horizon\in \IRp$, and $\Omega\subseteq C([0, \horizon ], \IR^n)$ closed in the induced topology be given.
	Let $\Psi: \Omega \to \Omega$ be a Lipschitz-continuous self-mapping for which there exists some $\alpha > 0$ such that for all $\theta\in[0,\horizon]$ and $x,\tilde x\in\Omega$ we have
	\[
		\norm{\restr{\Psi(x)}{[0,\theta]} - \restr{\Psi(\tilde x)}{[0,\theta]}}_{\revMinor{\infty}}
		\leq L\cdot \theta^\alpha\cdot \norm{\restr{x}{[0,\theta]} - \restr{\tilde x}{[0,\theta]}}_{\revMinor{\infty}}.
	\]
	Then, there exists a unique fixed-point $x^*\in\Omega$ of $\Psi$.
\end{lemma}

However, the assumptions of this \namecref{lem:BayenUniquenessLemma} neither guarantee existence of a fixed point nor their uniqueness.

\paragraph{Counterexample for existence}
We consider the following counterexample:
We define $n \coloneqq 1$, $\Omega \coloneqq C([0,1], \IRnn)$, and $\Psi$ as the mapping 
\[
	\Psi: \Omega \to \Omega, ~ x \mapsto (\theta \mapsto \theta \cdot x(\theta) + 1).
\]
Then, $\Omega$ is closed in $C([0, 1], \IR)$ and $\Psi$ is a Lipschitz-continuous self-mapping with Lipschitz-constant~$1$.
Moreover, for any $\theta\in[0,1], x,\tilde x\in\Omega$ we have
	\begin{equation*}
		\norm{\restr{\Psi(x)}{[0,\theta]} - \restr{\Psi(\tilde x)}{[0,\theta]}}_{\revMinor{\infty}}
		= \sup_{\theta'\in[0,\theta]} \abs{\theta'\cdot (x(\theta') - \tilde x(\theta'))}
		\leq \theta \cdot \norm{\restr{x}{[0,\theta]} - \restr{\tilde x}{[0,\theta]}}_{\revMinor{\infty}}.
	\end{equation*}
Hence, choosing $L\coloneqq \alpha\coloneqq 1$, $\Psi$ fulfils the assumptions of the lemma.

Now, assume that $x^*$ is a fixed point of $\Psi$.
Then, for all $\theta\in[0,1]$ we must have $x^*(\theta) = \theta\cdot x^*(\theta) + 1$, which implies $0 = 1$ for $\theta = 1$, a contradiction.
Thus $\Psi$ does not have any fixed point.

\paragraph{Counterexample for uniqueness}
Again, we define $n \coloneqq 1$, $\Omega \coloneqq C([0,1], \IRnn)$.
Here, $\Psi$ is the mapping 
\[
	\Psi: \Omega \to \Omega, ~ x \mapsto \theta \mapsto \begin{cases}
		0, &\text{if $\theta < \frac{1}{2}$}, \\
		\min\{\theta - \frac{1}{2}, x(\theta)\}, &\text{if $\theta \geq \frac{1}{2}$.}
	\end{cases}
\]
Clearly, $\Omega$ is closed in $C([0, 1], \IR)$ and $\Psi$ is a Lipschitz-continuous self-mapping with Lipschitz-constant~$1$.
Furthermore, for $\theta\in[0,\frac{1}{2}], x,\tilde x\in\Omega$ we have $\smallnorm{\restr{\Psi(x)}{[0,\theta]} - \restr{\Psi(\tilde x)}{[0,\theta]}} = 0$ and for $\theta\in(\frac{1}{2}, 1]$ we have \begin{align*}
	\norm{\restr{\Psi(x)}{[0,\theta]} - \restr{\Psi(\tilde x)}{[0,\theta]}}_{\revMinor{\infty}}
	&= \sup_{\theta'\in[\frac{1}{2},\theta]} \abs{\min\{\theta' - \frac{1}{2}, x(\theta')\} - \min\{\theta'- \frac{1}{2}, \tilde x(\theta')\}}
	\\& \leq \norm{\restr{x}{[0,\theta]} - \restr{\tilde x}{[0,\theta]}}_{\revMinor{\infty}}
	\leq 2\cdot \theta^1 \cdot \norm{\restr{x}{[0,\theta]} - \restr{\tilde x}{[0,\theta]}}_{\revMinor{\infty}}.
\end{align*}
Choosing $L\coloneqq 2$ and $\alpha\coloneqq 1$, $\Psi$ again fulfils the assumptions of the lemma.
However, for every $\revMinor{\mu}\in[0,1]$, the function $x^*(\theta)\coloneqq \revMinor{\mu}\cdot \CharF_{[\frac{1}{2}, 1]}(\theta)\cdot (\theta-\frac{1}{2})$ is a fixed point of $\Psi$.

\paragraph{Comment}

As we have seen, the assumptions of the lemma do not guarantee existence nor uniqueness of a fixed point of $\Psi: \Omega\to\Omega$.
However, since the continuity assumption on $\Psi$ implies that it is causal, \ie for all $x,\tilde x\in\Omega$ and $\theta\in[0,\horizon]$ we have that $\restr{x}{[0,\theta]} = \restr{\tilde x}{[0,\theta]}$ implies $\restr{\Psi(x)}{[0,\theta]} = \restr{\Psi(\tilde x)}{[0, \theta]}$, we can  ``restrict'' $\Psi$ to functions defined on smaller subintervals $[0,\horizon'] \subseteq [0,\horizon]$. Then, one can show that this restriction $\Psi'$ is guaranteed to have a unique fixed point for small enough intervals $[0,\horizon']$. 
More specifically, for any $\horizon' < (\frac{1}{L})^{1/\alpha}$, we define $\Psi'$ as the mapping \[
	\Psi': \restr{\Omega}{[0,\horizon']} \to \restr{\Omega}{[0,\horizon']}, f\mapsto \restr{\Psi(\tilde f)}{[0,\horizon']}
\]
where $\restr{\Omega}{[0,\horizon']}\coloneqq \{ \restr{f}{[0,\horizon']} \mid f\in\Omega \}$ and $\tilde f$ is any extension of $f$ in $\Omega$ (by causality, the exact choice of $\tilde f$ does not matter).
Then, $\Psi'$ has a unique fixed point in $\restr{\Omega}{[0,\horizon']}$ by the Banach Fixed-Point Theorem, as (by the assumptions of the \namecref{lem:BayenUniquenessLemma} and the choice of $'$) $\Psi'$ is a contraction with contraction constant $L\cdot \horizon'^\alpha < 1$.\footnote{For example, for the $\Psi$ defined in the existence-counterexample, the unique fixed point of $\Psi'$ (for any $\horizon'<1$) is given by $x^*(\theta)=1/(1-\theta)$.}

\Textcite{Bayen2019} also used this observation in the first (correct) part of their proof of \Cref{lem:BayenUniquenessLemma}. 
However, they then (incorrectly) claim that this step may be applied iteratively in order to ``extend'' the fixed point of $\Psi'$ to a fixed point of $\Psi$. Hence, under the assumptions of this \namecref{lem:BayenUniquenessLemma} its conclusion only holds for restrictions of~$\Psi$ to small enough intervals.

Nevertheless, as we show here in \Cref{thm:uniqueExtIfContraction}, this weaker version of \Cref{lem:BayenUniquenessLemma} still suffices to show the unique extension property as well as the extension existence property and, hence, the unique existence of coherent flows.
Thus, the conclusion \citeauthor{Bayen2019} draw from \Cref{lem:BayenUniquenessLemma} (namely \citealt[Theorem~3.4]{Bayen2019}) remains true (with some adjustments of the proof).

\clearpage

\section{List of Symbols}\label{sec:list-of-symbols}

\newenvironment{lostable}{\renewcommand{\arraystretch}{1.15}\begin{tabular}{p{.15\textwidth}p{.75\textwidth}}Symbol & Description \\\hline
}{\end{tabular}}
\newcommand{\lossection}[1]{\hline\multicolumn{2}{l}{\textbf{#1:}}\\\hline}
\newcommand{\losentry}[3][]{\ensuremath{#2} & #3 \ifthenelse{\equal{#1}{}}{}{(\cf \Cref{#1})}\\}

\par
\begin{lostable}
	\losentry{\IRnnInf}{the extended non-negative reals $\IRnn \cup \{\infty\}$}
	\losentry{\CharF[J]}{the indicator function of a set $J \subseteq \IR$}
	\losentry{\CharF[v \neq \dest_i]}{equal to $1$ if $v \neq \dest_i$ and to $0$ otherwise}
	\losentry{\pInt,\pLocInt}{set of (locally) $p$-integrable functions on $\IR$ (with $1 < p< \infty$)}
	\losentry{\norm{f}_p}{$p$-norm of $f$ (with $1 \leq p < \infty$)}
	\losentry{\int_J f\diff \leb}{integral of $f$ over $J \subseteq \IR$ with respect to the Lebesgue measure~$\leb$}
	\losentry{\foraall \theta}{for almost all $\theta\in\IR$}
	\losentry{\horizon}{a time horizon $\horizon \in \IRnnInf$}
	
	\losentry{\rateFcts, \rateFcts_T}{set of \revised{flow} rate functions, \ie the set of all non-negative, locally $p$-integrable functions in $\pLocInt$ with (essential) support in $\IRnn$ and in $[0,\horizon]$, respectively.}
	\losentry{\revised{\splitFcts, \splitFcts_\horizon}}{\vphantom{}\revised{set of measurable functions from $\IR$ to $[0,1]$ with (essential) support in $\IR$ and $[0,\horizon]$, respectively.}}
	
	\losentry{G=(V,E)}{directed graph with node set~$V$ and edge set~$E$}
	\losentry{\edgesFrom{v}, \edgesTo{v}}{set of edges leaving and entering node $v$, respectively}
	\losentry{\PathSet_{v,w}}{set of simple \stpath[v][w]s}
	
	\losentry{I}{finite set of commodities}
	\losentry{u_{v,i}}{network inflow rate of commodity~$i$ at node $v$}
	\losentry{\dest_i}{destination node of commodity~$i$}
	\losentry{f = (f^+_{e,i},f^-_{e,i})}{dynamic flow consisting of edge inflow rates~$f^+_{e,i}$ and edge outflow rates~$f^-_{e,i}$ for every edge $e \in E$ and every commodity $i \in I$}
	
	\losentry{\EdgeLoading, \EdgeLoading_\horizon}{an edge loading operator mapping vectors of edge inflow rates to vectors of edge outflow rates, and its pendant for finite-time horizon $T$}
	\losentry{\ROp, \ROp_\horizon}{a routing operator mapping flows to sets of allowed flow splits, and its pendant for finite-time horizon $\horizon$}
	\losentry{\ROe}{an element of $\ROp(f)$; if $\ROp$ is prescriptive, we write $\smash{\ROp(f) \eqqcolon \{ \ROe(f) \}}$.}

	\losentry{\Omega_{\horizon,\alpha}^+(f^+)}{set of possible extensions of the inflow rates $f^+$ of a flow with time horizon~$\horizon$ on $[\horizon,\horizon+\alpha]$}
	\losentry{\Omega_{\horizon,\alpha}(f)}{set of possible extensions of a flow $f$ with time horizon~$\horizon$ on $[\horizon,\horizon+\alpha]$}
	
	\losentry{\pCost_{i,\Path}\revMinor{(\theta, f)}}{predicted cost of entering path~$p$ when entering at time~$\theta$ under flow~$f$; in \Cref{sec:SPE}, this is a random variable.}
	\losentry{\pEdges_i(\theta,f)}{active edges of commodity~$i$ at time~$\theta$ under flow~$f$}
	\losentry{\pActive_{e,i}(f)}{times at which edge~$e$ is active for commodity~$i$ under flow~$f$}
	\losentry{\ttime_e(\theta,f)}{travel time induced by flow~$f$ on edge~$e$ when entering at time~$\theta$}
	\losentry{\exitTime_e(\theta,f)}{exit time from edge~$e$ when entering at time~$\theta$ under flow~$f$}
	\losentry{\Prob_i}{probability measure over prediction functions for commodity~$i$}
	\losentry{\tilde E_i(\theta, f, \hat C)}{set of percieved active edges at time~$\theta$ \wrt flow $f$ and predictor $\hat C$}
\end{lostable} 
\fi 
\end{document}